\documentclass[a4paper, intlimits, reqno]{amsart}

\usepackage[english]{babel}
\usepackage[T1]{fontenc}
\usepackage[utf8]{inputenc}

\usepackage{amsmath}
\usepackage{amssymb}
\usepackage{MnSymbol}
\usepackage{amsthm}
\allowdisplaybreaks
\usepackage{amsfonts}
\usepackage{mathrsfs} 
\usepackage{mathtools}
\usepackage{nicefrac}
\usepackage{enumitem}

\usepackage{multicol}
\usepackage{url}
\usepackage{dsfont}
\usepackage[numbers,sort&compress]{natbib}
\usepackage{doi}
\usepackage{prettyref}
\usepackage{xcolor}
\usepackage{accents}
\usepackage{orcidlink}

\newrefformat{defn}{Definition \ref{#1}}
\newrefformat{rem}{Remark \ref{#1}}
\newrefformat{sect}{Section \ref{#1}}
\newrefformat{sub}{Section \ref{#1}}
\newrefformat{prop}{Proposition \ref{#1}}
\newrefformat{thm}{Theorem \ref{#1}}
\newrefformat{cor}{Corollary \ref{#1}}
\newrefformat{ex}{Example \ref{#1}}
\newrefformat{conv}{Convention \ref{#1}}

\swapnumbers
\newtheoremstyle{dotless}{}{}{\itshape}{}{\bfseries}{}{}{}
\theoremstyle{dotless}
\theoremstyle{plain}
\newtheorem{thm}{Theorem}[section]

\newtheorem{prop}[thm]{Proposition}
\newtheorem{cor}[thm]{Corollary}
\theoremstyle{definition}
\newtheorem{defn}[thm]{Definition}
\newtheorem{rem}[thm]{Remark}
\newtheorem{exa}[thm]{Example}
\newtheorem{conv}[thm]{Convention}
\newcommand{\N} {\mathbb{N}}

\newcommand{\R} {\mathbb{R}}
\newcommand{\C} {\mathbb{C}}
\newcommand{\K} {\mathbb{K}}
\newcommand{\D} {\mathbb{D}}
\newcommand{\F} {\mathcal{F}(\Omega)}

\DeclareMathOperator{\id}{id}
\DeclareMathOperator{\re}{Re}
\providecommand{\differential}{\mathrm{d}}
\renewcommand{\d}{\differential}
\newcommand{\euler}{\mathrm{e}}
\renewcommand{\i}{\mathrm{i}}
\newcommand\rlim{
\mathchoice{\vcenter{\hbox{${\scriptstyle{+}}$}}}
{\vcenter{\hbox{$\scriptstyle{+}$}}}
{\vcenter{\hbox{$\scriptscriptstyle{+}$}}}
{\vcenter{\hbox{$\scriptscriptstyle{+}$}}}}
\newcommand\llim{
\mathchoice{\vcenter{\hbox{${\scriptstyle{-}}$}}}
{\vcenter{\hbox{$\scriptstyle{-}$}}}
{\vcenter{\hbox{$\scriptscriptstyle{-}$}}}
{\vcenter{\hbox{$\scriptscriptstyle{-}$}}}}
\newcommand{\vertiii}[1]{{\left\vert\kern-0.25ex\left\vert\kern-0.25ex\left\vert #1 
    \right\vert\kern-0.25ex\right\vert\kern-0.25ex\right\vert}}
\newcommand*{\dt}[1]{%
  \accentset{\mbox{\large\bfseries .}}{#1}}

\makeatletter
\newcommand{\fakephantomsection}{%
  \Hy@GlobalStepCount\Hy@linkcounter%
  \Hy@MakeCurrentHref{\@currenvir.\the\Hy@linkcounter}
  \Hy@raisedlink{\hyper@anchorstart{\@currentHref}\hyper@anchorend}%
}
\makeatother

\begin{document}

\title[Weighted composition semigroups]{Weighted composition semigroups on spaces of continuous functions and their subspaces}
\author[K.~Kruse]{Karsten Kruse\,\orcidlink{0000-0003-1864-4915}}
\address[Karsten Kruse]{University of Twente, Department of Applied Mathematics, P.O. Box 217, 7500 AE Enschede, The Netherlands, and Hamburg University of Technology, Institute of Mathematics, Am Schwarzenberg-Campus~3, 21073 Hamburg, Germany}

\email{k.kruse@utwente.nl}
\thanks{K.~Kruse acknowledges the support by the Deutsche Forschungsgemeinschaft (DFG) within the Research Training
 Group GRK 2583 ``Modeling, Simulation and Optimization of Fluid Dynamic Applications''.}



\date{\today}
\begin{abstract}
This paper is dedicated to weighted composition semigroups on spaces of continuous functions 
and their subspaces. We consider 
semigroups induced by semiflows and semicocycles on Banach spaces $\F$ of continuous functions on a Hausdorff space $\Omega$ 
such that the norm-topology is stronger than the compact-open topology like the Hardy spaces, the weighted Bergman spaces, 
the Dirichlet space, the Bloch type spaces, the space of bounded Dirichlet series and weighted spaces of continuous or 
holomorphic functions. 
It was shown by Gallardo-Guti\'errez, Siskakis and Yakubovich that there are no non-trivial norm-strongly continuous
weighted composition semigroups on Banach spaces $\mathcal{F}(\D)$ of holomorphic functions on the open unit disc $\D$ such that 
$H^{\infty}\subset\mathcal{F}(\D)\subset\mathcal{B}_{1}$ where $H^{\infty}$ is the Hardy space of bounded holomorphic functions 
on $\D$ and $\mathcal{B}_{1}$ the Bloch space. However, we show that there are non-trivial weighted composition semigroups 
on such spaces which are strongly continuous w.r.t.~the mixed topology between the norm-topology and the compact-open topology. 
We study such weighted composition semigroups in the general setting of Banach spaces of continuous functions 
and derive necessary and sufficient conditions on the spaces involved, the semiflows and semicocycles 
for strong continuity w.r.t.~the mixed topology and as a byproduct for norm-strong continuity as well. 
Moreover, we give several characterisations of their generator and their space of norm-strong continuity. 
\end{abstract}
\maketitle

\section{Introduction}

Let $(\F,\|\cdot\|)$ be a Banach space of scalar-valued continuous functions on a Hausdorff space $\Omega$ such that 
the $\|\cdot\|$-topology is stronger than the compact-open topology $\tau_{\operatorname{co}}$ on $\F$. 
Suppose that $\varphi\coloneqq (\varphi_{t})_{t\geq 0}$ is a semiflow and $m\coloneqq (m_{t})_{t\geq 0}$ 
an associated semicocycle on $\Omega$ such that the induced weighted composition semigroup $(C_{m,\varphi}(t))_{t\geq 0}$ 
given by $C_{m,\varphi}(t)f\coloneqq m_{t}\cdot (f\circ \varphi_{t})$ for all $t\geq 0$ and $f\in\F$ is a well-defined 
semigroup of linear maps from $\F$ to $\F$. 

In the case that $\Omega=\D\subset\C$ is the open unit disc, $\varphi$ a jointly continuous holomorphic semiflow and 
$\mathcal{F}(\D)$ a space of holomorphic functions on $\D$ such semigroups are well-studied 
in the unweighted case, i.e.~$m_{t}(z)=1$ for all $t\geq 0$ and $z\in\D$, on the Hardy spaces 
$H^{p}$ for $1\leq p<\infty$ in \cite{berkson1978}, 
on the weighted Bergman spaces $A_{\alpha}^{p}$ for 
$\alpha>-1$ and $1\leq p<\infty$ in \cite{siskakis1987}, on the Dirichlet space $\mathcal{D}$ in \cite{siskakis1996}  
and on more general spaces $\mathcal{F}(\D)$ in \cite{arendt2019,blasco2013,daskalogiannis2021,gutierrez2019,siskakis1998}. 
In particular, they are always $\|\cdot\|$-strongly continuous on $H^{p}$, $A_{\alpha}^{p}$ and $\mathcal{D}$.
 
The weighted case, where $m$ is a jointly continuous holomorphic semicocycle, 
is more complicated and got more attention recently 
\cite{arevalo2017,bernard2022,chalendar2023,chalendar2022,gutierrez2023,perlich2019,wu2021}. 
However, to the best of our knowledge from the spaces mentioned above 
only for the Hardy spaces $H^{p}$, $1\leq p<\infty$, and the weighted Bergman spaces $A_{\alpha}^{p}$, 
$1\leq p<\infty$ and $\alpha>-1$, sufficient (non-trivial) conditions on general
$m$ are known such that the weighted composition semigroup becomes $\|\cdot\|$-strongly continuous 
\cite{koenig1990,siskakis1986,siskakis1998,perlich2019,wu2021}. In \cite{arevalo2017} non-trivial sufficient conditions for 
$\|\cdot\|$-strong continuity are given for Banach spaces $\mathcal{F}(\D)$ of holomorphic functions 
in which the polynomials are dense in the case that $m_{t}=\varphi_{t}'$ for all $t\geq 0$.

Considering the Hardy space for $p=\infty$, it is shown in \cite{gutierrez2023} that the only 
$\|\cdot\|$-strongly continuous weighted composition semigroups on $\mathcal{F}(\mathbb{D})$ such that 
$H^{\infty}\subset\mathcal{F}(\mathbb{D})\subset\mathcal{B}_{1}$ are the trivial ones, 
i.e.~$\varphi_{t}=\id$ for all $t\geq 0$. Here, $\mathcal{B}_{1}$ stands for the Bloch space. 
In the unweighted case this has already been observed in \cite{anderson2017,gutierrez2019}. 
Similarly, it is shown in \cite[Theorem 7.1, p.~34]{contreras2023} that there no non-trivial 
$\|\cdot\|$-strongly continuous composition semigroups on the space $\mathscr{H}^{\infty}$ of bounded Dirichlet series 
on the open right half-plane. 
Nevertheless, there are non-trivial weighted composition semigroups on such spaces 
$\mathcal{F}(\mathbb{D})$ resp.~$\mathscr{H}^{\infty}$ as well and it is said in \cite[p.~494]{gutierrez2019} 
that it would be desirable to substitute the $\|\cdot\|$-strong continuity by a 
weaker property so that \cite[Main theorem, p.~490]{gutierrez2019} 
(\cite[Theorems 2.1, 3.1, p.~68--69]{gutierrez2023} in the weighted case), 
which describes the generator of a $\|\cdot\|$-strongly continuous (weighted) composition semigroup, remains valid. 
This is one of the problems we solve in the present paper. 
We substitute the $\|\cdot\|$-strong continuity by $\gamma$-strong continuity where 
$\gamma\coloneqq\gamma(\|\cdot\|,\tau_{\operatorname{co}})$ is the mixed topology 
between the $\|\cdot\|$-topology and $\tau_{\operatorname{co}}$, which was introduced in \cite{wiweger1961} and is 
a Hausdorff locally convex topology. 

Let us outline the content of our paper. In \prettyref{sect:notions} we recall the notions of a Saks space 
$(X,\|\cdot\|,\tau)$, where $(X,\|\cdot\|)$ is a normed space and $\tau$ a coarser norming Hausdorff locally convex topology on $X$, 
the mixed topology $\gamma\coloneqq\gamma(\|\cdot\|,\tau)$, some background on semigroups on Hausdorff locally convex spaces 
and in \prettyref{thm:mixed_equi_sg_bi_cont} how $\gamma$-strongly continuous, locally $\gamma$-equicontinuous semigroups 
are related to the concept of a $\tau$-bi-continuous semigroup, which was introduced in \cite{kuehnemund2001,kuehnemund2003}. 
We then give several examples of Saks spaces of the form $(\F,\|\cdot\|,\tau_{\operatorname{co}})$, which include 
among others the Hardy spaces $H^{p}$ for $1\leq p< \infty$, the weighted Bergman spaces $A_{\alpha}^{p}$ 
for $1\leq p<\infty$ and $\alpha>-1$, and the Dirichlet space $\mathcal{D}$ in \prettyref{ex:hardy_bergman_dirichlet}, 
the $v$-Bloch spaces w.r.t.~a continuous weight $v$ in \prettyref{ex:bloch_saks}, especially the Bloch type spaces 
$\mathcal{B}_{\alpha}$ for $\alpha>0$, as well as weighted spaces of continuous resp.~holomorphic functions, especially, 
the Hardy space $H^{\infty}$ and the space $\mathscr{H}^{\infty}$ of bounded Dirichlet series in 
\prettyref{ex:weighted_holom_saks}, \prettyref{ex:bDirichlet_series} and \prettyref{ex:cont_saks}. 

In \prettyref{sect:semiflows_cocycles} we recap the notions of a semiflow $\varphi$, a semicocycle $m$ for $\varphi$ and 
introduce the notion of a co-semiflow $(m,\varphi)$. We give equivalent characterisations of their joint continuity 
depending on the topological properties of $\Omega$, present several examples and generalise the concept of a generator 
of a semiflow, which was introduced in \cite{berkson1978} for jointly continuous holomorphic semiflows. 

In \prettyref{sect:sg_weighted_comp} we use the concepts and results of the preceding sections to prove one of 
our main results \prettyref{prop:strongly_mixed_cont}, which generalises \cite[Proposition 2.10, p.~5]{farkas2020} 
and \cite[Corollary 4.3, p.~20]{kruse_schwenninger2022}, that the weighted composition semigroup $(C_{m,\varphi}(t))_{t\geq 0}$ 
on a Saks space $(\F,\|\cdot\|,\tau_{\operatorname{co}})$ is $\gamma$-strongly continuous and locally $\gamma$-equicontinuous 
if the semigroup is locally bounded (w.r.t.~the operator norm) and the co-semiflow $(m,\varphi)$ is jointly continuous. 
Then we derive sufficient conditions depending on $(m,\varphi)$ for the local boundedness of the semigroup 
$(C_{m,\varphi}(t))_{t\geq 0}$ for the Saks spaces mentioned above, in particular in the case $\F=\mathscr{H}^{\infty}$ 
in \prettyref{thm:loc_bound_sg_bDirichlet_series}. 

In \prettyref{sect:generators} we turn to the generator $(A,D(A))$ of locally bounded, $\gamma$-strongly continuous 
weighted composition semigroups and show in \prettyref{prop:lie_generator} that it coincides with the Lie generator, 
i.e.~the pointwise generator, if the Saks space $(\F,\|\cdot\|,\tau_{\operatorname{co}})$ is sequentially 
complete (w.r.t.~the mixed topology $\gamma$) and the co-semiflow $(m,\varphi)$ is jointly continuous. 
This generalises \cite[Proposition 2.12, p.~6]{farkas2020} and \cite[Proposition 2.4, p.~118]{dorroh1996} 
where $\F=\mathcal{C}_{b}(\Omega)$ is the space of bounded continuous functions 
on a completely regular Hausdorff $k$-space resp.~Polish space $\Omega$ and $m$ trivial. 
The connection to $\tau$-bi-continuous semigroups from \prettyref{sect:notions} also allows us to deduce 
in \prettyref{prop:norm_generator} that on sequentially complete Saks spaces $(\F,\|\cdot\|,\tau_{\operatorname{co}})$ 
such semigroups are $\|\cdot\|$-strongly continuous if $(\F,\|\cdot\|)$ is reflexive, and to show that the space of 
$\|\cdot\|$-strong continuity coincides with the $\|\cdot\|$-closure of $D(A)$. 
In \prettyref{thm:mixed_gen_max_dom_diff_everywhere} 
we turn to the special case that $\Omega\subset\C$ is open, the jointly continuous co-semiflow $(m,\varphi)$ 
continuously differentiable w.r.t.~$t$ and $\F$ for instance a space of holomorphic functions, which results in 
the representation
\begin{equation}\label{eq:gen_intro}
D(A)=\{f\in\F\;|\;Gf'+gf\in\F\},\quad Af=Gf'+gf,\,f\in D(A),
\end{equation}
of the generator, where $G\coloneqq\dt{\varphi}_{0}$ and $g\coloneqq\dt{m}_{0}$ 
denote the derivatives w.r.t.~$t$ of $\varphi$ and $m$ in $t=0$, respectively. 
In \prettyref{thm:mixed_gen_max_dom_real} we also handle the more complicated case where $\Omega\subset\R$ is open and 
the space $\F$ need not be a space of continuously differentiable or holomorphic functions, for example 
$\F=\mathcal{C}_{b}(\Omega)$.

\prettyref{sect:converse_generators} is dedicated to the converse of \eqref{eq:gen_intro} in the sense 
that a $\gamma$-strongly continuous semigroup with such a generator for some holomorphic functions $G$ and $g$ on $\Omega$ 
must be a weighted composition semigroup w.r.t.~some jointly continuous holomorphic co-semiflow $(m,\varphi)$ 
at least if $\Omega=\D$, $(\mathcal{F}(\D),\|\cdot\|,\tau_{\operatorname{co}})$ is a sequentially complete Saks space 
of holomorphic functions and the embedding $\mathcal{H}(\overline{\D})\hookrightarrow (\mathcal{F}(\D),\|\cdot\|)$ 
continuous where $\mathcal{H}(\overline{\D})$ denotes the space of holomorphic germs near $\overline{\D}$ 
with its inductive limit topology (see \prettyref{thm:converse_gen}). 
In \prettyref{ex:cont_embed_induc_lim} we show that this embedding condition is fulfilled 
for the spaces like $H^{p}$ for $1\leq p\leq \infty$, $A_{\alpha}^{p}$ for $\alpha>-1$ and $1\leq p<\infty$, $\mathcal{D}$ and 
$\mathcal{B}_{\alpha}$ for $\alpha>0$. \prettyref{thm:converse_gen} in combination 
with \prettyref{thm:mixed_gen_max_dom_diff_everywhere} (see also \prettyref{thm:holomorphic_weighted_comp}) 
is the counterpart of \cite[Main theorem, p.~490]{gutierrez2019} and \cite[Theorems 2.1, 3.1, p.~68--69]{gutierrez2023} 
we were searching for. 

In the closing \prettyref{sect:applications} we generalise in \prettyref{prop:mult_sg_Cv}, 
\prettyref{prop:weighted_translation_Cb} and \prettyref{prop:weighted_flow_Cb} results 
from \cite{budde2019a,dorroh1993,dorroh1996,kruse_schwenninger2022} on multiplication semigroups 
on $\mathcal{C}_{b}(\Omega)$ for locally compact Hausdorff $\Omega$ and unweighted composition semigroups on 
$\mathcal{C}_{b}(\R)$ to the more general setting of weighted composition semigroups on weighted spaces of continuous functions. 
Further, \prettyref{thm:holomorphic_weighted_comp_norm_strong} gives us necessary and sufficient conditions for 
a weighted composition semigroup $(C_{m,\varphi}(t))_{t\geq 0}$ on a sequentially complete Saks space 
$(\F,\|\cdot\|,\tau_{\operatorname{co}})$ of holomorphic functions on an open connected set 
$\Omega\subset\C$ to be $\|\cdot\|$-strongly continuous resp.~$\gamma$-strongly continuous. 
In combination with the results from \prettyref{sect:sg_weighted_comp}, 
see \prettyref{thm:loc_bound_sg_hardy_bergman_dirichlet}, \prettyref{thm:loc_bound_sg_bloch} and \prettyref{thm:loc_bound_sg_Hv},
we obtain sufficient conditions on $(m,\varphi)$ so that $(C_{m,\varphi}(t))_{t\geq 0}$ is $\|\cdot\|$-strongly continuous 
on the Hardy spaces $H^{p}$ for $1<p<\infty$ and the weighted Bergman spaces $A_{\alpha}^{p}$ for $\alpha>-1$ 
and $1< p<\infty$, where we get back the ones from \cite{perlich2019,wu2021} by a different proof which improve the already known ones 
from \cite{koenig1990,siskakis1986,siskakis1998}, and on the Dirichlet space $\mathcal{D}$ 
as well as $\gamma$-strongly continuous on the Hardy space $H^{\infty}$ and the Bloch type spaces 
$\mathcal{B}_{\alpha}$ for $\alpha>0$.

\section{Background on semigroups on Saks spaces}
\label{sect:notions}

In this section we recall some basic notions and results in the context of semigroups on Hausdorff locally convex spaces, 
bi-continuous semigroups, the mixed topology and Saks spaces to keep this work practically self-contained. We refer 
the interested reader for more detailed information to \cite{cooper1978,komatsu1964,komura1968,kuehnemund2001,yosida1968}. 
For a Hausdorff locally convex space $(X,\tau)$ over the field $\K=\R$ or $\C$ we use the symbol $\mathcal{L}(X,\tau)$ 
for the space of continuous linear operators from $(X,\tau)$ to $(X,\tau)$.
If $(X,\|\cdot\|)$ is a normed space, we just write $\mathcal{L}(X)\coloneqq \mathcal{L}(X,\tau_{\|\cdot\|})$
where $\tau_{\|\cdot\|}$ is the $\|\cdot\|$-topology. First, we recall the notions of strong continuity and equicontinuity.

\begin{defn}\label{defn:strong_cont}
Let $(X,\tau)$ be a Hausdorff locally convex space, $I$ a set and $(T(t))_{t\in I}$ 
a family of linear maps $X\to X$.
\begin{enumerate}
\item Let $I$ be a Hausdorff space. $(T(t))_{t\in I}$ is called $\tau$\emph{-strongly continuous} 
if $T(t)\in\mathcal{L}(X,\tau)$ for every $t\in I$ and 
the map $T_{x}\colon I\to (X,\tau)$, $T_{x}(t)\coloneqq T(t)x$, is continuous for every $x\in X$. 
\item Let $\sigma$ be an additional Hausdorff locally convex topology on $X$. 
$(T(t))_{t\in I}$ is called \emph{$\sigma$-$\tau$-equicontinuous} if 
\[
\forall\;p\in\Gamma_{\tau}\;\exists\;\widetilde{p}\in\Gamma_{\sigma},\;C\geq 0\;
\forall\;t\in I,\,x\in X:\;p(T(t)x)\leq C\widetilde{p}(x)
\]
where $\Gamma_{\tau}$ and $\Gamma_{\sigma}$ are directed systems of continuous seminorms that generate $\tau$ and 
$\sigma$, respectively. If $\tau=\sigma$, we just write \emph{$\tau$-equicontinuous} instead of 
\emph{$\tau$-$\tau$-equicontinuous}.
\end{enumerate}
\end{defn}

In the context of semigroups of linear maps there are different degrees of equicontinuity 
and boundedness. In the following definition we use the symbol $\id$ for the identity map on a set $X$, i.e.~the map 
$\id\colon X\to X$, $\id(x)\coloneqq x$.

\begin{defn}
Let $X$ be a linear space and $(T(t))_{t\geq 0}$ a family of linear maps $X\to X$. 
\begin{enumerate}
\item $(T(t))_{t\geq 0}$ is called a \emph{semigroup} if $T(0)=\id$ and $T(t+s)=T(t)T(s)$ for all $t,s\geq 0$. 
\item Let $(X,\tau)$ be a Hausdorff locally convex space. $(T(t))_{t\geq 0}$ is called \emph{locally $\tau$-equicontinuous} 
if $(T(t))_{t\in[0,t_{0}]}$ is $\tau$-equicontinuous for all $t_{0}\geq 0$. 
$(T(t))_{t\geq 0}$ is called \emph{quasi-$\tau$-equicontinuous} if there exists $\omega\in\R$ such that 
$(\euler^{-\omega t}T(t))_{t\geq 0}$ is $\tau$-equicontinuous.
\item Let $(X,\|\cdot\|)$ be a normed space. $(T(t))_{t\geq 0}$ is called \emph{locally bounded} 
if for all $t_{0}\geq 0$ it holds that
\[
\sup_{t\in[0,t_{0}]}\|T(t)\|_{\mathcal{L}(X)}<\infty
\]
where $\|T(t)\|_{\mathcal{L}(X)}\coloneqq \sup_{x\in X,\,\|x\|\leq 1}\|T(t)x\|$. 
$(T(t))_{t\geq 0}$ is called \emph{exponentially bounded} if there exist $M\geq 1$ and 
$\omega\in\R$ such that $\|T(t)\|_{\mathcal{L}(X)}\leq M\euler^{\omega t}$ for all $t\geq 0$.
\end{enumerate}
\end{defn}

Since local boundedness of semigroups of linear maps on normed spaces will be an essential condition 
in our work, we give the following characterisation, which carries over from the case of 
norm-strongly continuous semigroups on Banach spaces. 

\begin{prop}\label{prop:loc_bounded_sg}
Let $(X,\|\cdot\|)$ be a normed space and $(T(t))_{t\geq 0}$ a semigroup of linear maps $X\to X$. 
Then the following assertions are equivalent.
\begin{enumerate}
\item $(T(t))_{t\geq 0}$ is exponentially bounded.
\item $(T(t))_{t\geq 0}$ is locally bounded.
\item There exists $t_{0}>0$ such that $\sup_{t\in[0,t_{0}]}\|T(t)\|_{\mathcal{L}(X)}<\infty$.
\end{enumerate}
If $(X,\|\cdot\|)$ is additionally complete, then each of the preceding assertions is equivalent to:
\begin{enumerate}
\item[(d)] There exists $t_{0}>0$ such that $\sup_{t\in[0,t_{0}]}\|T(t)x\|<\infty$ for all $x\in X$ and 
$T(t)\in\mathcal{L}(X)$ for all $t\in[0,t_{0}]$.
\end{enumerate}
\end{prop}
\begin{proof}
The implications (a)$\Rightarrow$(b)$\Rightarrow$(c)$\Rightarrow$(d) are obvious. 
The proof of \cite[Chap.~I, 5.5 Proposition, p.~39]{engel_nagel2000} yields the implication (c)$\Rightarrow$(a) 
if $\sup_{t\in[0,1]}\|T(t)\|_{\mathcal{L}(X)}<\infty$ 
(we note that proof of \cite[Chap.~I, 5.5 Proposition, p.~39]{engel_nagel2000} still works in our situation if we replace 
its condition that $(X,\|\cdot\|)$ is a Banach space and the semigroup $\|\cdot\|$-strongly continuous by the condition 
$\sup_{t\in[0,1]}\|T(t)\|_{\mathcal{L}(X)}<\infty$). So, if $t_{0}\geq 1$, we are done. 
If $t_{0}<1$, we set $M\coloneqq \sup_{t\in[0,t_{0}]}\|T(t)\|_{\mathcal{L}(X)}$. 
Then we get for all $t_{0}\leq t\leq 2t_{0}$ that 
\[
\|T(t)\|_{\mathcal{L}(X)}=\|T(t-t_{0})T(t_{0})\|_{\mathcal{L}(X)}\leq \|T(t-t_{0})\|_{\mathcal{L}(X)}\|T(t_{0})\|_{\mathcal{L}(X)}
\leq M^2
\]
and thus $\sup_{t\in[0,2t_{0}]}\|T(t)\|_{\mathcal{L}(X)}\leq M+M^2$. By repetition of this procedure we get in 
finitely many steps that $\sup_{t\in[0,1]}\|T(t)\|_{\mathcal{L}(X)}<\infty$.

The equivalence of the first three statements to (d) is a consequence of the uniform boundedness principle if $(X,\|\cdot\|)$ 
is additionally complete.
\end{proof}

Let us recall the definition of the mixed topology, \cite[Section 2.1]{wiweger1961}, and the notion of a Saks space, 
\cite[I.3.2 Definition, p.~27--28]{cooper1978}, which will be important for the rest of the paper.

\begin{defn}[{\cite[2.1 Definition, p.~3--4]{kruse2023c}}]\label{defn:mixed_top_Saks}
Let $(X,\|\cdot\|)$ be a normed space and $\tau$ a Hausdorff locally convex topology on $X$ that is coarser 
than the $\|\cdot\|$-topology $\tau_{\|\cdot\|}$. Then
\begin{enumerate}
\item the \emph{mixed topology} $\gamma \coloneqq \gamma(\|\cdot\|,\tau)$ is
the finest linear topology on $X$ that coincides with $\tau$ on $\|\cdot\|$-bounded sets and such that 
$\tau\subset \gamma \subset \tau_{\|\cdot\|}$; 
\item the triple $(X,\|\cdot\|,\tau)$ is called a \emph{Saks space} 
if there exists a directed system of continuous seminorms $\Gamma_{\tau}$ that generates the topology $\tau$ such that   
\begin{equation}\label{eq:saks}
\|x\|=\sup_{p\in\Gamma_{\tau}} p(x), \quad x\in X.
\end{equation}
\end{enumerate}
\end{defn}

The mixed topology is actually Hausdorff locally convex and the definition given above is equivalent to the one 
introduced by Wiweger \cite[Section 2.1]{wiweger1961} due to \cite[Lemmas 2.2.1, 2.2.2, p.~51]{wiweger1961}. 

We recall from \cite[p.~4]{kruse2023c} that it is often useful to have a characterisation of the mixed topology 
by generating systems of continuous seminorms, e.g.~the definition of dissipativity in Lumer--Phillips generation theorems 
for bi-continuous semigroups depends on the choice of the generating system of seminorms of the mixed topology 
(see \cite{kruse_seifert2022b}). For that purpose we recap the following auxiliary topology whose origin is 
\cite[Theorem 3.1.1, p.~62]{wiweger1961}. 

\begin{defn}[{\cite[Definition 3.9, p.~9]{kruse_schwenninger2022}}]\label{defn:submixed_top}
Let $(X,\|\cdot\|,\tau)$ be a Saks space and $\Gamma_{\tau}$ a directed system of continuous seminorms 
that generates the topology $\tau$ and fulfils \eqref{eq:saks}. We set 
\[
\mathcal{N}\coloneqq \{(p_{n},a_{n})_{n\in\N}\;|\;(p_{n})_{n\in\N}\subset\Gamma_{\tau},\,(a_{n})_{n\in\N}\in c_{0}^{+}\}
\]
where $c_0^{+}$ is the family of all real non-negative null-sequences.
For $(p_{n},a_{n})_{n\in\N}\in\mathcal{N}$ we define the seminorm
\[
 \vertiii{x}_{(p_{n},a_{n})_{n\in\N}}\coloneqq\sup_{n\in\N}p_{n}(x)a_{n},\quad x\in X.
\]
We denote by $\gamma_s\coloneqq\gamma_s(\|\cdot\|,\tau)$ the Hausdorff locally convex topology that is generated by 
the system of seminorms $(\vertiii{\cdot}_{(p_n,a_n)_{n\in\N}})_{(p_n,a_n)_{n\in\N}\in\mathcal{N}}$ and call it the 
\emph{submixed topology}.
\end{defn}

By \cite[I.1.10 Proposition, p.~9]{cooper1978}, \cite[Theorem 3.1.1, p.~62]{wiweger1961} 
(cf.~\cite[I.4.5 Proposition, p.~41--42]{cooper1978}) and 
\cite[Lemma A.1.2, p.~72]{farkas2003} we have the following relation between the mixed and the submixed topology.

\begin{rem}[{\cite[Remark 3.10, p.~9]{kruse_schwenninger2022}}]\label{rem:mixed=submixed}
Let $(X,\|\cdot\|,\tau)$ be a Saks space, $\Gamma_{\tau}$ a directed system of continuous seminorms 
that generates the topology $\tau$ and fulfils \eqref{eq:saks}, $\gamma\coloneqq\gamma(\|\cdot\|,\tau)$ the mixed 
and $\gamma_{s}\coloneqq\gamma_{s}(\|\cdot\|,\tau)$ the submixed topology.
\begin{enumerate}
\item[(a)] We have $\tau\subset \gamma_s\subset \gamma$ and $\gamma_{s}$ has the same convergent 
sequences as $\gamma$.
\item[(b)] If 
 \begin{enumerate}
 \item[(i)] for every $x\in X$, $\varepsilon>0$ and $p\in\Gamma_{\tau}$ there are $y,z\in X$ such that $x=y+z$, 
 $p(z)=0$ and $\|y\|\leq p(x)+\varepsilon$, or 
 \item[(ii)] the $\|\cdot\|$-closed unit ball $B_{\|\cdot\|}\coloneqq\{x\in X\;|\; \|x\|\leq 1\}$ is $\tau$-compact,
 \end{enumerate}
then $\gamma=\gamma_s$ holds. 
\end{enumerate}
\end{rem}

Further, we will use the following notions.  

\begin{defn}[{\cite[Definitions 2.2, 5.7, p.~423, 433]{kruse_seifert2022a}, \cite[2.4 Definition, p.~4--5]{kruse2023c}}]
Let $(X,\|\cdot\|,\tau)$ be a Saks space. 
\begin{enumerate}
\item We call $(X,\|\cdot\|,\tau)$ \emph{(sequentially) complete} if $(X,\gamma)$ is (sequentially) complete.
\item We call $(X,\|\cdot\|,\tau)$ \emph{semi-Montel} if $(X,\gamma)$ is a semi-Montel space.
\item We call $(X,\|\cdot\|,\tau)$ \emph{C-sequential} if $(X,\gamma)$ is C-sequential, i.e.~every convex 
sequentially open subset of $(X,\gamma)$ is already open (see \cite[p.~273]{snipes1973}).
\end{enumerate}
\end{defn}

A Saks space is complete if and only if the $\|\cdot\|$-closed unit ball 
$B_{\|\cdot\|}$ is $\tau$-complete by \cite[I.1.14 Proposition, p.~11]{cooper1978}.
We note that condition (ii) in \prettyref{rem:mixed=submixed} (b) is equivalent to 
$(X,\|\cdot\|,\tau)$ being a semi-Montel space 
by \cite[I.1.13 Proposition, p.~11]{cooper1978}. If $X$ is a space of $\K$-valued functions on some set $\Omega$, then 
the semi-Montel property may be used to derive linearisations of weak vector-valued versions of $X$, i.e.~of 
$X(E)\coloneqq\{f\colon\Omega\to E\;|\;\forall\;e'\in E':\; e'\circ f\in X\}$ 
where $E$ is a Hausdorff locally convex space and $E'$ its topological linear dual space (see \cite[3.3 Theorem, p.~7]{kruse2023c}). 
Since the space $(X,\gamma)$ is usually not barrelled by 
\cite[I.1.15 Proposition, p.~12]{cooper1978}, one cannot apply automatic local equicontinuity results like \cite[Proposition 1.1, p.~259]{komura1968} to $\gamma$-strongly continuous semigroups. 
A way to circumvent this problem is the condition that the space is C-sequential. 
A sufficient condition that guarantees 
that $(X,\gamma)$, and thus $(X,\|\cdot\|,\tau)$, is C-sequential is that $\tau$ is metrisable on $B_{\|\cdot\|}$ 
by \cite[Proposition 5.7, p.~2681--2682]{kruse_meichnser_seifert2018}. 

\begin{rem}\label{rem:seq_complete_mixed_Banach}
If $(X,\|\cdot\|,\tau)$ is a sequentially complete Saks space, then the normed space $(X,\|\cdot\|)$ is complete because 
$\gamma$ is a coarser topology than the norm $\|\cdot\|$-topology and completeness of a normed space 
is equivalent to sequential completeness.
\end{rem}

On sequentially complete Saks spaces $(X,\|\cdot\|,\tau)$ there is another notion of strong\-ly continuous semigroups, namely, 
so-called \emph{$\tau$-bi-continuous semigroups} which were introduced in \cite[Definition 1.3, p.~6--7]{kuehnemund2001} 
(cf.~\cite[Definition 3, p.~207]{kuehnemund2003}). 
Due to \cite[2.9 Remark (b)]{kruse_seifert2022b} (cf.~\cite[Proposition A.1.3, p.~73]{farkas2003}) and 
the comments after \cite[Definition 2.2, p.~423]{kruse_seifert2022a}, a semigroup $(T(t))_{t\geq 0}$ of linear maps $X\to X$ is $\tau$-bi-continuous if and only if it is locally sequentially $\gamma$-equicontinuous and the map 
$T_{x}\colon [0,\infty)\to (X,\gamma)$, $T_{x}(t)=T(t)x$, is continuous for every $x\in X$ (we note that in contrast to \prettyref{defn:strong_cont} (a) the definition of $\gamma$-strong continuity in \cite{kruse_seifert2022b} does not include the condition that $T(t)\in\mathcal{L}(X,\gamma)$ for every $t\geq 0$). 
Hence every $\gamma$-strongly continuous and locally $\gamma$-equicontinuous semigroup on a sequentially complete Saks space is 
$\tau$-bi-continuous and the converse is not true in general by \cite[Example 4.1, p.~320]{farkas2011}. 
However, on sequentially complete C-sequential Saks spaces 
the converse is also true by \cite[Theorem 7.4, p.~180]{kraaij2016} and 
\cite[Theorem 7.4, p.~52]{wilansky1981}, 
even more, every $\tau$-bi-continuous semigroup is $\gamma$-strongly continuous and 
quasi-$\gamma$-equicontinuous by \cite[Theorem 3.17, p.~13]{kruse_schwenninger2022}. 
Moreover, there is another notion related 
to quasi-$\gamma$-equicontinuity on Saks spaces, namely, \emph{$(\|\cdot\|,\tau)$-equitightness}, 
which was introduced in \cite[Definitions 3.4, 3.5, p.~6, 7]{kruse_schwenninger2022} and is important in perturbation theory for 
$\tau$-bi-continuous semigroups (see \cite{budde2019a,es_sarhir2006,farkas2003,kruse_schwenninger2022} and the references therein). 
The notion goes back to \cite[Definitions 1.2.20, 1.2.21, p.~12]{farkas2003} and 
\cite[Definition 1.1, p.~668]{es_sarhir2006}. 

\begin{defn}
Let $(X,\|\cdot\|,\tau)$ be a Saks space. A family $(T(t))_{t\in I}$ of linear maps $X\to X$ is called \emph{$(\|\cdot\|,\tau)$-equitight} if 
\[
\forall\;\varepsilon>0,\,p\in\Gamma_{\tau}\;\exists\;\widetilde{p}\in\Gamma_{\tau},\,C\geq 0\;
\forall\;t\in I,\,x\in X:\;p(T(t)x)\leq C \widetilde{p}(x)+\varepsilon\|x\|
\]
where $\Gamma_{\tau}$ is a family of continuous seminorms that generates $\tau$. 
A semigroup $(T(t))_{t\geq 0}$ of linear maps $X\to X$ is called 
\emph{locally $(\|\cdot\|,\tau)$-equitight} if $(T(t))_{t\in [0,t_{0}]}$ is 
$(\|\cdot\|,\tau)$-equitight for every $t_{0}\geq 0$. $(T(t))_{t\geq 0}$ is called 
\emph{quasi-$(\|\cdot\|,\tau)$-equitight} if there is $\alpha\in\R$ such that 
$(\euler^{-\alpha t}T(t))_{t\geq 0}$ is $(\|\cdot\|,\tau)$-equitight.
\end{defn}
 
In general, local $(\|\cdot\|,\tau)$-equitightness is really weaker than local 
$\tau$-equicontinuity by \cite[Examples 6 (a), p.~209--210]{kuehnemund2003} and \cite[Example 4.2 (b), p.~19]{kruse_schwenninger2022} and the same is true for their quasi-counterparts by 
\cite[Example 3.2, p.~549]{kunze2009} and \cite[Example 4.2 (a), p.~19]{kruse_schwenninger2022}. 
Local $(\|\cdot\|,\tau)$-equitightness (quasi-$(\|\cdot\|,\tau)$-equitightness) of a semigroup of linear maps is 
stronger than local $\gamma$-equicontinuity (quasi-$\gamma$-equicontinuity), 
but they are equivalent if $\gamma=\gamma_{s}$ and the semigroup is $\gamma$-strongly continuous 
by \cite[Proposition 3.16, p.~12--13]{kruse_schwenninger2022} and 
\cite[Remark 2.6 (b), p.~5--6]{kruse_schwenninger2022}. Summarising, we have the following theorem. 

\begin{thm}\label{thm:mixed_equi_sg_bi_cont}
Let $(X,\|\cdot\|,\tau)$ be a sequentially complete C-sequential Saks space and $(T(t))_{t\geq 0}$ a semigroup of linear maps 
$X\to X$. Then the following assertions are equivalent.
\begin{enumerate}
\item $(T(t))_{t\geq 0}$ is $\tau$-bi-continuous.
\item $(T(t))_{t\geq 0}$ is $\gamma$-strongly continuous and locally $\gamma$-equicontinuous.
\item $(T(t))_{t\geq 0}$ is $\gamma$-strongly continuous and quasi-$\gamma$-equicontinuous.
\end{enumerate}
If in addition $\gamma=\gamma_{s}$, then each of the preceding assertions is equivalent to 
each of the following ones:
\begin{enumerate}
\item[(d)] $(T(t))_{t\geq 0}$ is $\gamma$-strongly continuous and locally $(\|\cdot\|,\tau)$-equitight.
\item[(e)] $(T(t))_{t\geq 0}$ is $\gamma$-strongly continuous and quasi-$(\|\cdot\|,\tau)$-equitight.
\end{enumerate}
\end{thm}

We close this section with some examples of Saks spaces and a convention we will use throughout 
the paper. We denote by $\mathcal{C}(\Omega)$ the space of $\K$-valued continuous functions on a 
Hausdorff space $\Omega$ and by $\tau_{\operatorname{co}}$ the \emph{compact-open topology} 
on $\mathcal{C}(\Omega)$, i.e.~the topology of uniform convergence on compact subsets of $\Omega$.

\begin{conv}\label{conv:pre_Saks_function_space}
Let $\Omega$ be a Hausdorff space and $(\F,\|\cdot\|)$ a normed space such that 
$\F\subset\mathcal{C}(\Omega)$. Then $\mathcal{C}(\Omega)$ induces the relative compact-open 
topology ${\tau_{\operatorname{co}}}_{\mid \F}$ on $\F$ and we get a Saks space 
$(\F,\|\cdot\|,{\tau_{\operatorname{co}}}_{\mid \F})$ if and only if 
\begin{enumerate}
\item[(i)] for every compact set $K\subset\Omega$ there is $C\geq 0$ such that 
\[
\sup_{x\in K}|f(x)|\leq C\|f\|,\quad f\in\F,
\]
which is equivalent to the inclusion 
$I\colon (\F,\|\cdot\|) \to (\mathcal{C}(\Omega),\tau_{\operatorname{co}})$, $I(f)\coloneqq f$, being continuous, and
\item[(ii)] there exists a directed system of continuous seminorms $\Gamma_{{\tau_{\operatorname{co}}}_{\mid \F}}$ that generates the topology ${\tau_{\operatorname{co}}}_{\mid \F}$ such that   
\[
\|f\|=\sup_{p\in\Gamma_{{\tau_{\operatorname{co}}}_{\mid \F}}} p(f), \quad f\in \F.
\]
\end{enumerate}
If this is fulfilled, we write that 
$(\F,\|\cdot\|,\tau_{\operatorname{co}})\coloneqq 
(\F,\|\cdot\|,{\tau_{\operatorname{co}}}_{\mid \F})$ is a Saks space.
\end{conv}

Further, we denote by $\mathcal{H}(\Omega)\coloneqq \mathcal{C}^{1}_{\C}(\Omega)$ 
the space of holomorphic functions on an open set 
$\Omega\subset\C$ and set $\D\coloneqq \{z\in\C\;|\;|z|< 1\}$. 
Moreover, we write $\mathcal{C}^{1}(\Omega)\coloneqq \mathcal{C}^{1}_{\R}(\Omega)$ 
for the space of continuously differentiable functions on an open set $\Omega\subset\R$. 
Due to \cite[3.5 Corollary, p.~9--10]{kruse2023c}, \cite[Theorem 9.8, p.~260]{zhu2007} and \cite[Theorem 4.25, p.~81]{zhu2007} 
we have the following result.

\begin{exa}\label{ex:hardy_bergman_dirichlet}
For the following spaces $(\mathcal{F}(\D),\|\cdot\|)$ the triples 
$(\mathcal{F}(\D),\|\cdot\|,\tau_{\operatorname{co}})$ are complete semi-Montel C-sequential 
Saks spaces such that $\gamma=\gamma_{s}$ and $\mathcal{F}(\D)\subset\mathcal{H}(\D)$:
\begin{enumerate}
\item The \emph{Hardy spaces} $(H^{p},\|\cdot\|_{p})$ given by 
\[
H^{p}\coloneqq\Bigl\{f\in\mathcal{H}(\D)\;|\;
\|f\|_{p}^{p}\coloneqq \sup_{0<r<1}\frac{1}{2\pi}\int_{0}^{2\pi}|f(r\euler^{\i\theta})|^{p}\d\theta<\infty\Bigr\}.
\]
\item The \emph{weighted Bergman spaces} $(A_{\alpha}^{p},\|\cdot\|_{\alpha,p})$ for $\alpha>-1$ and $1\leq p<\infty$ given by 
\[
A_{\alpha}^{p}\coloneqq\Bigl\{f\in\mathcal{H}(\D)\;|\;
\|f\|_{\alpha,p}^{p}\coloneqq \frac{\alpha+1}{\pi}\int_{\D}|f(z)|^{p}(1-|z|^2)^{\alpha}\d z<\infty\Bigr\}.
\]
\item The \emph{Dirichlet space} $(\mathcal{D},\|\cdot\|_{\mathcal{D}})$ given by 
\[
\mathcal{D}\coloneqq\Bigl\{f\in\mathcal{H}(\D)\;|\;
\|f\|_{\mathcal{D}}^2\coloneqq |f(0)|^2+\frac{1}{\pi}\int_{\D}|f'(z)|^2\d z<\infty\Bigr\}
\]
with the inner product 
\[
\langle f, g\rangle\coloneqq f(0)\overline{g(0)}+\frac{1}{\pi}\int_{\D}f'(z)\overline{g'(z)}\d z,\quad f,g\in \mathcal{D}.
\]
\end{enumerate}
Moreover, the spaces in (a) for $1<p<\infty$, in (b) for $\alpha>-1$ and $1<p<\infty$, and in (c) 
are reflexive.
\end{exa}

Our next example is the Bloch space w.r.t.~a weight $v$.

\begin{exa}[{\cite[4.10 Corollary, p.~19--20]{kruse2023c}}]\label{ex:bloch_saks}
For a continuous function $v\colon\D\to(0,\infty)$ we define the \emph{$v$-Bloch space} 
\[
\mathcal{B}v(\D)\coloneqq\{f\in\mathcal{H}(\D)\;|\;
\|f\|_{\mathcal{B}v(\D)}\coloneqq |f(0)|+\sup_{z\in\D}|f'(z)|v(z)<\infty\}.
\]
Then $(\mathcal{B}v(\D),\|\cdot\|_{\mathcal{B}v(\D)},\tau_{\operatorname{co}})$ is a 
complete semi-Montel C-sequential Saks space such that $\gamma=\gamma_{s}$. 
\end{exa}

For $\alpha>0$ we get with $v_{\alpha}(z)\coloneqq (1-|z|^2)^{\alpha}$, $z\in\D$, 
the usual Bloch type space $\mathcal{B}_{\alpha}\coloneqq\mathcal{B}v_{\alpha}(\D)$ (see \cite[p.~1144]{zhu1993}). 

The case $p=\infty$ in \prettyref{ex:hardy_bergman_dirichlet} (a) can also be covered, namely, 
we have the following result for weighted $H^{\infty}$-spaces.

\begin{exa}[{\cite[4.6 Corollary, p.~17--18]{kruse2023c}}]\label{ex:weighted_holom_saks}
For an open set $\Omega\subset\C$ and a continuous function $v\colon\Omega\to (0,\infty)$ we set 
\[
\mathcal{H}v(\Omega)\coloneqq \{f\in\mathcal{H}(\Omega)\;|\;\|f\|_{v}\coloneqq\sup_{z\in\Omega}|f(z)|v(z)<\infty\}.
\]
Then $(\mathcal{H}v(\Omega),\|\cdot\|_{v},\tau_{\operatorname{co}})$ is a 
complete semi-Montel C-sequential Saks space such that $\gamma=\gamma_{s}$. 
\end{exa}

If $v=\mathds{1}$, we set $H^{\infty}(\Omega)\coloneqq \mathcal{H}v(\Omega)$ 
and $H^{\infty}\coloneqq H^{\infty}(\D)$, which is the Hardy space of bounded holomorphic functions on $\D$ 
(see \cite[p.~253]{zhu2007}). Here, $\mathds{1}$ denotes the map 
$\mathds{1}\colon\Omega\to\K$, $\mathds{1}(x)\coloneqq 1$, on a set $\Omega\subset\K$.
In our following example we consider the space of bounded Dirichlet series which is a topological subspace of $H^{\infty}(\C_{+})$ where $\C_{+}\coloneqq\{z\in\C\;|\;\re z >0\}$ is the open right half-plane, see e.g.~\cite[p.~27]{contreras2023}.

\begin{exa}\label{ex:bDirichlet_series}
We define the \emph{space} $\mathscr{H}^{\infty}$ \emph{of bounded Dirichlet series} as the topological subspace of 
$H^{\infty}(\C_{+})$ consisting of bounded holomorphic functions on $\C_{+}$ that can be written as a Dirichlet series on 
some half-plane contained in $\C_{+}$. Further, we denote by $\|\cdot\|_{\mathscr{H}^{\infty}}$ the restriction of the norm
$\|\cdot\|_{\mathds{1}}$ of $H^{\infty}(\C_{+})$ to $\mathscr{H}^{\infty}$. 
Then $(\mathscr{H}^{\infty},\|\cdot\|_{\mathscr{H}^{\infty}},\tau_{\operatorname{co}})$ is a 
complete semi-Montel C-sequential Saks space such that $\gamma=\gamma_{s}$. 
\end{exa}
\begin{proof}
Due to \cite[Lemma 18, p.~227]{bayart2002} the unit ball $B_{\|\cdot\|_{\mathscr{H}^{\infty}}}$ of $\mathscr{H}^{\infty}$ is 
$\tau_{\operatorname{co}}$-compact. Therefore $(\mathscr{H}^{\infty},\|\cdot\|_{\mathscr{H}^{\infty}},\tau_{\operatorname{co}})$ 
is a complete semi-Montel Saks space by \cite[I.3.1 Lemma, p.~27]{cooper1978}, \cite[I.1.14 Proposition, p.~11]{cooper1978}
and \cite[I.1.13 Proposition, p.~11]{cooper1978}. Condition (ii) of \prettyref{rem:mixed=submixed} (b) yields that 
$\gamma=\gamma_{s}$. Further, the topology $\tau_{\operatorname{co}}$ is metrisable on $\mathscr{H}^{\infty}$ and 
thus $(\mathscr{H}^{\infty},\|\cdot\|_{\mathscr{H}^{\infty}},\tau_{\operatorname{co}})$ is C-sequential by 
\cite[Proposition 5.7, p.~2681--2682]{kruse_meichnser_seifert2018}.
\end{proof}

For our last example of this section we recall that a completely regular Hausdorff space 
$\Omega$ is called a $k_{\R}$-space if any map $f\colon\Omega\to\R$ whose restriction to each compact $K\subset\Omega$ 
is continuous, is already continuous on $\Omega$ (see \cite[p.~487]{michael1973}). 
In particular, Polish spaces and locally compact Hausdorff spaces are $k_{\R}$-spaces.
Due to \cite[4.5 Corollary, p.~15--16]{kruse2023c} and \cite[Remark 3.19 (a), p.~14]{kruse_schwenninger2022}
we have the following result.

\begin{exa}\label{ex:cont_saks}
For a completely regular Hausdorff space $\Omega$ and a continuous function $v\colon\Omega\to (0,\infty)$ we set 
\[
\mathcal{C}v(\Omega)\coloneqq \{f\in\mathcal{C}(\Omega)\;|\;\|f\|_{v}\coloneqq\sup_{x\in\Omega}|f(x)|v(x)<\infty\}.
\]
Then $(\mathcal{C}v(\Omega),\|\cdot\|_{v},\tau_{\operatorname{co}})$ is a Saks space such that $\gamma=\gamma_{s}$. 
This Saks space is complete if $\Omega$ is a $k_{\R}$-space. 
It is C-sequential if $\Omega$ is a hemicompact Hausdorff $k_{\R}$-space or a Polish space. 
\end{exa}

If $\Omega\subset\C$ is open, then $\mathcal{H}v(\Omega)$ from \prettyref{ex:weighted_holom_saks} is a subspace 
of $\mathcal{C}v(\Omega)$ and the mixed topologies are compatible, i.e.
\[
\gamma\left({\|\cdot\|_{v}}_{\mid \mathcal{H}v(\Omega)},{\tau_{\operatorname{co}}}_{\mid \mathcal{H}v(\Omega)}\right)
=\gamma\left(\|\cdot\|_{v},{\tau_{\operatorname{co}}}_{\mid \mathcal{C}v(\Omega)}\right)_{\mid \mathcal{H}v(\Omega)}
\]
by \cite[I.4.6 Lemma, p.~44]{cooper1978}. If $v=\mathds{1}$, then we get the space of bounded continuous functions 
$\mathcal{C}_{b}(\Omega)=\mathcal{C}v(\Omega)$ on $\Omega$ and $\|\cdot\|_{v}=\|\cdot\|_{\infty}$ is the supremum norm. 
Further examples of Saks spaces may be found, for instance, 
in \cite{cooper1978,kruse2023c,kruse_schwenninger2022,kruse_seifert2022b,kuehnemund2001}.

\section{Semiflows, semicocycles and semicoboundaries}
\label{sect:semiflows_cocycles}

In this section we recall the notions and properties of semiflows, associated semicoclyes and of semicoboundaries, 
which form a special class of semicocycles. 

\begin{defn}
Let $I$, $\Omega$ and $Y$ be Hausdorff spaces and $\varphi\coloneqq (\varphi_{t})_{t\in I}$ a family 
of functions $\varphi_{t}\colon\Omega\to Y$. 
\begin{enumerate}
\item We call $\varphi$ \emph{separately continuous} if 
$\varphi_{t}$ and $\varphi_{(\cdot)}(x)\colon I\to Y$ are continuous for all $t\in I$ and $x\in\Omega$. 
\item We call $\varphi$ \emph{jointly continuous} if the map $I\times\Omega\to Y$, $(t,x)\mapsto \varphi_{t}(x)$, is continuous 
where $I\times\Omega$ is equipped with the product topology. 
\item Let $I\coloneqq [0,\infty)$. We say that $\varphi_{(\cdot)}(x)\in\mathcal{C}^{1}[0,\infty)$ 
for $x\in\Omega$ if $\varphi_{(\cdot)}(x)$ is continuously differentiable on $[0,\infty)$ 
where differentiability in $t=0$ means right-differentiability in $t=0$. 
Further, we set $\dt{\varphi}_{t_{0}}(x)\coloneqq (\frac{\partial}{\partial t}\varphi_{(\cdot)}(x))(t_{0})$ 
for all $t_{0}\in [0,\infty)$. 
\item Let $\Omega\subset\K$ be open. If $\varphi_{t}\in\mathcal{C}_{\K}^{1}(\Omega)$ 
for $t\in I$, we set $\varphi_{t}'(x_{0})\coloneqq (\frac{\partial}{\partial x}\varphi_{t})(x_{0})$ for all $x_{0}\in\Omega$. 
\end{enumerate}
\end{defn}

Let us come to semiflows.

\begin{defn}\label{defn:semiflow}
Let $\Omega$ be a Hausdorff space. A family $\varphi\coloneqq (\varphi_{t})_{t\geq 0}$ 
of continuous functions $\varphi_{t}\colon\Omega\to\Omega$ is called a \emph{semiflow} if 
\begin{enumerate}
\item[(i)] $\varphi_{0}(x)=x$ for all $x\in\Omega$, and 
\item[(ii)] $\varphi_{t+s}(x)=(\varphi_{t}\circ\varphi_{s})(x)$ for all $t,s\geq 0$ and $x\in\Omega$.
\end{enumerate}
We call a semiflow $\varphi$ \emph{trivial} and write $\varphi=\id$ if $\varphi_{t}=\id$ for all $t\geq 0$. 
We call a semiflow $\varphi$ a $C_{0}$\emph{-semiflow} if $\lim_{t\to 0\rlim}\varphi_{t}(x)=x$ for all $x\in\Omega$. 
If $\Omega\subset\C$ is open, we call a semiflow $\varphi$ \emph{holomorphic} if $\varphi_{t}\in\mathcal{H}(\Omega)$ 
for all $t\geq 0$.
\end{defn}

A lot of examples of jointly continuous holomorphic semiflows and their whole classification are given in 
\cite[p.~4--5]{siskakis1998} for $\Omega=\D$, in \cite[Proposition 1.4.26, p.~98]{abate1989} for $\Omega=\C$, 
in \cite[Proposition 1.4.27, p.~98]{abate1989} for $\Omega=\C\setminus\{0\}$, 
in \cite[Proposition 1.4.29, p.~99]{abate1989} for $\Omega=\{z\in\C\;|\;r<|z|<1\}$, $0<r<1$, and 
in \cite[Proposition 1.4.30, p.~99]{abate1989} for $\Omega=\D\setminus\{0\}$. 
Further examples may be found in \cite[Chap.~8]{bracci2020} and also in the following sections of our paper. 

\begin{prop}\label{prop:jointly_cont_semiflow}
Let $\varphi$ be a semiflow on a locally compact Hausdorff space $\Omega$. 
\begin{enumerate}
\item Let $\Omega$ be $\sigma$-compact. Then $\varphi$ is jointly continuous if and only if $\varphi$ is $C_{0}$.
\item Let $\Omega$ be metrisable. Then $\varphi$ is jointly continuous if and only if 
$\varphi$ is separately continuous.
\end{enumerate}
\end{prop}
\begin{proof}
(a) The implication $\Rightarrow$ is obvious, the other implication follows directly from 
\cite[Theorems 2.2, 2.3, p.~692]{dorroh1967}.

(b) Again, the implication $\Rightarrow$ is obvious, the other implication follows directly from 
\cite[2., p.~318--319]{chernoff1975}.
\end{proof}

An important concept for semiflows is their generator.

\begin{defn}\label{defn:generator_semiflow}
Let $\varphi$ be a semiflow on a Hausdorff space $\Omega$ such that 
$\varphi_{(\cdot)}(x)\in\mathcal{C}^{1}[0,\infty)$ for all $x\in\Omega$. 
A continuous function $G\colon\Omega\to\Omega$ is called the \emph{generator} of $\varphi$ if 
$\dt{\varphi}_{t}(x)=G(\varphi_{t}(x))$ for all $t\geq 0$ and $x\in\Omega$.
\end{defn}

\begin{rem}\label{rem:generator_semiflow}
If existing, the generator of $\varphi$ is uniquely determined because we have $G(x)=G(\varphi_{0}(x))=\dt{\varphi}_{0}(x)$ 
for all $x\in\Omega$.
\end{rem}

The generator is also called the \emph{speed} of the semiflow (see \cite[p.~210]{batty1987} where the symbol 
$\lambda$ is used for $G$). For a separately continuous semiflow $\varphi$ the existence of the generator is equivalent 
to right-differentiability in $t=0$ and continuity of $\dt{\varphi}_{0}$.

\begin{prop}\label{prop:charac_generator}
Let $\Omega$ be a Hausdorff space and $\varphi$ a separately continuous semiflow on $\Omega$.
Then $\varphi_{(\cdot)}(x)\in\mathcal{C}^{1}[0,\infty)$ for all $x\in\Omega$ and $\dt{\varphi}_{t}\in\mathcal{C}(\Omega)$ 
for all $t\geq 0$ if and only if 
$\varphi_{(\cdot)}(x)$ is right-differentiable in $t=0$ for all $x\in\Omega$ and $\dt{\varphi}_{0}\in\mathcal{C}(\Omega)$. 
In this case $\dt{\varphi}_{t}(x)=\dt{\varphi}_{0}(\varphi_{t}(x))$ for all $t\geq 0$ and $x\in\Omega$, 
and $\dt{\varphi}_{0}$ is the generator of $\varphi$.
\end{prop}
\begin{proof}
We only need to prove the implication $\Leftarrow$. Let $\varphi_{(\cdot)}(x)$ be right-differentiable in $t=0$ 
for all $x\in\Omega$ and $\dt{\varphi}_{0}\in\mathcal{C}(\Omega)$. For $x\in\Omega$ we claim that 
$\varphi_{(\cdot)}(x)$ is continuously right-differentiable on $[0,\infty)$ with right-derivative 
$\dt{\varphi}_{0}(\varphi_{t}(x))$ for all $t\geq 0$. Indeed, we have 
\[
\lim_{s\to 0\rlim}\frac{\varphi_{t+s}(x)-\varphi_{t}(x)}{s}
=\lim_{s\to 0\rlim}\frac{\varphi_{s}(\varphi_{t}(x))-\varphi_{t}(x)}{s}
=\dt{\varphi}_{0}(\varphi_{t}(x))
\]
for all $t\geq 0$. Thus $\varphi_{(\cdot)}(x)$ is right-differentiable on $[0,\infty)$ and the right-derivative is continuous 
(in $t$) because $\dt{\varphi}_{0}\in\mathcal{C}(\Omega)$ and $\varphi$ is separately continuous. 
It follows that the continuous function $\varphi_{(\cdot)}(x)$ is continuously differentiable on $[0,\infty)$ with 
$\dt{\varphi}_{t}(x)=\dt{\varphi}_{0}(\varphi_{t}(x))$ for all $t\geq 0$ 
(see e.g.~\cite[Chap.~2, Corollary 1.2, p.~43]{pazy1983}). 
\end{proof}

Further sufficient and necessary conditions for a given continuous function $G\colon\R\to\R$ to be the generator of 
a jointly continuous flow on $\R$ are contained in \cite[Lemma 2.2, p.~212]{batty1987}. 
The notion of a generator in the case of a jointly continuous holomorphic semiflow was introduced in \cite[p.~103]{berkson1978}. 
In this case the generator always exists and is not only continuous but even holomorphic.  

\begin{thm}[{\cite[(1.1) Theorem, p.~101--102]{berkson1978}}]\label{thm:generator_hol_semiflow}
Let $\Omega\subset\C$ be open and $\varphi$ a jointly continuous holomorphic semiflow on $\Omega$. 
Then $\varphi_{(\cdot)}(z)\in\mathcal{C}^{1}[0,\infty)$ for all $z\in\Omega$ and there is $G\in\mathcal{H}(\Omega)$ such that
$\dt{\varphi}_{t}(z)=G(\varphi_{t}(z))$ for all $t\geq 0$ and $z\in\Omega$.
\end{thm}

We also have the following generalisation of \cite[Proposition 10.1.8 (1), p.~276--277]{bracci2020} where $\K=\C$ 
and $\Omega=\D$.

\begin{prop}\label{prop:diff_generator}
Let $\varphi$ be a semiflow on an open set $\Omega\subset\K$ such that $\varphi_{(\cdot)}(x)\in\mathcal{C}^{1}[0,\infty)$ 
for all $x\in\Omega$ and $\varphi_{t}\in\mathcal{C}^{1}_{\K}(\Omega)$ for all $t\geq 0$. 
Then 
\[
\dt{\varphi}_{0}(\varphi_{t}(x))=\varphi_{t}'(x)\dt{\varphi}_{0}(x),\quad t\geq 0,\,x\in\Omega.
\]
If in addition $\varphi$ has a generator $G$, then 
\[
\dt{\varphi}_{t}(x)=G(\varphi_{t}(x))=\varphi_{t}'(x)G(x),\quad t\geq 0,\,x\in\Omega.
\]
\end{prop}
\begin{proof}
For all $s,t\geq 0$ and $x\in\Omega$ we have $\varphi_{s}(\varphi_{t}(x))=\varphi_{t+s}(x)=(\varphi_{t}\circ\varphi_{s})(x)$. 
By differentiating w.r.t.~$s$ we get 
\[
\dt{\varphi}_{s}(\varphi_{t}(x))=\varphi_{t}'(\varphi_{s}(x))\dt{\varphi}_{s}(x)
\]
for all $s,t\geq 0$ and $x\in\Omega$, which yields $\dt{\varphi}_{0}(\varphi_{t}(x))=\varphi_{t}'(x)\dt{\varphi}_{0}(x)$ 
for $s=0$. The rest of the statement follows from the definition of a generator and \prettyref{rem:generator_semiflow}.
\end{proof}

Next, let us recall the notion of a semicocycle for a semiflow.

\begin{defn}
Let $\varphi\coloneqq (\varphi_{t})_{t\geq 0}$ be a semiflow on a Hausdorff space $\Omega$. 
A family $m\coloneqq (m_{t})_{t\geq 0}$ of continuous functions $m_{t}\colon\Omega\to\K$ 
is called a multiplicative \emph{semicocycle} for $\varphi$ if 
\begin{enumerate}
\item[(i)] $m_{0}(x)=1$ for all $x\in\Omega$, and 
\item[(ii)] $m_{t+s}(x)=m_{t}(x)m_{s}(\varphi_{t}(x))$ for all $t,s\geq 0$ and $x\in\Omega$.
\end{enumerate}
We call a semicocycle $m$ \emph{trivial} and write $m=\mathds{1}$ if $m_{t}=\mathds{1}$ for all $t\geq 0$. 
We call a semicocycle $m$ a $C_{0}$\emph{-semicocycle} if $\lim_{t\to 0\rlim}m_{t}(x)=1$ for all $x\in\Omega$. 
If $\Omega\subset\C$ is open, we call a semicocycle $m$ \emph{holomorphic} if $m_{t}\in\mathcal{H}(\Omega)$ 
for all $t\geq 0$.
\end{defn}

If $\varphi$ is a holomorphic semiflow on an open set $\Omega\subset\C$, then a simple example of a holomorphic 
semicocycle $m$ for $\varphi$ is given by the complex derivatives $m_{t}\coloneqq \varphi_{t}'$ for $t\geq 0$ 
by the chain rule. There is an analogon of \prettyref{prop:loc_bounded_sg} for semicocycles 
due to K\"onig \cite{koenig1990} which will be important later on. 
The similarity to \prettyref{prop:loc_bounded_sg} is not a coincidence 
because they use the same ideas, which can be found in the proofs of 
\cite[VIII.1.4 Lemma, VIII.1.5 Corollary, p.~618--619]{dunford1958}. 

\begin{prop}\label{prop:semicocycle_bounded}
Let $\varphi$ be a semiflow on a Hausdorff space $\Omega$ and $m$ a semicocycle for $\varphi$. 
Then the following assertions are equivalent.
\begin{enumerate}
\item There exist $M\geq 1$ and $\omega\in\R$ such that 
$\|m_{t}\|_{\infty}\leq M\euler^{\omega t}$ for all $t\geq 0$.
\item It holds that $\sup_{t\in[0,t_{0}]}\|m_{t}\|_{\infty}<\infty$ 
for all $t_{0}\geq 0$.
\item There exists $t_{0}>0$ such that $\sup_{t\in[0,t_{0}]}\|m_{t}\|_{\infty}<\infty$. 
\item It holds that $\|m_{t}\|_{\infty}<\infty$ for all $t\geq 0$.
\item It holds that $\limsup_{t\to 0\rlim}\|m_{t}\|_{\infty}<\infty$. 
\end{enumerate}
\end{prop}
\begin{proof}
The implications (d)$\Rightarrow$(a)$\Rightarrow$(e)$\Rightarrow$(c)$\Rightarrow$(d) follow 
from the proof of \cite[Lem\-ma 2.1 (a), p.~472]{koenig1990} (we note that it is not relevant 
for the proof that $(h_{t})_{t\geq 0}\coloneqq m$ in the cited lemma is assumed to be holomorphic 
and $\Omega$ to be equal to $\D$). Moreover, the implications (a)$\Rightarrow$(b)$\Rightarrow$(c) clearly hold.
\end{proof}

\begin{defn}
Let $\varphi$ be a semiflow on a Hausdorff space $\Omega$ and $m$ a semicocycle for $\varphi$. 
We call the tuple $(m,\varphi)$ a \emph{co-semiflow} on $\Omega$. 
We call a co-semiflow $(m,\varphi)$ jointly continuous (separately continuous, $C_{0}$, holomorphic) if 
$\varphi$ and $m$ are both jointly continuous (separately continuous, $C_{0}$, holomorphic). 
\end{defn}

\begin{prop}\label{prop:jointly_cont_cocycle}
Let $(m,\varphi)$ be a co-semiflow on an open subset $\Omega$ of a metric space and $\varphi$ jointly continuous.  
Then $m$ is jointly continuous if and only if $m$ is $C_{0}$.
\end{prop}
\begin{proof}
The implication $\Rightarrow$ clearly holds. The other implication follows from \cite[Definition 3.1, p.~1203]{elin2019}, 
the proof of \cite[Theorem 3.1, p.~1204]{elin2019} with $\mathcal{D}\coloneqq \Omega$ and $\mathcal{A}\coloneqq\C$, 
and the observation that the assumption that $\mathcal{D}$ is an open connected subset of a Banach space 
(see \cite[p.~1200]{elin2019}) is not needed in the proof of \cite[Theorem 3.1, p.~1204]{elin2019}. 
\end{proof}

Analogously to \prettyref{prop:charac_generator} we have the following result for semicocycles.

\begin{prop}\label{prop:cont_diff_semicocycle}
Let $\Omega$ be a Hausdorff space and $(m,\varphi)$ a separately continuous co-semiflow on $\Omega$.
Then $m_{(\cdot)}(x)\in\mathcal{C}^{1}[0,\infty)$ for all $x\in\Omega$ and $\dt{m}_{t}\in\mathcal{C}(\Omega)$ 
for all $t\geq 0$ if and only if 
$m_{(\cdot)}(x)$ is right-differentiable in $t=0$ for all $x\in\Omega$ and $\dt{m}_{0}\in\mathcal{C}(\Omega)$. 
In this case $\dt{m}_{t}(x)=m_{t}(x)\dt{m}_{0}(\varphi_{t}(x))$ and $m_{t}(x)=\exp(\int_{0}^{t}\dt{m}_{0}(\varphi_{s}(x))\d s)$ 
for all $t\geq 0$ and $x\in\Omega$.
\end{prop}
\begin{proof}
We only need to prove the implication $\Leftarrow$. Let $m_{(\cdot)}(x)$ be right-differentiable in $t=0$ 
for all $x\in\Omega$ and $\dt{m}_{0}\in\mathcal{C}(\Omega)$. For $x\in\Omega$ we claim that 
$m_{(\cdot)}(x)$ is continuously right-differentiable on $[0,\infty)$ with right-derivative 
$m_{t}(x)\dt{m}_{0}(\varphi_{t}(x))$ for all $t\geq 0$. Indeed, we have 
\[
\lim_{s\to 0\rlim}\frac{m_{t+s}(x)-m_{t}(x)}{s}
=m_{t}(x)\lim_{s\to 0\rlim}\frac{m_{s}(\varphi_{t}(x))-1}{s}
=m_{t}(x)\dt{m}_{0}(\varphi_{t}(x))
\]
for all $t\geq 0$. Thus $m_{(\cdot)}(x)$ is right-differentiable on $[0,\infty)$ and the right-derivative is continuous (in $t$) 
because $\dt{m}_{0}\in\mathcal{C}(\Omega)$ and $(m,\varphi)$ is separately continuous. 
It follows from \cite[Chap.~2, Corollary 1.2, p.~43]{pazy1983} that the continuous function $m_{(\cdot)}(x)$ is continuously differentiable on $[0,\infty)$ with 
$\dt{m}_{t}(x)=m_{t}(x)\dt{m}_{0}(\varphi_{t}(x))$ for all $t\geq 0$.  
Thus for $x\in\Omega$ we know that the map $t\mapsto m_{t}(x)$ solves the initial value problem 
\[
\dt{w}(t)=w(t)g(\varphi_{t}(x)),\;t\geq 0,\;w(0)=1,
\]
with $g\coloneqq\dt{m}_{0}\in\mathcal{C}(\Omega)$.
Another solution of this initial value problem is given by the map $t\mapsto\exp(\int_{0}^{t}\dt{m}_{0}(\varphi_{s}(x))\d s)$. 
Since the solution of this initial value problem is unique (e.g.~by \cite[Chap.~1, Theorem 3, p.~7]{hurewicz1970}), 
we get that $m_{t}(x)=\exp(\int_{0}^{t}\dt{m}_{0}(\varphi_{s}(x))\d s)$ for all $t\geq 0$.
\end{proof}

We have the following construction of a semicocycle given a jointly continuous semiflow and a continuous function on a 
locally compact metric space.

\begin{prop}\label{prop:jointly_cont_int_cocycle}
Let $\varphi$ be a jointly continuous semiflow on a locally compact metric space $\Omega$ and $g\in\mathcal{C}(\Omega)$. 
Then the following assertions hold.
\begin{enumerate}
\item The family $m\coloneqq (m_{t})_{t\geq 0}$ given by 
$m_{t}(x)\coloneqq\exp(\int_{0}^{t}g(\varphi_{s}(x))\d s)$ for all $t\geq 0$ and $x\in\Omega$ 
is a jointly continuous semicocycle for $\varphi$. In particular, $m_{(\cdot)}(x)\in\mathcal{C}^{1}[0,\infty)$, 
$\dt{m}_{t}(x)=m_{t}(x)g(\varphi_{t}(x))$, $\dt{m}_{0}(x)=g(x)$ and $m_{t}(x)\neq 0$ for all $t\geq 0$ and $x\in\Omega$.
\item If $\Omega\subset\K$ is open and $\varphi_{t},g\in\mathcal{C}^{1}_{\K}(\Omega)$ for all $t\geq 0$, 
then $m_{t}\in\mathcal{C}^{1}_{\K}(\Omega)$ for all $t\geq 0$ with $m$ from part (a). 
\end{enumerate}
\end{prop}
\begin{proof}
(a) For $t\geq 0$ we note that $g\circ\varphi_{t}\in\mathcal{C}(\Omega)$, the map $g(\varphi_{(\cdot)}(x))$ is continuous 
and therefore integrable on $[0,t]$ and for every $x_{0}\in\Omega$ there is a compact neighbourhood $U\subset\Omega$ 
of $x_{0}$ such that $|g(\varphi_{(\cdot)}(x))|\leq \sup\{|g(\varphi_{s}(w))|\;|\;(s,w)\in[0,t]\times U\}<\infty$ on $[0,t]$ 
for all $x\in U$ because $\Omega$ is locally compact, $g\in\mathcal{C}(\Omega)$ and $\varphi$ jointly continuous. 
Setting $F_{t}\colon\Omega\to\K$, $F_{t}(x)\coloneqq\int_{0}^{t}g(\varphi_{s}(x))\d s$, we deduce that 
$F_{t}$ is continuous on the metric space $\Omega$ by \cite[5.6 Satz, p.~147]{elstrodt2005} and thus 
$m_{t}=\exp\circ F_{t}$ as well. From here it is easy to check that $m$ is a $C_{0}$-semicocycle for $\varphi$ 
and so jointly continuous by \prettyref{prop:jointly_cont_cocycle}. 
The rest of statement (a) follows from the integral form of $m_{t}(x)$ and \prettyref{prop:cont_diff_semicocycle}.

(b) In the case $\K=\R$ the statement follows from \cite[5.7 Satz, p.~147--148]{elstrodt2005} and in the case $\K=\C$ from 
\cite[5.8 Satz, p.~148]{elstrodt2005}.
\end{proof}

On connected proper subsets of $\C$ every jointly continuous holomorphic semicocycle of 
a jointly continuous holomorphic semiflow is actually of the integral form in \prettyref{prop:jointly_cont_int_cocycle} (a).

\begin{prop}\label{prop:cocycle_int_hol}
Let $\Omega\subset\C$ be open and connected, and $(m,\varphi)$ a jointly continuous holomorphic co-semiflow on $\Omega$. 
Then it holds $m_{(\cdot)}(z)\in\mathcal{C}^{1}[0,\infty)$, $\dt{m}_{t}\in\mathcal{H}(\Omega)$ and 
$m_{t}(z)=\exp(\int_{0}^{t}\dt{m}_{0}(\varphi_{s}(z))\d s)$ for all $t\geq 0$ and $z\in\Omega$.
\end{prop}
\begin{proof}
Due to \cite[Theorem 4, p.~3392]{jafari2005} we have $m_{(\cdot)}(z)\in\mathcal{C}^{1}[0,\infty)$ for all $z\in\Omega$ 
and $\dt{m}_{t}\in\mathcal{H}(\Omega)$ for all $t\geq 0$. Then it follows from \prettyref{prop:cont_diff_semicocycle} 
that $m_{t}(z)=\exp(\int_{0}^{t}\dt{m}_{0}(\varphi_{s}(z))\d s)$ for all $t\geq 0$ and $z\in\Omega$. 
\end{proof}

\prettyref{prop:cocycle_int_hol} improves \cite[Theorem 3, p.~3392]{jafari2005} with $g=\dt{m}_{0}\in\mathcal{H}(\Omega)$ from 
simply connected open $\Omega\subset\C$ to just connected open $\Omega\subset\C$. 
Moreover, \prettyref{prop:cocycle_int_hol} implies \cite[Lemma 2.1 (b), p.~472]{koenig1990}. 
There is another way to construct semi\-co\-cycles for a semiflow apart from the one in \prettyref{prop:jointly_cont_int_cocycle}, 
namely, so-called (semi)coboundaries, see e.g.~\cite[p.~240]{jafari1997}, \cite[p.~469--470]{koenig1990} 
and \cite[p.~513]{parry1972}. For that construction we need the notion of a fixed point of a semiflow. 

\begin{defn}
Let $\Omega$ be a Hausdorff space and $\varphi$ a semiflow on $\Omega$. We call $x\in\Omega$ a \emph{fixed point} of $\varphi$ 
if it is a common fixed point of all $\varphi_{t}$, i.e.~$\varphi_{t}(x)=x$ for all $t\geq 0$. We denote the set of all fixed 
points of $\varphi$ by $\operatorname{Fix}(\varphi)\coloneqq\{x\in\Omega\;|\;\forall\;t\geq 0:\;\varphi_{t}(x)=x\}$.
\end{defn}

Let $\Omega$ be a Hausdorff space and $\varphi$ a semiflow on 
$\Omega$. Let $\omega\in\mathcal{C}(\Omega)$, $\omega\neq 0$, such that its set of zeros 
$N_{\omega}\coloneqq\{x\in\Omega\;|\;\omega(x)=0\}$ fulfils that 
$N_{\omega}\subset\operatorname{Fix}(\varphi)$, and that $\Omega\setminus N_{\omega}$ is dense in $\Omega$. 
We set 
\[
m_{t}^{\omega}(x)
\coloneqq m_{t}^{\omega,\varphi}(x)
\coloneqq \frac{\omega(\varphi_{t}(x))}{\omega(x)},\quad t\geq 0,\, x\in\Omega\setminus N_{\omega},
\]
and note that $m_{t}^{\omega}\colon \Omega\setminus N_{\omega}\to \K$ is continuous for all $t\geq 0$. Moreover, 
\begin{equation}\label{eq:coboundary_outside_zeros}
m_{0}^{\omega}(x)=1\quad\text{and}\quad m_{t+s}^{\omega}(x)=m_{t}^{\omega}(x)m_{s}^{\omega}(\varphi_{t}(x)),\quad t,s\geq 0,\, 
x\in \Omega\setminus N_{\omega}.
\end{equation}
If $N_{\omega}\neq \varnothing$, suppose additionally that $m_{t}^{\omega}$ is continuously extendable on $\Omega$ 
for all $t\geq 0$ and denote the (unique) extension by $m_{t}^{\omega}$ as well. 
Then \eqref{eq:coboundary_outside_zeros} also holds for $x\in N_{\omega}$ by continuity and 
the density of $\Omega\setminus N_{\omega}$ in $\Omega$. Thus $m^{\omega}\coloneqq (m_{t}^{\omega})_{t\geq 0}$ is a 
semicocycle for $\varphi$ under this assumption. 

\begin{defn}
Let $\Omega$ be a Hausdorff space and $\varphi$ a semiflow on $\Omega$. A semicocycle $m$ for $\varphi$ is called 
a \emph{semicoboundary} for $\varphi$ if there is $\omega\in\mathcal{C}(\Omega)$, $\omega\neq 0$, 
such that $N_{\omega}\subset\operatorname{Fix}(\varphi)$, the set $\Omega\setminus N_{\omega}$ is dense in $\Omega$, and 
$m=m^{\omega}$.
\end{defn}

\begin{prop}\label{prop:fixed_points}
Let $\Omega$ be a Hausdorff space and $\varphi$ a semiflow on $\Omega$ such 
that $\varphi_{(\cdot)}(x)\in\mathcal{C}^{1}[0,\infty)$ for all $x\in\Omega$. 
Then the following assertions hold.
\begin{enumerate}
\item $x_{0}\in\operatorname{Fix}(\varphi)$ if and only if $\dt{\varphi}_{t}(x_{0})=0$ 
for all $t\geq 0$.
\item Suppose that $\varphi$ has a generator $G$. Then $\operatorname{Fix}(\varphi)\subset N_{G}$.
\end{enumerate}
\end{prop}
\begin{proof}
(a) $x_{0}\in\operatorname{Fix}(\varphi)$ if and only if $\varphi_{t}(x_{0})=x_{0}$ for all 
$t\geq 0$. By differentiating w.r.t.~$t$ we get $\dt{\varphi}_{t}(x_{0})=0$ 
for all $t\geq 0$. 

Conversely, suppose that $\dt{\varphi}_{t}(x_{0})=0$ for all $t\geq 0$. This implies that 
there is a constant $C(x_{0})$ w.r.t.~$t$ such that $\varphi_{t}(x_{0})=C(x_{0})$ for all 
$t\geq 0$. Since $x_{0}=\varphi_{0}(x_{0})=C(x_{0})$, we obtain that 
$x_{0}\in\operatorname{Fix}(\varphi)$.

(b) Since $G$ is a generator of $\varphi$, we have $G(x_{0})=\dt{\varphi}_{0}(x_{0})=0$ 
for any $x_{0}\in\operatorname{Fix}(\varphi)$ by (a), which implies $\operatorname{Fix}(\varphi)\subset N_{G}$.
\end{proof}

\begin{prop}\label{prop:fixed_points_holomorphic}
Let $\Omega\subset\C$ be open and $\varphi$ a jointly continuous holomorphic semiflow on $\Omega$ 
with generator $G$. Then the following assertions hold.
\begin{enumerate}
\item If $\Omega$ is connected, then $\operatorname{Fix}(\varphi)=N_{G}$ and $\varphi_{t}$ is 
injective for all $t\geq 0$.
\item If $\varphi$ is non-trivial, $\Omega$ simply connected and $\Omega\neq\C$, 
then $|\operatorname{Fix}(\varphi)|\leq 1$ where $|\operatorname{Fix}(\varphi)|$ denotes the cardinality of 
$\operatorname{Fix}(\varphi)$. 
\end{enumerate}
\end{prop}
\begin{proof}
(a) The first part of (a) is just \cite[Proposition 1.4.13 (i), p.~89]{abate1989}. 
The second part is \cite[Proposition 1.4.6, p.~85]{abate1989} in combination 
with the fact that every open connected subset of $\C$ is a Riemann surface in the sense of 
\cite[p.~14]{abate1989}.

(b) We choose a biholomorphic map $h\colon\D\to\Omega$ by the Riemann mapping theorem and set 
$\psi_{t}\coloneqq h^{-1}\circ\varphi_{t}\circ h$ for every $t\geq 0$. 
Then $\psi\coloneqq (\psi_{t})_{t\geq 0}$ is a non-trivial 
jointly continuous holomorphic semiflow on $\D$ 
with $\operatorname{Fix}(\psi)=h^{-1}(\operatorname{Fix}(\varphi))$. 
Due to \cite[Remark 10.1.6, p.~275]{bracci2020} we know that $|\operatorname{Fix}(\psi)|\leq 1$, 
implying that statement (b) holds.
\end{proof}

\begin{exa}\label{ex:semicoboundary}
Let $\Omega\subset\C$ be open and connected, $\varphi$ a holomorpic semiflow on $\Omega$ and $\omega\in\mathcal{H}(\Omega)$, 
$\omega\neq 0$, such that $N_{\omega}\subset\operatorname{Fix}(\varphi)$. Then the semicoboundary $m^{\omega}$ 
for $\varphi$ satisfies 
\[
m_{t}^{\omega}(z)=
\begin{cases}
\frac{\omega(\varphi_{t}(z))}{\omega(z)} &,\; z\in\Omega\setminus N_{\omega},\\
\bigl(\varphi_{t}'(z)\bigr)^{\operatorname{ord}_{\omega}(z)} &,\; z\in N_{\omega},
\end{cases}
\]
and $m_{t}^{\omega}\in\mathcal{H}(\Omega)$ for all $t\geq 0$ where $\operatorname{ord}_{\omega}(z)\in\N$ 
is the order of the zero $z\in N_{\omega}$ of $\omega$. If $\varphi$ is additionally jointly continuous, 
then $m^{\omega}$ is jointly continuous.
\end{exa}
\begin{proof}
First, we note that $\Omega\setminus N_{\omega}$ is dense in $\Omega$ since $N_{\omega}$ is discrete in $\Omega$. 
For $b\in N_{\omega}$ there is $\psi\in\mathcal{H}(\Omega)$ such that $\psi(b)\neq 0$ and 
$\omega(z)=(z-b)^{n}\psi(z)$ for all $z\in\Omega$ with $n\coloneqq\operatorname{ord}_{\omega}(b)$.
Let $t\geq 0$. Then we have for all $z\in \Omega\setminus N_{\omega}$
\[
\frac{\omega(\varphi_{t}(z))}{\omega(z)}
=\Bigl(\frac{\varphi_{t}(z)-b}{z-b}\Bigr)^{n}\frac{\psi(\varphi_{t}(z))}{\psi(z)}
\underset{z\to b}{\to} (\varphi_{t}'(b)\bigr)^{n}\frac{\psi(\varphi_{t}(b))}{\psi(b)}
=(\varphi_{t}'(b)\bigr)^{n}
\]
because $N_{\omega}\subset\operatorname{Fix}(\varphi)$. 
Therefore $m_{t}^{\omega}$ is continuously extendable on $\Omega$ and this extension is holomorphic on $\Omega$ 
by Riemann's removable singularity theorem.

Now, if $\varphi$ is additionally jointly continuous, then we have $\lim_{t\to 0\rlim}m_{t}^{\omega}(z)=\omega(z)/\omega(z)=1$ 
for all $z\in \Omega\setminus N_{\omega}$. 
Furthermore, the map $[0,\infty)\to\C$, $t\mapsto \varphi_{t}'(z)$, is continuous for every $z\in\Omega$ by 
\cite[Lemma 2.1, p.~242]{jafari1997} with connected $G\coloneqq\Omega$. Hence we have $\lim_{t\to 0\rlim}m_{t}^{\omega}(z)
=(\varphi_{0}'(z))^{\operatorname{ord}_{\omega}(z)}=1^{\operatorname{ord}_{\omega}(z)}=1$ for all $z\in N_{\omega}$. 
We conclude that $m_{t}^{\omega}$ is $C_{0}$ and thus jointly continuous by 
\prettyref{prop:jointly_cont_cocycle}.
\end{proof}

For $\Omega=\D$ the previous example is already contained in \cite[p.~361--362]{siskakis1986}. 
We already observed that $(\varphi_{t}')_{t\geq 0}$ is a simple example of a semicocycle of a holomorphic semiflow $\varphi$ 
on an open set $\Omega\subset\C$. If $\Omega$ is also connected, then it is even a semicoboundary 
by \prettyref{thm:generator_hol_semiflow}, \prettyref{prop:diff_generator} and \prettyref{prop:fixed_points_holomorphic} (a), 
which is jointly continuous by the arguments in the example above.

\begin{exa}\label{ex:semicoboundary_gen}
Let $\Omega\subset\C$ be open and connected, and $\varphi$ a jointly continuous holomorphic semiflow on $\Omega$ with 
generator $G\neq 0$. 
Then $m_{t}^{G}(z)=\varphi_{t}'(z)$ for all $t\geq 0$ and $z\in\Omega$, and $m^{G}$ is jointly continuous. 
\end{exa}

For $\Omega=\D$ the previous example is also contained in \cite[Example 7.4, p.~247--248]{siskakis1998}. 
We have the following relation between semicoboundaries and the 
semicocycles from \prettyref{prop:jointly_cont_int_cocycle} of a 
jointly continuous holomorphic semiflow on simply connected proper subsets $\Omega\subset\C$, 
which generalises \cite[Lemma 2.2, p.~472]{koenig1990} where $\Omega=\D$.

Let $\Omega\subset\C$ be open and connected, $\varphi$ a jointly continuous holomorphic semiflow 
on $\Omega$ with generator $G$ (see \prettyref{thm:generator_hol_semiflow}), and $\omega\in\mathcal{H}(\Omega)$, $\omega\neq 0$, 
such that $N_{\omega}\subset\operatorname{Fix}(\varphi)$. 
The function $z\mapsto \frac{\omega'(z)}{\omega(z)}$ has a pole of order one
in $z_{0}\in N_{\omega}$. Due to \prettyref{prop:fixed_points_holomorphic} (a)
we have $N_{\omega}\subset\operatorname{Fix}(\varphi)=N_{G}$ and thus the holomorphic 
function $\Omega\setminus N_{\omega}\to\C$, 
$z\mapsto \frac{G(z)\omega'(z)}{\omega(z)}$, is continuously extendable in any $z_{0}\in N_{\omega}$. 
By Riemann's removable singularity theorem this extension is holomorphic on $\Omega$ 
and we denote it by $g_{\omega, G}$.

\begin{prop}\label{prop:semicobound_int}
Let $\Omega\subsetneq\C$ be open and simply connected, and $\varphi$ a jointly continuous holomorphic semiflow on $\Omega$ 
with generator $G$. Then the following assertions hold.
\begin{enumerate}
\item If $\omega\in\mathcal{H}(\Omega)$, $\omega\neq 0$, such that $N_{\omega}\subset\operatorname{Fix}(\varphi)$, then 
\begin{equation}\label{eq:semicobound_int}
m_{t}^{\omega}(z)=\exp\Bigl(\int_{0}^{t}g(\varphi_{s}(z))\d s\Bigr),\quad t\geq 0,\, z\in\Omega,
\end{equation}
with $g\coloneqq g_{\omega, G}\in\mathcal{H}(\Omega)$.
\item Let $G\neq 0$ and $g\in\mathcal{H}(\Omega)$. Then there is $\omega\in\mathcal{H}(\Omega)$ such that 
\eqref{eq:semicobound_int} holds if and only if $g(b)/G'(b)\in\N_{0}$ for every $b\in\operatorname{Fix}(\varphi)$. 
In this case $\operatorname{ord}_{\omega}(b)=g(b)/G'(b)$. 
\end{enumerate}
\end{prop}
\begin{proof}
We choose a biholomorphic map $h\colon\D\to\Omega$ by the Riemann mapping theorem and set 
$\psi_{t}\coloneqq h^{-1}\circ\varphi_{t}\circ h$ for every $t\geq 0$. 
Then $\psi\coloneqq (\psi_{t})_{t\geq 0}$ is a jointly continuous holomorphic semiflow on $\D$ 
with $\operatorname{Fix}(\psi)=h^{-1}(\operatorname{Fix}(\varphi))$. 
Let $G_{\psi}$ denote the generator of $\psi$ which exists by \prettyref{thm:generator_hol_semiflow}. We note that 
\begin{align}\label{eq:semicobound_int_gen}
  G_{\psi}(\psi_{t}(z))
&=\dt{\psi}_{t}(z)
 =(h^{-1})'\bigl((\varphi_{t}\circ h)(z)\bigr)\dt{\varphi}_{t}(h(z))
 =(h^{-1})'\bigl((h\circ\psi_{t})(z)\bigr)G(h(z))\nonumber\\
&=\frac{1}{h'(\psi_{t}(z))}G(h(z)),
\end{align}
implying $G_{\psi}(z)=G_{\psi}(\psi_{0}(z))=\frac{1}{h'(z)}G(h(z))$.

(a) If $\varphi$ is trivial, then $\varphi_{t}'(z)=1$ for all $t\geq 0$ and $z\in\Omega$, $G=0$ 
and $\operatorname{Fix}(\varphi)=\Omega$. 
This implies that $g_{\omega, G}(z)=0$ and $m_{t}^{\omega}(z)=1$ for all $z\in\Omega$ 
and so \eqref{eq:semicobound_int} holds. 

Now, suppose that $\varphi$ is non-trivial. Remarking that $\omega\circ h\in\mathcal{H}(\D)$, $\omega\circ h\neq 0$, with $N_{\omega\circ h}
=h^{-1}(N_{\omega})$, we have 
\[
m_{t}^{\omega\circ h,\psi}(z)=\exp\Bigl(\int_{0}^{t}g_{\omega\circ h, G_{\psi}}(\psi_{s}(z))\d s\Bigr)
\]
for every $t\geq 0$ and $z\in\D$ by \cite[Lemma 2.2 (a), p.~472]{koenig1990}. 
Moreover, for $z\in\Omega\setminus\operatorname{Fix}(\varphi)$ we observe that 
$\varphi_{s}(z)\in\Omega\setminus N_{\omega}$ for all $s\geq 0$. 
Indeed, assume that there is $s_{0}\geq 0$ such that $\varphi_{s_{0}}(z)\in N_{\omega}$. 
Then $u\coloneqq\varphi_{s_{0}}(z)\in\operatorname{Fix}(\varphi)$ and 
$\varphi_{s_{0}}(u)=u$. Due to the injectivity of $\varphi_{s_{0}}$ 
by \prettyref{prop:fixed_points_holomorphic} (a) this yields 
$z=u\in\operatorname{Fix}(\varphi)$, which is a contradiction. 
Hence we get for all $s\geq 0$ and $z\in\Omega\setminus\operatorname{Fix}(\varphi)$ 
\begin{align*}
  g_{\omega\circ h, G_{\psi}}(\psi_{s}(h^{-1}(z)))
&=\frac{G_{\psi}\bigl(\psi_{s}(h^{-1}(z))\bigr)(\omega\circ h)'\bigl(\psi_{s}(h^{-1}(z))\bigr)}{(\omega\circ h)\bigl(\psi_{s}(h^{-1}(z))\bigr)}\\
&\underset{\mathclap{\eqref{eq:semicobound_int_gen}}}{=}\frac{\frac{1}{h'(\psi_{s}(h^{-1}(z)))}G(\varphi_{s}(z))\omega'(\varphi_{s}(z))h'(\psi_{s}(h^{-1}(z)))}{\omega(\varphi_{s}(z))}\\
&=\frac{G(\varphi_{s}(z))\omega'(\varphi_{s}(z))}{\omega(\varphi_{s}(z))}
 =g_{\omega, G}(\varphi_{s}(z)),
\end{align*}
yielding $ g_{\omega\circ h, G_{\psi}}(\psi_{s}(h^{-1}(z)))=g_{\omega, G}(\varphi_{s}(z))$ for all 
$z\in\Omega$ by continuity because $|\operatorname{Fix}(\varphi)|\leq 1$ 
by \prettyref{prop:fixed_points_holomorphic} (b). 
It follows for all $t\geq 0$ and $z\in\Omega\setminus N_{\omega}$ that 
\begin{align*}
  \frac{\omega(\varphi_{t}(z))}{\omega(z)}
&=\frac{(\omega\circ h)\bigl(\psi_{t}(h^{-1}(z))\bigr)}{(\omega\circ h)(h^{-1}(z))}
 =m_{t}^{\omega\circ h,\psi}(h^{-1}(z))\\
&=\exp\Bigl(\int_{0}^{t}g_{\omega\circ h, G_{\psi}}(\psi_{s}(h^{-1}(z)))\d s\Bigr)
 =\exp\Bigl(\int_{0}^{t}g_{\omega, G}(\varphi_{s}(z))\d s\Bigr),
\end{align*}
which implies $m_{t}^{\omega}(z)=\exp(\int_{0}^{t}g_{\omega, G}(\varphi_{s}(z))\d s)$ 
for all $z\in\Omega$ by continuity.

(b) First, due to \cite[Lemma 2.2 (b), p.~472]{koenig1990} there is $\widetilde{w}\in\mathcal{H}(\D)$ 
such that 
\begin{equation}\label{eq:semicobound_int_circle}
m_{t}^{\widetilde{\omega},\psi}(z)=\exp\Bigl(\int_{0}^{t}(g\circ h)(\psi_{s}(z))\d s\Bigr)
\end{equation} 
for all $t\geq 0$ and $z\in\D$ if and only if 
$\frac{(g\circ h)(\widetilde{b})}{G_{\psi}'(\widetilde{b})}\in\N_{0}$ for all 
$\widetilde{b}\in\operatorname{Fix}(\psi)$. 
In this case $\operatorname{ord}_{\widetilde{\omega}}(\widetilde{b})=\frac{(g\circ h)(\widetilde{b})}{G_{\psi}'(\widetilde{b})}$.

Second, we observe for all $b\in\operatorname{Fix}(\varphi)=h(\operatorname{Fix}(\psi))$ 
that $G_{\psi}(h^{-1}(b))=0$ by \prettyref{prop:fixed_points} (b). Hence we 
have with $\widetilde{b}\coloneqq h^{-1}(b)\in\operatorname{Fix}(\psi)$ that
\begin{align*}
 \frac{g(b)}{G'(b)}
&=\frac{g(h(\widetilde{b}))}{G'(h(\widetilde{b}))}
 =\frac{(g\circ h)(\widetilde{b})}{(G\circ h)'(\widetilde{b})\frac{1}{h'(\widetilde{b})}}
 \underset{\eqref{eq:semicobound_int_gen}}{=}
 \frac{(g\circ h)(\widetilde{b})}{(h'\cdot G_{\psi})'(\widetilde{b})\frac{1}{h'(\widetilde{b})}}\\
&=\frac{(g\circ h)(\widetilde{b})}{\bigl(h''(\widetilde{b})G_{\psi}(\widetilde{b})
  +h'(\widetilde{b})G_{\psi}'(\widetilde{b})\bigr)\frac{1}{h'(\widetilde{b})}}
 =\frac{(g\circ h)(\widetilde{b})}{G_{\psi}'(\widetilde{b})}.
\end{align*}

Third, suppose there is $\omega\in\mathcal{H}(\Omega)$ such that \eqref{eq:semicobound_int} holds. 
Then we obtain with $\widetilde{\omega}\coloneqq \omega\circ h$ that 
for all $t\geq 0$ and $z\in\D\setminus N_{\widetilde{w}}$
\[
 \frac{\widetilde{\omega}(\psi_{t}(z))}{\widetilde{\omega}(z)}
=\frac{\omega(\varphi_{t}(h(z)))}{\omega(h(z))}
=\exp\Bigl(\int_{0}^{t}g(\varphi_{s}(h(z)))\d s\Bigr)
=\exp\Bigl(\int_{0}^{t}(g\circ h)(\psi_{s}(h(z)))\d s\Bigr),
\]
which extends to \eqref{eq:semicobound_int_circle} for all $z\in\D$ by continuity. 
Due to the first and the second part of (b) this means that $\frac{g(b)}{G'(b)}\in\N_{0}$ 
for all $b\in\operatorname{Fix}(\varphi)$. 

Fourth, suppose that $\frac{g(b)}{G'(b)}\in\N_{0}$ for all $b\in\operatorname{Fix}(\varphi)$. 
Then the first and the second part of (b) imply that there is 
$\widetilde{w}\in\mathcal{H}(\D)$ such that \eqref{eq:semicobound_int_circle} holds. 
Setting $\omega\coloneqq\widetilde{w}\circ h^{-1}$, we see that 
for all $t\geq 0$ and $z\in\Omega\setminus N_{w}$
\[
 \frac{\omega(\varphi_{t}(z))}{\omega(z)}
=\frac{\widetilde{\omega}(\psi_{t}(h^{-1}(z)))}{\widetilde{\omega}(h^{-1}(z))}
=\exp\Bigl(\int_{0}^{t}(g\circ h)(\psi_{s}(h^{-1}(z)))\d s\Bigr)
=\exp\Bigl(\int_{0}^{t}g(\varphi_{s}(z))\d s\Bigr),
\]
which extends to \eqref{eq:semicobound_int} for all $z\in\Omega$ by continuity. 
\end{proof}

If $\Omega\subsetneq\C$ is open and simply connected, and $\varphi$ a jointly continuous holomorphic semiflow on $\Omega$ 
with generator $G\neq 0$, then it follows from \prettyref{prop:semicobound_int} (a) 
and \prettyref{ex:semicoboundary_gen} that 
\begin{equation}\label{eq:space_deriv_semiflow_integral}
\varphi_{t}'(z)=m_{t}^{G}(z)=\exp\Bigl(\int_{0}^{t}G'(\varphi_{s}(z))\d s\Bigr),\quad t\geq 0,\, z\in\Omega,
\end{equation}
which generalises \cite[Proposition 10.1.8 (2), p.~276--277]{bracci2020} where $\Omega=\D$. 

\begin{cor}\label{cor:semicobound_int}
Let $\Omega\subsetneq\C$ be open and simply connected, and $(m,\varphi)$ a jointly continuous holomorphic co-semiflow 
on $\Omega$. Then the following assertions hold.
\begin{enumerate}
\item If $\operatorname{Fix}(\varphi)=\varnothing$, then $m$ is a semicoboundary.
\item Suppose that $\operatorname{Fix}(\varphi)\neq\varnothing$. 
Then there is $\omega\in\mathcal{H}(\Omega)$ with 
$N_{\omega}=\varnothing$ such that $m=m^{\omega}$ if and only if $m_{t}(b)=1$ for all $t\geq 0$ and 
$b\in\operatorname{Fix}(\varphi)$.
\end{enumerate}
\end{cor}
\begin{proof}
By \prettyref{thm:generator_hol_semiflow} the generator $G$ of $\varphi$ exists. 
Due to \prettyref{prop:cocycle_int_hol} there is $g\in\mathcal{H}(\Omega)$ 
such that $m_{t}(z)=\exp(\int_{0}^{t}g(\varphi_{s}(z))\d s)$ for all $t\geq 0$ and $z\in\Omega$. 

(a) We have $G\neq 0$ by \prettyref{prop:fixed_points_holomorphic} (a). 
Since $\operatorname{Fix}(\varphi)=\varnothing$, it follows from \prettyref{prop:semicobound_int} (b) 
that $m$ is a semicoboundary. 

(b) First, suppose that $\varphi$ is trivial. Then $\operatorname{Fix}(\varphi)=\Omega$ and $m^{\omega}=\mathds{1}$ for any 
$\omega\in\mathcal{H}(\Omega)$ with $N_{\omega}=\varnothing$. Thus $m=m^{\omega}$ if and only if $m=\mathds{1}$.

Second, let us consider the case that $\varphi$ is non-trivial. 
Due to \prettyref{prop:fixed_points_holomorphic} (b) we know that $|\operatorname{Fix}(\varphi)|=1$.
Suppose there is $\omega\in\mathcal{H}(\Omega)$ with $N_{\omega}=\varnothing$ such that $m=m^{\omega}$. 
By \prettyref{prop:semicobound_int} (b) we get that $g(b)/G'(b)\in\N_{0}$ for $b\in\operatorname{Fix}(\varphi)$  
and $g(b)/G'(b)=\operatorname{ord}_{\omega}(b)=0$. Hence $g(b)=0$ and this implies that 
\[
m_{t}(b)=\exp\Bigl(\int_{0}^{t}g(\varphi_{s}(b))\d s\Bigr)=\exp\Bigl(\int_{0}^{t}g(b)\d s\Bigr)=\exp(0)=1
\] 
for all $t\geq 0$.

Conversely, suppose that $m_{t}(b)=1$ for all $t\geq 0$. By differentiating w.r.t.~$t$ we obtain $\dt{m}_{t}(b)=0$ for all 
$t\geq 0$. From \prettyref{prop:jointly_cont_int_cocycle} (a) we deduce that 
$g(b)=\dt{m}_{0}(b)=0$. We conclude that $g(b)/G'(b)=0\in\N_{0}$ and so there is 
$\omega\in\mathcal{H}(\Omega)$ with $N_{\omega}=\varnothing$ such that $m=m^{\omega}$ by \prettyref{prop:semicobound_int} (b).
\end{proof}

\prettyref{cor:semicobound_int} is already known due \cite[Theorem 5, p.~3393]{jafari2005} (here $\Omega=\C$ is also allowed). 
However, the proof is different. 

\section{Semigroups of weighted composition operators}
\label{sect:sg_weighted_comp}

Before introducing weighted composition semigroups induced by a co-semiflow $(m,\varphi)$, 
we start this section with weighted composition families induced by a tuple 
$(m,\varphi)$ which need not be a co-semiflow. First, we generalise a part of \cite[Proposition 1, p.~307]{hornor2003}.

\begin{prop}\label{prop:mixed_equicont}
Let $\Omega$ be a Hausdorff space and $(\F,\|\cdot\|,\tau_{\operatorname{co}})$ a Saks space such that 
$\F\subset\mathcal{C}(\Omega)$. Let $I$ be a set, $\varphi\coloneqq (\varphi_{t})_{t\in I}$ and $m\coloneqq (m_{t})_{t\in I}$ 
be families of functions $\varphi_{t}\colon\Omega\to\Omega$ and $m_{t}\colon\Omega\to\K$ such that 
\begin{enumerate}
\item[(i)] $C_{m,\varphi}(t)f\coloneqq m_{t}\cdot(f\circ \varphi_{t})\in\F$ for all $t\in I$ and $f\in\F$, and 
\item[(ii)] $\varphi_{I}(K)\coloneqq \bigcup_{t\in I}\varphi_{t}(K)$ is relatively compact in $\Omega$ 
and $m_{I}(K)\coloneqq\bigcup_{t\in I}m_{t}(K)$ is bounded in $\K$ for all compact $K\subset\Omega$.
\end{enumerate}
Then $(C_{m,\varphi}(t))_{t\in I}$ is $\tau_{\operatorname{co}}$-equicontinuous. 
If in addition $\sup_{t\in I}\|C_{m,\varphi}(t)\|_{\mathcal{L}(\F)}<\infty$, then $(C_{m,\varphi}(t))_{t\in I}$ 
is $\gamma$-equicontinuous. 
\end{prop}
\begin{proof}
First, for compact $K\subset\Omega$ we note that 
\[
 \sup_{x\in K}|C_{m,\varphi}(t)f(x)|
=\sup_{x\in K}|m_{t}(x)f(\varphi_{t}(x))|
\leq\sup_{z\in m_{I}(K)}|z|\sup_{x\in \overline{\varphi_{I}(K)}}|f(x)|
\]
for all $t\in I$ and $f\in\F$, implying that the family $(C_{m,\varphi}(t))_{t\in I}$ of linear maps $\F\to\F$ by condition (i) is 
$\tau_{\operatorname{co}}$-equicontinuous on the whole space $\F$ 
by condition (ii) and the continuity of the functions in $\F$. 

Now, suppose that $\sup_{t\in I}\|C_{m,\varphi}(t)\|_{\mathcal{L}(\F)}<\infty$. 
Since $\tau_{\operatorname{co}}$ is a coarser topology than $\gamma$, we obtain from the first part that 
the family $(C_{m,\varphi}(t))_{t\in I}$ is $\gamma$-$\tau_{\operatorname{co}}$-equicontinuous. 
It follows from \cite[3.16 Proposition, (f)$\Leftrightarrow$(g), p.~12--13]{kruse_schwenninger2022} that 
$(C_{m,\varphi}(t))_{t\in I}$ is even $\gamma$-equicontinuous.
\end{proof}

\begin{rem}\label{rem:mixed_equicont}
Let $\Omega$ be a Hausdorff space, $I$ a compact Hausdorff space, $\varphi\coloneqq (\varphi_{t})_{t\in I}$ 
and $m\coloneqq (m_{t})_{t\in I}$ families of functions $\varphi_{t}\colon\Omega\to\Omega$ and $m_{t}\colon\Omega\to\K$. 
If $\varphi$ and $m$ are both jointly continuous, then condition (ii) of \prettyref{prop:mixed_equicont} is fulfilled 
since $\varphi_{I}(K)$ and $m_{I}(K)$ are compact for all compact $K\subset\Omega$.
\end{rem}

Our next goal is to derive necessary and sufficient conditions for the weighted composition family 
$(C_{m,\varphi}(t))_{t\in I}$ to be $\gamma$-strongly continuous. For the necessary condition we need the following definition. 

\begin{defn}
Let $\Omega$ and $I$ be Hausdorff spaces, $\F\subset\mathcal{C}(\Omega)$ a linear space, 
and $\varphi\coloneqq (\varphi_{t})_{t\in I}$ a family of functions 
$\varphi_{t}\colon\Omega\to\Omega$. We say that the topology of $\Omega$ is 
\emph{initial-like w.r.t.~}$(\varphi,\F)$ if for every compact set $K\subset\Omega$ 
the continuity of the map $I\times K\to \K$, $(t,x)\mapsto f(\varphi_{t}(x))$, for all $f\in\F$ 
implies the continuity of the map $I\times K\to \Omega$, $(t,x)\mapsto\varphi_{t}(x)$.
\end{defn}

\begin{rem}\fakephantomsection\label{rem:initial_like}
\begin{enumerate}
\item Obviously, if $\Omega$ carries the initial topology induced by $\F$, then the topology of 
$\Omega$ is initial-like w.r.t.~$(\varphi,\F)$ for any family of functions 
$\varphi\coloneqq (\varphi_{t})_{t\in I}$ with $\varphi_{t}\colon\Omega\to\Omega$. 
For instance, a completely regular Hausdorff space $\Omega$ carries the initial topology induced 
by the space $\mathcal{C}_{b}(\Omega)$ of bounded continuous functions on $\Omega$ 
(see \cite[2.55 Theorem, p.~49]{aliprantis2006} and \cite[2.56 Corollary, p.~50]{aliprantis2006}). 
\item If $\Omega\subset\K$ and the identity $\operatorname{id}\colon\Omega\to\Omega$, $x\mapsto x$, 
belongs to $\F$, then the topology of $\Omega$ is initial-like w.r.t.~$(\varphi,\F)$ 
for any family of functions $\varphi\coloneqq (\varphi_{t})_{t\in I}$ with 
$\varphi_{t}\colon\Omega\to\Omega$.
\end{enumerate}
\end{rem}

Now, we use the ideas of the proofs of \cite[Proposition 2.10, p.~5]{farkas2020} (see also \cite[Theorem 4.5, p.~51--52]{sentilles1970}) and 
\cite[Corollary 4.3 (a), (b), p.~20]{kruse_schwenninger2022} where $\F=\mathcal{C}_{b}(\Omega)$ is the space of 
bounded continuous functions on a completely regular Hausdorff space $\Omega$, $I=[0,\infty)$ and $m=\mathds{1}$.

\begin{prop}\label{prop:strongly_mixed_cont}
Let $\Omega$ be a Hausdorff space and $(\F,\|\cdot\|,\tau_{\operatorname{co}})$ a Saks space such that 
$\F\subset\mathcal{C}(\Omega)$. Let $I$ be a metric space, $\varphi\coloneqq (\varphi_{t})_{t\in I}$ and 
$m\coloneqq (m_{t})_{t\in I}$ be families of continuous functions $\varphi_{t}\colon\Omega\to\Omega$ 
and $m_{t}\colon\Omega\to\K$ such that $C_{m,\varphi}(t)f\coloneqq m_{t}\cdot(f\circ \varphi_{t})\in\F$ for all $t\in I$ 
and $f\in\F$. 
Further, suppose that $\sup_{t\in I}\|C_{m,\varphi}(t)\|_{\mathcal{L}(\F)}<\infty$. Then the following assertions hold.
\begin{enumerate}
\item If $\varphi$ and $m$ are jointly continuous, then $(C_{m,\varphi}(t))_{t\in I}$ is 
$\gamma$-strongly continuous. 
\item If $\Omega$ is a $k_{\R}$-space 
whose topology is initial-like w.r.t.~$(\varphi,\F)$, $\mathds{1}\in\F$, $I$ is locally compact, 
$m_{t}(x)\neq 0$ for all $(t,x)\in I\times\Omega$ 
and $(C_{m,\varphi}(t))_{t\in I}$ is $\gamma$-strongly continuous, 
then $\varphi$ and $m$ are jointly continuous.
\end{enumerate}
\end{prop}
\begin{proof}
First, we observe that $C_{m,\varphi}(t)$ is linear and $\gamma$-continuous, thus $C_{m,\varphi}(t)\in\mathcal{L}(\F,\gamma)$, 
for every $t\in I$ due to \prettyref{prop:mixed_equicont} and \prettyref{rem:mixed_equicont} applied to the singleton 
$I_{t}\coloneqq\{t\}$ and the continuity of $\varphi_{t}$ and $m_{t}$. 

Since $I$ is a metric space and $\F\subset\mathcal{C}(\Omega)$, the family $(C_{m,\varphi}(t))_{t\in I}$ is 
$\gamma$-strongly continuous if and only if the map
\[
I\to \mathcal{C}(K),\;t\mapsto C_{m,\varphi}(t)f_{\mid K},
\]
is continuous for every compact $K\subset\Omega$ and $f\in\F$ by \cite[I.1.10 Proposition, p.~9]{cooper1978} 
and the assumption $\sup_{t\in I}\|C_{m,\varphi}(t)\|_{\mathcal{L}(\F)}<\infty$. 
It follows from \cite[Lemma 4.16, p.~56]{eisner2015} that this is equivalent to the continuity of the map 
\begin{equation}\label{eq:strongly_mixed_cont_1}
I\times K\to \K,\;(t,x)\mapsto  m_{t}(x)f(\varphi_{t}(x)),
\end{equation}
for every compact $K\subset\Omega$ and $f\in\F$.

(a) If $\varphi$ and $m$ are jointly continuous, then the map \eqref{eq:strongly_mixed_cont_1} is clearly continuous 
for every compact $K\subset\Omega$ and $f\in\F$. 

(b) Since $\mathds{1}\in\F$, the continuity of the map \eqref{eq:strongly_mixed_cont_1} 
implies the continuity of the map 
\begin{equation}\label{eq:strongly_mixed_cont_2}
I\times K\to \K,\;(t,x)\mapsto  m_{t}(x),
\end{equation}
for every compact $K\subset\Omega$. The continuity of the maps 
\eqref{eq:strongly_mixed_cont_1} and \eqref{eq:strongly_mixed_cont_2}, that $m_{t}(x)\neq 0$ 
for all $(t,x)\in I\times\Omega$ and that the topology of $\Omega$ is initial-like 
w.r.t.~$(\varphi,\F)$ yield the continuity of the map 
\begin{equation}\label{eq:strongly_mixed_cont_3}
I\times K\to \K,\;(t,x)\mapsto  \varphi_{t}(x),
\end{equation}
for every compact $K\subset\Omega$. Conversely, the continuity of the maps \eqref{eq:strongly_mixed_cont_2} 
and \eqref{eq:strongly_mixed_cont_3} clearly implies the continuity of the map \eqref{eq:strongly_mixed_cont_1} 
for every compact $K\subset\Omega$ and $f\in\F$. 
Hence the continuity of the map \eqref{eq:strongly_mixed_cont_1} for every compact $K\subset\Omega$ and $f\in\F$ is equivalent 
to the continuity of the maps \eqref{eq:strongly_mixed_cont_2} and \eqref{eq:strongly_mixed_cont_3} 
for every compact $K\subset\Omega$. Now, if $I$ is locally compact and $\Omega$ a $k_{\R}$-space, then 
$I\times\Omega$ is also a $k_{\R}$-space by a comment after the proof of \cite[Th\'eor\`{e}me (2.1), p.~54--55]{buchwalter1972}. 
Thus the $\gamma$-strong continuity of $(C_{m,\varphi}(t))_{t\in I}$ 
implies the continuity of the maps \eqref{eq:strongly_mixed_cont_2} and \eqref{eq:strongly_mixed_cont_3} 
for every compact $K\subset\Omega$, which then implies the joint continuity of $\varphi$ and $m$ because 
$I\times\Omega$ is a $k_{\R}$-space.  
\end{proof}

\begin{rem}
Looking at the proof, we see that we can drop the condition that $\mathds{1}\in\F$ in \prettyref{prop:strongly_mixed_cont} (b) if 
$m_{t}=\mathds{1}$ for all $t\in I$. 
\end{rem}

From now on we restrict to the case that $(m,\varphi)$ is a co-semiflow on a Hausdorff space $\Omega$. 
If $\F\subset\mathcal{C}(\Omega)$ is a linear space and 
$C_{m,\varphi}(t)f\coloneqq m_{t}\cdot(f\circ \varphi_{t})\in\F$ for all $t\geq 0$ and $f\in\F$, 
then a simple computation shows that 
$(C_{m,\varphi}(t))_{t\geq 0}$ is a semigroup of linear operators on $\F$, 
i.e.~$C_{m,\varphi}(t)\colon\F\to\F$ is linear and $C_{m,\varphi}(t+s)=C_{m,\varphi}(t)C_{m,\varphi}(s)$ 
for all $t,s\geq 0$.

\begin{defn}
Let $(m,\varphi)$ be a co-semiflow on a Hausdorff space $\Omega$ and $\F\subset\mathcal{C}(\Omega)$ a linear space. 
The tuple $(m,\varphi)$ is called a \emph{co-semiflow for} $\F$ if 
$C_{m,\varphi}(t)f\coloneqq m_{t}\cdot(f\circ \varphi_{t})\in\F$ for all $t\geq 0$ and $f\in\F$. 
In this case $(C_{m,\varphi}(t))_{t\geq 0}$ is called the \emph{weighted composition semigroup} 
on $\F$ w.r.t.~the co-semiflow $(m,\varphi)$.
\end{defn}

\begin{rem}\label{rem:co_semiflow_for_space}
Let $\Omega$ be a Hausdorff space, $\F\subset\mathcal{C}(\Omega)$ a linear space and $(m,\varphi)$ a co-semiflow for $\F$.
\begin{enumerate}
\item If $\mathds{1}\in\F$, then $m_{t}=C_{m,\varphi}(t)\mathds{1}\in\F$ for all $t\geq 0$. 
\item If $\id\in\F$, then $m_{t}\varphi_{t}=C_{m,\varphi}(t)\id\in\F$ for all $t\geq 0$. 
\end{enumerate} 
\end{rem}

If the semicocycle is actually a semicoboundary, then the weighted composition semigroup may have a quite simple structure.

\begin{rem}\label{rem:semicoboundary_similar}
Let $\Omega$ be a Hausdorff space, $\F\subset\mathcal{C}(\Omega)$ a linear space and $(m,\varphi)$ a co-semiflow for $\F$ 
such that there is $\omega\in\mathcal{C}(\Omega)$ with $N_{\omega}=\varnothing$ and $m=m^{\omega}$. 
Then a direct calculation shows that 
\[
C_{m,\varphi}(t)=M_{\frac{1}{\omega}}C_{\mathds{1},\varphi}(t)M_{\omega}
\]
for all $t\geq 0$ where $M_{\omega}f\coloneqq \omega f$ and $M_{\frac{1}{\omega}}f\coloneqq \frac{1}{\omega}f$ for all $f\in\F$. 
This means that $C_{m,\varphi}(t)$ and $C_{\mathds{1},\varphi}(t)$ are similar as linear operators 
on $\F$ for all $t\geq 0$ if $(\mathds{1},\varphi)$ is also a co-semiflow for $\F$ 
(cf.~\cite[p.~67]{gutierrez2023} in the case $\Omega=\D$ and $\mathcal{F}(\D)$ being a space of holomorphic functions). 
\end{rem}

Sufficient and necessary conditions for the existence of $\omega$ in \prettyref{rem:semicoboundary_similar} are given 
in \prettyref{cor:semicobound_int} in the case that $(m,\varphi)$ is a jointly continuous holomorphic co-semiflow 
on a proper open simply connected subset $\Omega$ of $\C$.

The question which tuples $(m,\varphi)$ are co-semiflows for a given space $\F$ is difficult on its own. 
We recall the following results deduced from characterisations of weighted composition operators.

\begin{rem}\label{rem:co_semiflows_for_example_spaces}
Let $(m,\varphi)$ be a holomorphic co-semiflow on $\D$. 
\begin{enumerate}
\item For $1\leq p<\infty$ there is a necessary and sufficient condition 
in terms of Carleson measures so that $(m,\varphi)$ is a co-semiflow for $H^{p}$ and $C_{m,\varphi}(t)\in\mathcal{L}(H^{p})$ 
for all $t\geq 0$ by \cite[Theorem 2.2, p.~227]{contreras2001}. 
\item For $\alpha>-1$ and $1\leq p<\infty$ there is a 
necessary and sufficient condition in terms of weighted Berezin transforms so that $(m,\varphi)$ is a co-semiflow for 
$A_{\alpha}^{p}$ and $C_{m,\varphi}(t)\in\mathcal{L}(A_{\alpha}^{p})$ for all $t\geq 0$ 
by \cite[Proposition 2, p.~504]{cuckovic2004}.
\item For the Dirichlet space we have that $(m,\varphi)$ is a co-semiflow for $\mathcal{D}$ and 
$C_{m,\varphi}(t)\in\mathcal{L}(\mathcal{D})$ for all $t\geq 0$ if and only if $m_{t}\in\mathcal{D}$, $(m\varphi',\varphi)$ is a 
co-semiflow for $H^{2}$ and $C_{m\varphi',\varphi}(t)\in\mathcal{L}(H^{2})$ for all $t\geq 0$ by 
\cite[Theorem 2.1, p.~175]{contreras2004} with $p=q=2$.
\item For $\alpha>0$ there is a neccessary and sufficient condition such that $(m,\varphi)$ is a co-semiflow 
for $\mathcal{B}_{\alpha}$ and $C_{m,\varphi}(t)\in\mathcal{L}(\mathcal{B}_{\alpha})$ for all $t\geq 0$ given 
in \cite[Theorem 2.1, p.~193]{ohno2003}.
\item Let $v_{0}\colon [0,1]\to[0,\infty)$ be a continuous non-increasing function such that $v_{0}(x)\neq 0$ 
for all $x\in [0,1)$, $v_{0}(1)=0$ and set $v\colon \overline{\D}\to [0,\infty)$, 
$v(z)\coloneqq v_{0}(|z|)$ (see \cite[p.~873]{montes2000}). Then there is a necessary and sufficient condition such that 
$(m,\varphi)$ is a co-semiflow for $\mathcal{H}v(\D)$ and $C_{m,\varphi}(t)\in\mathcal{L}(\mathcal{H}v(\D))$ for all $t\geq 0$ 
given in \cite[Theorem 2.1, p.~875]{montes2000}. 
\end{enumerate}
\end{rem}

Using \prettyref{prop:mixed_equicont} and \prettyref{prop:strongly_mixed_cont} (a), we obtain the following generalisation 
of \cite[Proposition 2.10, p.~5]{farkas2020} and \cite[Corollary 4.3, p.~20]{kruse_schwenninger2022}. 

\begin{thm}\label{thm:sg_weighted_comp_str_cont_equicont}
Let $\Omega$ be a Hausdorff space, $(\F,\|\cdot\|,\tau_{\operatorname{co}})$ a Saks space such that 
$\F\subset\mathcal{C}(\Omega)$ and $(C_{m,\varphi}(t))_{t\geq 0}$ a locally bounded weighted composition semigroup on 
$\F$ w.r.t.~a jointly continuous co-semiflow $(m,\varphi)$. Then the following assertions hold.
\begin{enumerate}
\item $(C_{m,\varphi}(t))_{t\geq 0}$ is $\gamma$-strongly continuous, 
locally $\tau_{\operatorname{co}}$-equicontinuous and locally $\gamma$-equicontinuous.
\item If $(\F,\|\cdot\|,\tau_{\operatorname{co}})$ is a sequentially complete Saks space, 
then $(C_{m,\varphi}(t))_{t\geq 0}$ is a $\tau_{\operatorname{co}}$-bi-continuous semigroup on $\F$. 
\item If $(\F,\|\cdot\|,\tau_{\operatorname{co}})$ is a sequentially complete C-sequential Saks space, 
then $(C_{m,\varphi}(t))_{t\geq 0}$ is quasi-$\gamma$-equicontinuous.
\item If $(\F,\|\cdot\|,\tau_{\operatorname{co}})$ is a sequentially complete C-sequential Saks space 
and $\gamma=\gamma_{s}$, then $(C_{m,\varphi}(t))_{t\geq 0}$ is quasi-$(\|\cdot\|,\tau_{\operatorname{co}})$-equitight.
\end{enumerate}
\end{thm}
\begin{proof}
(a) The $\gamma$-strong continuity follows from \prettyref{prop:strongly_mixed_cont} (a) and the local boundedness 
with $I\coloneqq [0,t_{0}]$ for every $t_{0}\geq 0$.
The local $\tau_{\operatorname{co}}$-equicontinuity and local $\gamma$-equicontinuity
are a consequence of \prettyref{prop:mixed_equicont}, \prettyref{rem:mixed_equicont} and the local boundedness 
with $I\coloneqq [0,t_{0}]$ for every $t_{0}\geq 0$. 

(b)+(c)+(d) The remaining parts follow from \prettyref{thm:mixed_equi_sg_bi_cont} and the comments before it.
\end{proof}

We note that $(\mathcal{C}_{b}(\R),\|\cdot\|_{\infty},\tau_{\operatorname{co}})$ 
is a sequentially complete C-sequential Saks space, $\gamma=\gamma_{s}$ and the left translation semigroup 
on $\mathcal{C}_{b}(\R)$, 
i.e.~the (un)weighted composition semigroup $(C_{\mathds{1},\varphi}(t))_{t\geq 0}$ w.r.t.~the jointly continuous 
co-semiflow $(\mathds{1},\varphi)$ where $\varphi_{t}(x)\coloneqq t+x$ 
for all $t\geq 0$, $x\in\R$, is exponentially bounded, locally $\tau_{\operatorname{co}}$-equicontinuous, 
quasi-$\gamma$-equicontinuous and quasi-$(\|\cdot\|,\tau_{\operatorname{co}})$-equitight by (the proof of) 
\cite[Theorem 4.1, p.~19]{kruse_schwenninger2022} and \cite[Example 4.2 (a), p.~19]{kruse_schwenninger2022} 
(or by \prettyref{thm:sg_weighted_comp_str_cont_equicont}), 
but not quasi-$\tau_{\operatorname{co}}$-equicontinuous by \cite[Example 3.2, p.~549]{kunze2009}. 
This shows that \prettyref{thm:sg_weighted_comp_str_cont_equicont} is sharp in the sense 
that we cannot expect quasi-$\tau_{\operatorname{co}}$-equicontinuity of weighted composition semigroups in general.

Looking at \prettyref{thm:sg_weighted_comp_str_cont_equicont}, we see that the local boundedness of 
$(C_{m,\varphi}(t))_{t\geq 0}$ is a crucial ingredient. 
The rest of this section is dedicated to deriving sufficient conditions on $(m,\varphi)$ and $\F$ 
such that $(C_{m,\varphi}(t))_{t\geq 0}$ becomes locally bounded. Our strategy can be described as follows. 
Let $\Omega$ be a Hausdorff space, $(\F,\|\cdot\|,\tau_{\operatorname{co}})$ a Saks space such that 
$\F\subset\mathcal{C}(\Omega)$ and $(m,\varphi)$ a co-semiflow on $\Omega$. 
We decompose $C_{m,\varphi}=C_{m,\id}C_{\mathds{1},\varphi}$ and see that $(m,\varphi)$ is a co-semiflow for $\F$ 
if $(m,\id)$ and $(\mathds{1},\varphi)$ are co-semiflows for $\F$. In this case we have the estimate
\[
    \|C_{m,\varphi}(t)\|_{\mathcal{L}(\F)}
\leq\|C_{m,\id}(t)\|_{\mathcal{L}(\F)}\|C_{\mathds{1},\varphi}(t)\|_{\mathcal{L}(\F)}
\]
for all $t\geq 0$. Therefore the semigroup $(C_{m,\varphi}(t))_{t\geq 0}$ is locally bounded 
if $(m,\id)$ and $(\mathds{1},\varphi)$ are co-semiflows for $\F$, and the \emph{multiplication semigroup} $(C_{m,\id}(t))_{t\geq 0}$ as well as 
the \emph{unweighted composition semigroup} $(C_{\mathds{1},\varphi}(t))_{t\geq 0}$ are locally bounded. 
This strategy might not be optimal but gives rather simple, more applicable, sufficient conditions 
that guarantee the local boundedness of $(C_{m,\varphi}(t))_{t\geq 0}$.

\begin{rem}\fakephantomsection\label{rem:unweighted_comp_co_semiflows_for_example_spaces}
\begin{enumerate}
\item If $\varphi$ is a holomorphic semiflow on $\D$ and $1\leq p<\infty$, then $(\mathds{1},\varphi)$ is a co-semiflow for 
$H^{p}$ and $C_{\mathds{1},\varphi}(t)\in\mathcal{L}(H^{p})$ for all $t\geq 0$ by \cite[Corollary, p.~29]{duren1970}. 
\item If $\varphi$ is a holomorphic semiflow on $\D$, $\alpha>-1$ and $1\leq p<\infty$, then $(\mathds{1},\varphi)$ 
is a co-semiflow for $A_{\alpha}^{p}$ and $C_{\mathds{1},\varphi}(t)\in\mathcal{L}(A_{\alpha}^{p})$ for all $t\geq 0$ 
by \cite[3.4 Proposition, p.~884]{maccluer1986}.  
\item If $\varphi$ is a holomorphic semiflow on $\D$, then $(\mathds{1},\varphi)$ is a co-semiflow for $\mathcal{D}$ and 
$C_{\mathds{1},\varphi}(t)\in\mathcal{L}(\mathcal{D})$ for all $t\geq 0$ by \cite[p.~166]{siskakis1996}. 
\item Let $\varphi$ be a holomorphic semiflow on $\D$ and $v$ defined as in \prettyref{rem:co_semiflows_for_example_spaces} (e). 
Then there is a necessary and sufficient condition such that $(\mathds{1},\varphi)$ is a co-semiflow for $\mathcal{B}v(\D)$ 
and $C_{\mathds{1},\varphi}(t)\in\mathcal{L}(\mathcal{B}v(\D))$ for all $t\geq 0$ given 
in \cite[Theorem 2.3, p.~876]{montes2000}.
\item Let $\Omega$ be a completely regular Hausdorff space, $v\colon\Omega\to (0,\infty)$ continuous and $\varphi$ 
a semiflow on $\Omega$. Suppose that for every $x\in\Omega$ there is $f\in\mathcal{C}(\Omega)$ such that 
$vf\in\mathcal{C}_{0}(\Omega)$ and $f(x)\neq 0$. This condition is for instance fulfilled if $\Omega$ is locally compact. 
Then $(\mathds{1},\varphi)$ is a co-semiflow for $\mathcal{C}v(\Omega)$ and 
$C_{\mathds{1},\varphi}(t)\in\mathcal{L}(\mathcal{C}v(\Omega))$ for all $t\geq 0$ if and only if for every $t\geq 0$ there is 
$K_{t}\geq 0$ such that $v(x)\leq K_{t}v(\varphi_{t}(x))$ for all $x\in\Omega$ by \cite[2.2 Theorem, p.~307]{singh1988}. 
\item Let $\Omega\subset\C$ be open, $v\colon\Omega\to (0,\infty)$ continuous and and $\varphi$ 
a holomorphic semiflow on $\Omega$. Then $(\mathds{1},\varphi)$ is a co-semiflow for $\mathcal{H}v(\Omega)$ and 
$C_{\mathds{1},\varphi}(t)\in\mathcal{L}(\mathcal{H}v(\Omega))$ for all $t\geq 0$ if for every $t\geq 0$ there is 
$K_{t}\geq 0$ such that $v(z)\leq K_{t}v(\varphi_{t}(z))$ for all $z\in\Omega$ by \cite[2.2 Theorem, p.~307]{singh1988}.
\end{enumerate}
\end{rem}

Let $\Omega$ be a Hausdorff space and $\F\subset\mathcal{C}(\Omega)$ a linear space. 
We define the \emph{multiplier space} $\mathcal{M}(\F)$ of continuous multipliers of $\F$ by 
$\mathcal{M}(\F)\coloneqq\{g\in\mathcal{C}(\Omega)\;|\;\forall\;f\in\F:\;gf\in\F\}$.\footnote{The multiplier space is defined as 
$\mathcal{M}_{0}(\F)\coloneqq\{g\colon\Omega\to\K\;|\;\forall\;f\in\F:\;gf\in\F\}$ in \cite[p.~272]{brown1984}. 
Thus we have $\mathcal{M}(\F)\subset\mathcal{M}_{0}(\F)$. If $\mathds{1}\in\F$, then we even have 
$\mathcal{M}(\F)=\mathcal{M}_{0}(\F)$ because $\F\subset\mathcal{C}(\Omega)$. We incorporated continuity in the definition 
of our multiplier space since our semicocycles are by definition continuous.}

\begin{prop}\label{prop:multiplier_space}
Let $\Omega$ be a Hausdorff space, $(\F,\|\cdot\|,\tau_{\operatorname{co}})$ a sequentially complete Saks space 
such that $\F\subset\mathcal{C}(\Omega)$ and suppose that for every $x\in\Omega$ there is $f\in\F$ such that $f(x)\neq 0$. 
Then we have $\mathcal{M}(\F)\subset \mathcal{C}_{b}(\Omega)$. Further, $C_{m,\id}(t)\in\mathcal{L}(\F)$ for all $t\geq 0$ 
if $(m,\id)$ is a co-semiflow on $\Omega$ and $m_{t}\in\mathcal{M}(\F)$ for all $t\geq 0$. 
\end{prop}
\begin{proof}
Due to our assumption, \prettyref{rem:seq_complete_mixed_Banach} and \prettyref{conv:pre_Saks_function_space} $(\F,\|\cdot\|)$ is 
a functional Banach space in the sense of \cite[p.~57]{duren1969} and thus our statement follows from 
\cite[Lemma 11, p.~57]{duren1969}. 
\end{proof}

\begin{prop}\label{prop:multiplier_space_exa}
\begin{enumerate}
\item $\mathcal{M}(H^{p})=\mathcal{M}(A_{\alpha}^{p})=H^{\infty}$ for $\alpha>-1$ and $1\leq p<\infty$.
\item $\mathcal{M}(\mathcal{B}_{\alpha})=H^{\infty}$ for $\alpha>1$, $\mathcal{M}(\mathcal{B}_{\alpha})=\mathcal{B}_{\alpha}$ for 
$0<\alpha<1$ and 
$\mathcal{M}(\mathcal{B}_{1})=\{f\in H^{\infty}\;|\;\sup_{z\in\D}|f'(z)|(1-|z|^2)\ln((1-|z|^{2})^{-1})<\infty\}$.
\item $\mathcal{M}(\mathcal{C}v(\Omega))=\mathcal{C}_{b}(\Omega)$ for all locally compact Hausdorff spaces $\Omega$ 
and continuous $v\colon\Omega\to (0,\infty)$.
\item $\mathcal{M}(\mathcal{H}v(\Omega))=H^{\infty}(\Omega)$ for all open sets $\Omega\subset\C$ and continuous 
$v\colon\Omega\to (0,\infty)$ if $\mathds{1}\in\mathcal{H}v(\Omega)$.
\end{enumerate}
\end{prop}
\begin{proof}
In (a) we have $H^{\infty}\subset \mathcal{M}(H^{p})$ and $H^{\infty}\subset \mathcal{M}(A_{\alpha}^{p})$, in (c) we have 
$\mathcal{C}_{b}(\Omega)\subset\mathcal{M}(\mathcal{C}v(\Omega))$, and in (d) we have 
$H^{\infty}(\Omega)\subset\mathcal{M}(\mathcal{H}v(\Omega))$. Therefore the statements in part (a), (c) and (d) follow 
from \prettyref{prop:multiplier_space}, \prettyref{ex:hardy_bergman_dirichlet}, \prettyref{ex:cont_saks} and 
\prettyref{ex:weighted_holom_saks} since in (a) $\mathds{1}\in H^{p}, A_{\alpha}^{p}$ and since in (c) 
for every $x\in\Omega$ there is $f\in\F$ such that $f(x)\neq 0$ because $\Omega$ is a locally compact Hausdorff space. 
Part (b) is \cite[Theorem 27, p.~1170]{zhu1993}.
\end{proof}

The multiplier space $\mathcal{M}(\mathcal{D})$ of the Dirichlet space is more complicated and its elements can be described 
in terms of the Carleson measure by \cite[Theorems 1.1 (c), 2.3, 2.7, p.~115, 122, 125]{stegenga1980} with $\alpha=\frac{1}{2}$. 
From \prettyref{prop:multiplier_space} and $\mathds{1}\in\mathcal{D}$ it follows that 
$\mathcal{M}(\mathcal{D})\subset(\mathcal{D}\cap H^{\infty})$.

\begin{thm}\label{thm:loc_bound_sg_hardy_bergman_dirichlet}
Let $(m,\varphi)$ be a holomorphic co-semiflow on $\D$ such that $\varphi$ is jointly continuous. 
Then $(m,\varphi)$ is a co-semiflow for $\mathcal{F}(\D)$ 
and the weighted composition semigroup $(C_{m,\varphi}(t))_{t\geq 0}$ on $\mathcal{F}(\D)$ is locally bounded 
in each of the following cases if
\begin{enumerate}
\item $\mathcal{F}(\D)=H^{p}$ for $1\leq p<\infty$ and $\limsup_{t\to 0\rlim}\|m_{t}\|_{\infty}<\infty$,
\item $\mathcal{F}(\D)=A_{\alpha}^{p}$ for $\alpha>-1$ and $1\leq p<\infty$ and $\limsup_{t\to 0\rlim}\|m_{t}\|_{\infty}<\infty$,
\item $\mathcal{F}(\D)=\mathcal{D}$ and $m_{t}\in\mathcal{M}(\mathcal{D})$ for all $t\geq 0$.
\end{enumerate}
\end{thm}
\begin{proof}
The condition $\limsup_{t\to 0\rlim}\|m_{t}\|_{\infty}<\infty$ in (a) and (b) yields that $m_{t}\in H^{\infty}$ for all $t\geq 0$ 
by \prettyref{prop:semicocycle_bounded}, which is also true in (c) because $\mathcal{M}(\mathcal{D})\subset H^{\infty}$. 
Hence in all the cases $(\mathds{1},\varphi)$, $(m,\id)$ and $(m,\varphi)$ are co-semiflows for $\mathcal{F}(\D)$ by 
\prettyref{rem:unweighted_comp_co_semiflows_for_example_spaces} and \prettyref{prop:multiplier_space_exa}. 

(a) We have 
\[
\|f\circ\varphi_{t}\|_{p}^{p}\leq \frac{1+|\varphi_{t}(0)|}{1-|\varphi_{t}(0)|}\|f\|_{p}^{p}
\] 
for all $t\geq 0$ and $f\in H^{p}$ by \cite[Corollary, p.~29]{duren1970}, yielding 
\[
\|C_{\mathds{1},\varphi}(t)\|_{\mathcal{L}(H^{p})}\leq \left(\frac{1+|\varphi_{t}(0)|}{1-|\varphi_{t}(0)|}\right)^{\frac{1}{p}}
\]
for all $t\geq 0$. Further, we note that $\|C_{m,\id}(t)\|_{\mathcal{L}(H^{p})}\leq\|m_{t}\|_{\infty}$ for all $t\geq 0$.

(b) Due to \cite[Theorem 8.1.15, p.~211]{bracci2020} and \cite[Proposition 10.1.7 (1)$\Leftrightarrow$(2), p.~275]{bracci2020}  
there is $t_{0}>0$ such that $\|\varphi_{t}-\id\|_{\infty}\leq 1$ for all $t\in[0,t_{0}]$ since $\varphi$ is jointly continuous, 
which implies that $\|\varphi_{t}\|_{\infty}\leq 2$ for all $t\in[0,t_{0}]$. Therefore we obtain
\[
    \|C_{\mathds{1},\varphi}(t)\|_{\mathcal{L}(A_{\alpha}^{p})}
\leq K(\varphi_{t})\left(\frac{\|\varphi_{t}\|_{\infty}+|\varphi_{t}(0)|}{\|\varphi_{t}\|_{\infty}-|\varphi_{t}(0)|}\right)^{\frac{\alpha+2}{p}}
\]
for all $t\in [0,t_{0}]$ by \cite[Lemma 1, p.~399]{siskakis1987} where $K(\varphi_{t})\coloneqq 1$ if $\alpha\geq 0$, and
$K(\varphi_{t})\coloneqq (\|\varphi_{t}\|_{\infty}+|\varphi_{t}(0)|)^{\alpha/p}(\|\varphi_{t}\|_{\infty}+3|\varphi_{t}(0)|)^{-\alpha/p}$ if $-1<\alpha<0$. 
Furthermore, we note that $\|C_{m,\id}(t)\|_{\mathcal{L}(A_{\alpha}^{p})}\leq\|m_{t}\|_{\infty}$ 
for all $t\geq 0$.

(c) Due to \cite[Theorem 2 (a), p.~26]{martin2005} and \prettyref{prop:fixed_points_holomorphic} (a) we have 
\[
\|C_{\mathds{1},\varphi}(t)\|_{\mathcal{L}(\mathcal{D})}^{2}
\leq 1+\frac{1}{2}\Bigl(L(\varphi_{t})+\bigl(L(\varphi_{t})(4+L(\varphi_{t}))\bigr)^{\frac{1}{2}}\Bigr)
\]
for all $t\geq 0$ where $L(\varphi_{t})\coloneqq -\ln(1-|\varphi_{t}(0)|^{2})$. 
By \cite[Proposition 5.1, p.~101]{zhu2007} we have $H^{\infty}\subset\mathcal{B}_{1}$ and 
\begin{equation}\label{eq:Bloch_1_H_infty}
\|g\|_{\mathcal{B}_{1}}\leq\|g\|_{\infty}
\end{equation}
for all $g\in H^{\infty}$. By \cite[p.~114]{stegenga1980} with $\alpha=\frac{1}{2}$ there is $M>0$ such that 
\begin{equation}\label{eq:dirichlet_estim}
\frac{1}{\pi}\int_{\D}|f(z)|^{2}(1-|z|^{2})^{-2}\d z\leq M\|f\|_{\mathcal{D}}^{2}
\end{equation}
for all $f\in\mathcal{D}$. We deduce that
\begin{align*}
  \|C_{m,\id}(t)f\|_{\mathcal{D}}^{2}
&=|m_{t}(0)f(0)|^{2}+\frac{1}{\pi}\int_{\D}|(m_{t}f)'(z)|^{2}\d z\\
&\leq \|m_{t}\|_{\infty}^{2}\|f\|_{\mathcal{D}}^{2}+\frac{2}{\pi}\int_{\D}|m_{t}'(z)f(z)|^{2}+|m_{t}(z)f'(z)|^{2}\d z\\
&\leq 2\|m_{t}\|_{\infty}^{2}\|f\|_{\mathcal{D}}^{2}
 +\frac{2}{\pi}\int_{\D}|m_{t}'(z)|^{2}(1-|z|^{2})^{2}|f(z)|^{2}(1-|z|^{2})^{-2}\d z\\
&\leq 2\|m_{t}\|_{\infty}^{2}\|f\|_{\mathcal{D}}^{2}
 +\frac{2}{\pi}\|m_{t}\|_{\mathcal{B}_{1}}^{2}\int_{\D}|f(z)|^{2}(1-|z|^{2})^{-2}\d z\\
&\underset{\mathclap{\eqref{eq:Bloch_1_H_infty},\,\eqref{eq:dirichlet_estim}}}{\leq}\quad
 2(1+M)\|m_{t}\|_{\infty}^{2}\|f\|_{\mathcal{D}}^{2}
\end{align*}
for all $t\geq 0$ and $f\in\mathcal{D}$, which implies 
$\|C_{m,\id}(t)\|_{\mathcal{L}(\mathcal{D})}^{2}\leq 2(1+M)\|m_{t}\|_{\infty}^{2}$ for all $t\geq 0$.

Therefore we derive in (a), (b) and (c) from the joint continuity of $\varphi$ that $(C_{\mathds{1},\varphi}(t))_{t\geq 0}$ 
is locally bounded. \prettyref{prop:semicocycle_bounded} yields that $(C_{m,\id}(t))_{t\geq 0}$ is locally bounded and thus 
$(C_{m,\varphi}(t))_{t\geq 0}$
is locally bounded, too.
\end{proof}

Any semiflow $\varphi$ on $\D$ given in \cite[p.~4--5]{siskakis1998} is holomorphic and $C_{0}$, thus jointly continuous by 
\prettyref{prop:jointly_cont_semiflow} (a). For any such $\varphi$ take the semicocycle $m$ given by 
$m_{t}\colon\D\to\C$, $m_{t}(z)\coloneqq \exp(\int_{0}^{t}g(\varphi_{s}(z))\d s)$, for all $t\geq 0$ 
for some $g\in\mathcal{H}(\D)$ (see \prettyref{prop:jointly_cont_int_cocycle} (a)). 
If $M\coloneqq\sup_{z\in\D}\re(g(z))<\infty$, then 
$
\|m_{t}\|_{\infty}\leq \euler^{tM}
$
for all $t\geq 0$ and so $\limsup_{t\to 0\rlim}\|m_{t}\|_{\infty}\leq 1$ (cf.~\cite[p.~474]{koenig1990}). 
Hence $(C_{m,\varphi}(t))_{t\geq 0}$ is a locally bounded semigroup on $\mathcal{F}(\D)$ in case (a) and (b) 
of \prettyref{thm:loc_bound_sg_hardy_bergman_dirichlet}. 

An example in case (c) of the Dirichlet space $\mathcal{D}$ is the jointly continuous holomorphic co-semiflow 
$(\varphi',\varphi)$ on $\D$ given by $\varphi_{t}\colon\D\to\D$, $\varphi_{t}(z)\coloneqq \euler^{-ct}z$, for all $t\geq 0$ 
for some $c\in\C$ with $\re(c)\geq 0$ since $\varphi_{t}'(z)=\euler^{-ct}$ for all $t\geq 0$ and $z\in\D$, 
which implies $m_{t}\coloneqq\varphi_{t}'\in\mathcal{D}$ for all $t\geq 0$. 
The same is true if we choose $\varphi_{t}(z)\coloneqq \euler^{-t}z+1-\euler^{-t}$ for all $t\geq 0$ and $z\in\D$. 
Thus $(C_{\varphi',\varphi}(t))_{t\geq 0}$ is a locally bounded semigroup on $\mathcal{D}$ by 
\prettyref{thm:loc_bound_sg_hardy_bergman_dirichlet} (c) in both cases.

\begin{thm}\label{thm:loc_bound_sg_bloch}
Let $(m,\varphi)$ be a holomorphic co-semiflow on $\D$ and $\alpha>0$. Then $(m,\varphi)$ is a co-semiflow for 
$\mathcal{B}_{\alpha}$ and the weighted composition semigroup $(C_{m,\varphi}(t))_{t\geq 0}$ on $\mathcal{B}_{\alpha}$ 
is locally bounded if 
\[
K_{\alpha}(\varphi_{t})\coloneqq\sup_{z\in\D}|\varphi_{t}'(z)|(1-|\varphi_{t}(z)|^{2})^{-\alpha}(1-|z|^{2})^{\alpha}<\infty
\]
for all $t\geq 0$, there exists $t_{0}>0$ such that $\sup_{t\in [0,t_{0}]}K_{\alpha}(\varphi_{t})<\infty$ and 
\begin{enumerate}
\item for $\alpha>1$ if $\limsup_{t\to 0\rlim}\|m_{t}\|_{\infty}<\infty$,
\item for $\alpha=1$ if $\limsup_{t\to 0\rlim}\|m_{t}\|_{\infty}<\infty$ and 
\[
\sup_{t\in [0,t_{0}]}\sup_{z\in\D}|m_{t}'(z)|(1-|z|^{2})\ln((1-|z|^{2})^{-1})<\infty,
\]
\item for $0<\alpha<1$ if $m_{t}\in\mathcal{B}_{\alpha}$ for all $t\geq 0$ and 
$
\sup_{t\in [0,t_{0}]}\|m_{t}\|_{\mathcal{B}_{\alpha}}<\infty.
$
\end{enumerate}
\end{thm}
\begin{proof}
By the proof of \cite[Theorem 2.2 (i), p.~115]{xiao2001} we have 
\[
    \|C_{\mathds{1},\varphi}(t)f\|
\leq\|f\|_{\mathcal{B}_{\alpha}}\sup_{z\in\D}|\varphi_{t}'(z)|(1-|\varphi_{t}(z)|^{2})^{-\alpha}(1-|z|^{2})^{\alpha}
= K_{\alpha}(\varphi_{t})\|f\|_{\mathcal{B}_{\alpha}}
\]
for all $t\geq 0$ and $f\in\mathcal{B}_{\alpha}$, yielding $f\circ\varphi_{t}\in\mathcal{B}_{\alpha}$ and 
that $(\mathds{1},\varphi)$ is a co-semiflow for $\mathcal{B}_{\alpha}$. In addition, 
the existence of $t_{0}>0$ such that $\sup_{t\in [0,t_{0}]}K_{\alpha}(\varphi_{t})<\infty$ gives that 
$(C_{\mathds{1},\varphi}(t))_{t\geq 0}$ is a locally bounded semigroup.

Moreover, the conditions in (a), (b) and (c) guarantee that $m_{t}\in\mathcal{M}(\mathcal{B}_{\alpha})$ for all $t\geq 0$ 
by \prettyref{prop:semicocycle_bounded} and \prettyref{prop:multiplier_space_exa} (b). 
Hence in all the cases $(m,\id)$ and so $(m,\varphi)$ are co-semiflows for $\mathcal{B}_{\alpha}$.

(a) For $\alpha>1$ we have by the proof of \cite[Proposition 7, p.~1147]{zhu1993} that there is $L_{\alpha}>0$ such that
\begin{equation}\label{eq:Bloch_alpha>1}
|f(z)-f(0)|\leq L_{\alpha}\|f\|_{\mathcal{B}_{\alpha}}(1-|z|^{2})^{-(\alpha-1)}
\end{equation}
for all $z\in\D$ and $f\in\mathcal{B}_{\alpha}$. It follows that 
\begin{align*}
    \|m_{t}f\|_{\mathcal{B}_{\alpha}}
&\leq |m_{t}(0)f(0)|+\sup_{z\in\D}|m_{t}'(z)f(z)|(1-|z|^{2})^{\alpha}+\sup_{z\in\D}|m_{t}(z)f'(z)|(1-|z|^{2})^{\alpha}\\
&\leq \sup_{z\in\D}|m_{t}'(z)|(1-|z|^{2})\sup_{z\in\D}|f(z)|(1-|z|^{2})^{\alpha-1}
 +2\|m_{t}\|_{\infty}\|f\|_{\mathcal{B}_{\alpha}}\\
&\leq\|m_{t}\|_{\mathcal{B}_{1}}\sup_{z\in\D}|f(z)|(1-|z|^{2})^{\alpha-1}
 +2\|m_{t}\|_{\infty}\|f\|_{\mathcal{B}_{\alpha}}\\
&\underset{\mathclap{\eqref{eq:Bloch_alpha>1}}}{\leq}\|m_{t}\|_{\mathcal{B}_{1}}
\sup_{z\in\D}\bigl(L_{\alpha}\|f\|_{\mathcal{B}_{\alpha}}+|f(0)|(1-|z|^{2})^{\alpha-1}\bigr)
 +2\|m_{t}\|_{\infty}\|f\|_{\mathcal{B}_{\alpha}}\\
&\underset{\mathclap{\eqref{eq:Bloch_1_H_infty}}}{\leq}\|m_{t}\|_{\infty}
 \bigl(L_{\alpha}\|f\|_{\mathcal{B}_{\alpha}}+|f(0)|\bigr)
 +2\|m_{t}\|_{\infty}\|f\|_{\mathcal{B}_{\alpha}}\\ 
&\leq (3+L_{\alpha})\|m_{t}\|_{\infty}\|f\|_{\mathcal{B}_{\alpha}}
\end{align*}
for all $t\geq 0$ and $f\in\mathcal{B}_{\alpha}$.

(b) For $\alpha=1$ we get as in the proof of \cite[Proposition 7, p.~1147]{zhu1993}, 
using \cite[Proposition 7 3), p.~1146]{zhu1993} with $t=\alpha-1=0$ and $s=0$, that there is $L_{1}>0$ such that
\begin{equation}\label{eq:Bloch_alpha=1}
|f(z)-f(0)|\leq L_{1}\|f\|_{\mathcal{B}_{1}}\ln((1-|z|^{2})^{-1})
\end{equation}
for all $z\in\D$ and $f\in\mathcal{B}_{1}$. It follows that 
\begin{align*}
 \|m_{t}f\|_{\mathcal{B}_{1}}
&\leq \sup_{z\in\D}|m_{t}'(z)|(1-|z|^{2})(|f(z)-f(0)|+|f(0)|)+2\|m_{t}\|_{\infty}\|f\|_{\mathcal{B}_{1}}\\
&\underset{\mathclap{\eqref{eq:Bloch_alpha=1}}}{\leq} L_{1}\|f\|_{\mathcal{B}_{1}}\sup_{z\in\D}|m_{t}'(z)|(1-|z|^{2})\ln((1-|z|^{2})^{-1})+|f(0)|\sup_{z\in\D}|m_{t}'(z)|(1-|z|^{2})\\
&\phantom{\leq} +2\|m_{t}\|_{\infty}\|f\|_{\mathcal{B}_{1}}\\
&\underset{\mathclap{\eqref{eq:Bloch_1_H_infty}}}{\leq} \bigl(3\|m_{t}\|_{\infty}
+L_{1}\sup_{z\in\D}|m_{t}'(z)|(1-|z|^{2})\ln((1-|z|^{2})^{-1})\bigr)\|f\|_{\mathcal{B}_{1}}
\end{align*}
for all $t\geq 0$ and $f\in\mathcal{B}_{1}$.

(c) Again, for $0<\alpha<1$ we get as in the proof of \cite[Proposition 7, p.~1147]{zhu1993}, 
using \cite[Proposition 7 1), p.~1146]{zhu1993} with $t=\alpha-1<0$ and $s=0$, that there is $L_{\alpha}>0$ such that
\begin{equation}\label{eq:Bloch_alpha<1}
|f(z)-f(0)|\leq L_{\alpha}\|f\|_{\mathcal{B}_{\alpha}}
\end{equation}
for all $z\in\D$ and $f\in\mathcal{B}_{\alpha}$. This implies $\mathcal{B}_{\alpha}\subset H^{\infty}$ 
and $\|g\|_{\infty}\leq(1+L_{\alpha})\|g\|_{\mathcal{B}_{\alpha}}$ for all $g\in\mathcal{B}_{\alpha}$ and $0<\alpha<1$. 
It follows that 
\begin{align*}
 \|m_{t}f\|_{\mathcal{B}_{\alpha}}
&\leq \sup_{z\in\D}|m_{t}'(z)|(1-|z|^{2})^{\alpha}(|f(z)-f(0)|+|f(0)|)+2\|m_{t}\|_{\infty}\|f\|_{\mathcal{B}_{\alpha}}\\
&\underset{\mathclap{\eqref{eq:Bloch_alpha<1}}}{\leq} (1+L_{\alpha})\|f\|_{\mathcal{B}_{\alpha}}
 \sup_{z\in\D}|m_{t}'(z)|(1-|z|^{2})^{\alpha}+2\|m_{t}\|_{\infty}\|f\|_{\mathcal{B}_{\alpha}}\\
&\leq \bigl(2\|m_{t}\|_{\infty}+(1+L_{\alpha})\|m_{t}\|_{\mathcal{B}_{\alpha}}\bigr)\|f\|_{\mathcal{B}_{\alpha}}
 \leq 3(1+L_{\alpha})\|m_{t}\|_{\mathcal{B}_{\alpha}}\|f\|_{\mathcal{B}_{\alpha}}
\end{align*}
for all $t\geq 0$ and $f\in\mathcal{B}_{\alpha}$.

Hence our conditions in (a) and (b) combined with \prettyref{prop:semicocycle_bounded}, and in (c) guarantee that 
$(C_{m,\id}(t))_{t\geq 0}$ is a locally bounded semigroup and thus $(C_{m,\varphi}(t))_{t\geq 0}$ as well.
\end{proof}

The jointly continuous holomorphic co-semiflow 
$(\varphi',\varphi)$ on $\D$ given by $\varphi_{t}\colon\D\to\D$, $\varphi_{t}(z)\coloneqq \euler^{-ct}z$, 
for all $t\geq 0$ for some $c\in\C$ with $\re(c)\geq 0$, 
fulfils $m_{t}(z)\coloneqq\varphi_{t}'(z)=\euler^{-ct}$, 
\[
 K_{\alpha}(\varphi_{t})
=\sup_{z\in\D}\euler^{-\re(c)t}(1-\euler^{-2\re(c)t}|z|^{2})^{-\alpha}(1-|z|^{2})^{\alpha}
\leq \euler^{-\re(c)t}
\]
and $\|\varphi_{t}'\|_{\infty}=\euler^{-\re(c)t}$ as well as $m_{t}'(z)=\varphi_{t}''(z)=0$ and 
$\|m_{t}\|_{\mathcal{B}_{\alpha}}=|\varphi_{t}'(0)|=\euler^{-\re(c)t}$ for all $t\geq 0$ and $z\in\D$. 
Thus $(C_{\varphi',\varphi}(t))_{t\geq 0}$ is a locally bounded semigroup on $\mathcal{B}_{\alpha}$ for all $\alpha>0$ 
by \prettyref{thm:loc_bound_sg_bloch}.

Let us turn to the space of bounded Dirichlet series. We say that a holomorphic function $\varphi\colon\C_{+}\to\C_{+}$ 
belongs to the class $\mathscr{G}_{\infty}$ if there exist $c_{\varphi}\in\N_{0}$ and a Dirichlet 
series $\rho_{\varphi}$ which converges on some open half-plane and extends holomorphically to $\C_{+}$ 
such that $\varphi(z)=c_{\varphi}z+\rho_{\varphi}(z)$ for all $z\in\C_{+}$ (see \cite[Definition 2.6, p.~9]{contreras2023}).

\begin{thm}\label{thm:loc_bound_sg_bDirichlet_series}
Let $(m,\varphi)$ be a holomorphic co-semiflow on $\C_{+}$.  
Then $(m,\varphi)$ is a co-semiflow for $\mathscr{H}^{\infty}$ and 
the weighted composition semigroup $(C_{m,\varphi}(t))_{t\geq 0}$ on $\mathscr{H}^{\infty}$ is locally bounded 
if $\varphi_{t}\in\mathscr{G}_{\infty}$ and $m_{t}\in\mathscr{H}^{\infty}$ for all $t\geq 0$. 
The converse is true if $\inf_{z\in\C_{+}}|m_{t}(z)|>0$ 
for all $t\geq 0$. 
\end{thm}
\begin{proof}
By \cite[Proposition 2, p.~219]{bayart2021} $(\mathds{1},\varphi)$ is a co-semiflow for $\mathscr{H}^{\infty}$ 
if and only $\varphi_{t}\in\mathscr{G}_{\infty}$ for all $t\geq 0$. 
Further, we observe that $\|C_{\mathds{1},\varphi}(t)f\|_{\mathscr{H}^{\infty}}\leq \|f\|_{\mathscr{H}^{\infty}}$ 
for all $f\in\mathscr{H}^{\infty}$ if $\varphi_{t}\in\mathscr{G}_{\infty}$ for $t\geq 0$.
Due to \cite[Theorem 7 a), p.~9--10]{vidal2022} we have $\mathcal{M}(\mathscr{H}^{\infty})=\mathscr{H}^{\infty}$, 
which implies that $(m,\id)$ is a co-semiflow for $\mathscr{H}^{\infty}$ if and only if 
$m_{t}\in\mathscr{H}^{\infty}$ for all $t\geq 0$. 
We note that $\|C_{m,\id}(t)f\|_{\mathscr{H}^{\infty}}\leq\|m_{t}\|_{\infty}\|f\|_{\mathscr{H}^{\infty}}$ 
for all $f\in\mathscr{H}^{\infty}$ if $m_{t}\in\mathscr{H}^{\infty}$ for $t\geq 0$.

Now, the first implication follows from our considerations above and \prettyref{prop:semicocycle_bounded} using 
that $\mathscr{H}^{\infty}\subset H^{\infty}(\C_{+})$. 
Let us consider the converse implication. The property $m_{t}\in\mathscr{H}^{\infty}$ for all $t\geq 0$ 
follows from the definition of the weighted composition 
semigroup and the observation that $\mathds{1}\in\mathscr{H}^{\infty}$ since 
$m_{t}=C_{m,\varphi}(t)\mathds{1}$ for all $t\geq 0$. The property $\varphi_{t}\in\mathscr{G}_{\infty}$ for all $t\geq 0$ follows from our first 
observation of the proof, writing $C_{\mathds{1},\varphi}(t)=(1/m_{t})C_{m,\varphi}(t)$ and noting that $1/m_{t}$ 
belongs to the Banach algebra $\mathscr{H}^{\infty}$ by \cite[Theorem 6.2.1, p.~147]{queffelec2020} 
if and only if $\inf_{z\in\C_{+}}|m_{t}(z)|>0$.
\end{proof}

The jointly continuous holomorphic co-semiflow 
$(\mathds{1},\varphi)$ on $\C_{+}$ given by $\varphi_{t}\colon\C_{+}\to\C_{+}$, $\varphi_{t}(z)\coloneqq z+t$, 
for all $t\geq 0$ fulfils with $m_{t}(z)\coloneqq 1$ for all $t\geq 0$ and $z\in\C_{+}$ 
that $\varphi_{t}\in\mathscr{G}_{\infty}$ and $m_{t}\in\mathscr{H}^{\infty}$ for all $t\geq 0$. 
Thus $(C_{\mathds{1},\varphi}(t))_{t\geq 0}$ is a locally bounded semigroup on $\mathscr{H}^{\infty}$ 
by \prettyref{thm:loc_bound_sg_bDirichlet_series}.

\begin{thm}\label{thm:loc_bound_sg_Hv}
Let $\Omega\subset\C$ be open, $v\colon\Omega\to (0,\infty)$ continuous, $\mathds{1}\in\mathcal{H}v(\Omega)$ 
and $(m,\varphi)$ a holomorphic co-semiflow on $\Omega$. 
If $\limsup_{t\to 0\rlim}\|m_{t}\|_{\infty}<\infty$,  
\[
K(\varphi_{t})\coloneqq\sup_{z\in\Omega}\frac{v(z)}{v(\varphi_{t}(z))}<\infty
\]
for all $t\geq 0$ and there exists $t_{0}>0$ such that $\sup_{t\in [0,t_{0}]}K(\varphi_{t})<\infty$, 
then $(m,\varphi)$ is a co-semiflow for $\mathcal{H}v(\Omega)$ and 
the weighted composition semigroup $(C_{m,\varphi}(t))_{t\geq 0}$ on $\mathcal{H}v(\Omega)$ is locally bounded. 
If $v=\mathds{1}$, then the converse holds as well. 
\end{thm}
\begin{proof}
We deduce from \prettyref{rem:unweighted_comp_co_semiflows_for_example_spaces} (f) and \prettyref{prop:multiplier_space_exa} (d) 
that $(m,\id)$, $(\mathds{1},\varphi)$ and so $(m,\varphi)$ are co-semiflows for $\mathcal{H}v(\Omega)$. 
We observe that
\[
 \|C_{\mathds{1},\varphi}(t)f\|_{v}
=\sup_{z\in\Omega}|f(\varphi_{t}(z))|v(\varphi_{t}(z))\frac{v(z)}{v(\varphi_{t}(z))}
\leq K(\varphi_{t})\|f\|_{v}
\]
for all $t\geq 0$ and $f\in\mathcal{H}v(\Omega)$, yielding 
$\|C_{\mathds{1},\varphi}(t)\|_{\mathcal{L}(\mathcal{H}v(\Omega))}\leq K(\varphi_{t})$. 

Moreover, we note that 
\[
 \|C_{m,\id}(t)f\|_{v}
=\sup_{z\in\Omega}|m_{t}(z)f(z)|v(z)
\leq \|m_{t}\|_{\infty}\|f\|_{v}
\]
for all $t\geq 0$ and $f\in\mathcal{H}v(\Omega)$, yielding 
$\|C_{m,\id}(t)\|_{\mathcal{L}(\mathcal{H}v(\Omega))}\leq\|m_{t}\|_{\infty}$. 
Therefore $(C_{m,\id}(t))_{t\geq 0}$ is locally bounded by \prettyref{prop:semicocycle_bounded}. 
The same is true for the semigroup $(C_{\mathds{1},\varphi}(t))_{t\geq 0}$ by the existence of $t_{0}>0$ and so 
$(C_{m,\varphi}(t))_{t\geq 0}$ is locally bounded as well. 

If $v=\mathds{1}$, then the converse holds as well by \prettyref{prop:semicocycle_bounded} 
since then $K(\varphi_{t})=1$ and $\|m_{t}\|_{\infty}=\|m_{t}\|_{v}=\|C_{m,\varphi}(t)\mathds{1}\|_{v}$ for all $t\geq 0$.
\end{proof}

Analogously we obtain the corresponding result for $\mathcal{C}v(\Omega)$ by using 
\prettyref{rem:unweighted_comp_co_semiflows_for_example_spaces} (e) and \prettyref{prop:multiplier_space_exa} (c).

\begin{thm}\label{thm:loc_bound_sg_Cv}
Let $\Omega$ be a locally compact Hausdorff space, $v\colon\Omega\to (0,\infty)$ continuous and $(m,\varphi)$ a co-semiflow 
on $\Omega$. If $\limsup_{t\to 0\rlim}\|m_{t}\|_{\infty}<\infty$,  
\[
K(\varphi_{t})\coloneqq\sup_{x\in\Omega}\frac{v(x)}{v(\varphi_{t}(x))}<\infty
\]
for all $t\geq 0$ and there exists $t_{0}>0$ such that $\sup_{t\in [0,t_{0}]}K(\varphi_{t})<\infty$, 
then $(m,\varphi)$ is a co-semiflow for $\mathcal{C}v(\Omega)$ and 
the weighted composition semigroup $(C_{m,\varphi}(t))_{t\geq 0}$ on $\mathcal{C}v(\Omega)$ is locally bounded. 
If $\mathds{1}\in\mathcal{C}v(\Omega)$ and $v=\mathds{1}$, then the converse holds as well. 
\end{thm}

Apart from the weight $v=\mathds{1}$ there are other weights that fulfil the conditions of \prettyref{thm:loc_bound_sg_Hv} 
and \prettyref{thm:loc_bound_sg_Cv}. For instance, take the left-translation semiflow $\varphi$ given by 
$\varphi_{t}\colon\K\to\K$, $\varphi_{t}(x)\coloneq x+t$, for $t\geq 0$ and set $v\colon\K\to (0,\infty)$, 
$v(x)\coloneqq \euler^{-|x|}$. Then $K(\varphi_{t})\leq \euler^{t}$ 
for all $t\geq 0$ and we can choose any $t_{0}>0$ (cf.~\cite[(i), p.~310]{singh1988}). 
Taking now any (holomorphic) semicocycle $m$ for $\varphi$ such that $\limsup_{t\to 0\rlim}\|m_{t}\|_{\infty}<\infty$, 
we get a locally bounded semigroup $(C_{m,\varphi}(t))_{t\geq 0}$ on $\mathcal{H}v(\C)$ (if $\K=\C$ and 
$\mathds{1}\in\mathcal{H}v(\C)$) resp.~$\mathcal{C}v(\K)$ by \prettyref{thm:loc_bound_sg_Hv} 
resp.~\prettyref{thm:loc_bound_sg_Cv}.

\section{Generators of weighted composition semigroups}
\label{sect:generators}

In this section we give several characterisations of the generator of a $\gamma$-strongly continuous weighted 
composition semigroup on a Saks space.
Let $(X,\|\cdot\|,\tau)$ be a Saks space and $(T(t))_{t\geq 0}$ a $\gamma$-strongly continuous semigroup on $X$. 
We define the \emph{generator} $(A,D(A))$ of $(T(t))_{t\geq 0}$ according to \cite[p.~260]{komura1968} by 
\[
D(A)\coloneqq \Bigl\{x\in X\;|\;\gamma\text{-}\lim_{t\to 0\rlim}\frac{T(t)x-x}{t}\;\text{exists in }X\Bigr\}
\] 
and 
\[
Ax\coloneqq \gamma\text{-}\lim_{t\to 0\rlim}\frac{T(t)x-x}{t},\quad x\in D(A).
\]
If $(X,\|\cdot\|,\tau)$ is sequentially complete, then $D(A)$ is $\gamma$-dense in $X$ 
by \cite[Proposition 1.3, p.~261]{komura1968}.
The \emph{bi-generator} $(A_{\|\cdot\|,\tau},D(A_{\|\cdot\|,\tau}))$ of $(T(t))_{t\geq 0}$ is given by 
\[
D(A_{\|\cdot\|,\tau})\coloneqq \biggl\{x\in X\;|\;\tau\text{-}\lim_{t\to 0\rlim}\frac{T(t)x-x}{t}\;\text{exists in }X,\; 
\sup_{0<t\leq 1}\frac{\|T(t)x-x\|}{t}<\infty\biggr\}
\] 
and 
\[
A_{\|\cdot\|,\tau}x\coloneqq \tau\text{-}\lim_{t\to 0\rlim}\frac{T(t)x-x}{t},\quad x\in D(A_{\|\cdot\|,\tau}).
\]
In the context of $\tau$-bi-continuous semigroups their generators are actually defined as bi-generators 
(see \cite[Definition 1.2.6, p.~7]{farkas2003}). The notion of the bi-generator was originally introduced in \cite{kuehnemund2001,kuehnemund2003} (and corrected in \cite{farkas2003}).

\begin{prop}\label{prop:bi_generator}
Let $(X,\|\cdot\|,\tau)$ be a Saks space and $(T(t))_{t\geq 0}$ a $\gamma$-strongly continuous, 
locally bounded semigroup on $X$. Then we have 
\[
D(A)=D(A_{\|\cdot\|,\tau})\quad\text{and}\quad A=A_{\|\cdot\|,\tau} .
\]
\end{prop}
\begin{proof}
The inclusion $D(A_{\|\cdot\|,\tau})\subset D(A)$ follows from \cite[I.1.10 Proposition, p.~9]{cooper1978}, which 
says that a sequence in $X$ is $\gamma$-convergent if and only if it is $\tau$-convergent and $\|\cdot\|$-bounded. 
Further, $Af=A_{\|\cdot\|,\tau}f$ for $f\in D(A)$ because $\tau$ is coarser than $\gamma$. 

Conversely, suppose that there is $x\in D(A)$ such that $\sup_{t\in(0,1]}\tfrac{\|T(t)x-x\|}{t}=\infty$. 
Due to the local boundedness of $(T(t))_{t\geq 0}$ this implies that there is a sequence 
$(t_{n})_{n\in\N}$ in $(0,1]$ such that $t_{n}\to 0\rlim$ and $\sup_{n\in\N}\tfrac{\|T(t_{n})x-x\|}{t_{n}}=\infty$. 
Since $\gamma\text{-}\lim_{n\to\infty}\frac{T(t_{n})x-x}{t_{n}}$ exists in $X$, this is a contradiction 
because $\gamma$-convergent sequences are $\|\cdot\|$-bounded by \cite[I.1.10 Proposition, p.~9]{cooper1978}.
\end{proof}

Due to \prettyref{thm:sg_weighted_comp_str_cont_equicont} (a) and (b) we directly get the following corollary 
of \prettyref{prop:bi_generator}.

\begin{cor}\label{cor:bi_generator}
Let $\Omega$ be a Hausdorff space, $(\F,\|\cdot\|,\tau_{\operatorname{co}})$ a Saks space such that 
$\F\subset\mathcal{C}(\Omega)$ and $(A,D(A))$ the generator of a locally bounded weighted composition semigroup 
$(C_{m,\varphi}(t))_{t\geq 0}$ on $\F$ w.r.t.~a jointly continuous co-semiflow $(m,\varphi)$. Then we have 
\[
D(A)=D(A_{\|\cdot\|,\tau})\quad\text{and}\quad A=A_{\|\cdot\|,\tau}.
\]
\end{cor}

For our next observation we recall the definition of the generator of a norm-strongly continuous semigroup on a Banach space. 
Let $(X,\|\cdot\|)$ be a Banach space and $(T(t))_{t\geq 0}$ a $\|\cdot\|$-strongly continuous semigroup on $X$. 
We define the \emph{norm-generator} $(A_{\|\cdot\|},D(A_{\|\cdot\|}))$ of $(T(t))_{t\geq 0}$ 
according to \cite[Chap.~2, 1.2 Definition, p.~49]{engel_nagel2000} by 
\[
D(A_{\|\cdot\|})\coloneqq \Bigl\{x\in X\;|\;\|\cdot\|\text{-}\lim_{t\to 0\rlim}\frac{T(t)x-x}{t}\;\text{exists in }X\Bigr\}
\] 
and 
\[
A_{\|\cdot\|}x\coloneqq \|\cdot\|\text{-}\lim_{t\to 0\rlim}\frac{T(t)x-x}{t},\quad x\in D(A_{\|\cdot\|}).
\]

\begin{prop}\label{prop:norm_generator}
Let $\Omega$ be a Hausdorff space, $(\F,\|\cdot\|,\tau_{\operatorname{co}})$ a sequentially complete Saks space 
such that $\F\subset\mathcal{C}(\Omega)$ and $(A,D(A))$ the generator of the weighted composition semigroup 
$(C_{m,\varphi}(t))_{t\geq 0}$ on $\F$ w.r.t.~a jointly continuous co-semiflow $(m,\varphi)$. 
Then the following assertions hold.
\begin{enumerate}
\item If $(C_{m,\varphi}(t))_{t\geq 0}$ is $\|\cdot\|$-strongly continuous, then  
\[
D(A)=D(A_{\|\cdot\|})\quad\text{and}\quad A=A_{\|\cdot\|}.
\]
\item If $(\F,\|\cdot\|)$ is reflexive and $(C_{m,\varphi}(t))_{t\geq 0}$ locally bounded, 
then the semigroup $(C_{m,\varphi}(t))_{t\geq 0}$ is $\|\cdot\|$-strongly continuous.
\item Let $\bigl[(m,\varphi),\F\bigr]$ denote the space of $\|\cdot\|$-strong continuity of $(C_{m,\varphi}(t))_{t\geq 0}$, i.e.
\[
\bigl[(m,\varphi),\F\bigr]\coloneqq\{f\in\F\;|\;\|\cdot\|\text{-}\lim_{t\to 0\rlim}C_{m,\varphi}(t)f=f\}.
\] 
If $(C_{m,\varphi}(t))_{t\geq 0}$ is locally bounded, then we have
\[
\bigl[(m,\varphi),\F\bigr]=\overline{D(A)}^{\|\cdot\|}
\]
where $\overline{D(A)}^{\|\cdot\|}$ denotes the closure of $D(A)$ w.r.t.~the $\|\cdot\|$-topology.
\end{enumerate}
\end{prop}
\begin{proof}
Due to \cite[Chap.~I, 5.5 Proposition, p.~39]{engel_nagel2000} a $\|\cdot\|$-strongly continuous semigroup 
is exponentially bounded and thus locally bounded. 
Hence part (a) follows from \prettyref{thm:sg_weighted_comp_str_cont_equicont} (b) and 
\cite[Lemma 5.15, p.~2684]{kruse_meichnser_seifert2018}. 
Parts (b) and (c) are a consequence of \prettyref{thm:sg_weighted_comp_str_cont_equicont} (b) and 
\cite[Corollary 1.25, p.~26]{kuehnemund2001} resp.~\cite[Theorem 5.6, p.~340]{budde2019}.
\end{proof}

Let $\Omega$ be a Hausdorff space, $\F\subset\mathcal{C}(\Omega)$ a linear space and $(m,\varphi)$ a co-semiflow for $\F$. 
Then the \emph{Lie generator} of the co-semiflow $(m,\varphi)$ is given by
\begin{align*}
D(A_{m,\varphi})\coloneqq \Bigl\{f\in\F\;|\;&\forall\;x\in\Omega:\;
g(x)\coloneqq\lim_{t\to 0\rlim}\frac{m_{t}(x)f(\varphi_{t}(x))-f(x)}{t}\;\text{exists in }\K\\
&\text{ and }g\in\F \Bigr\}
\end{align*}
and 
\[
A_{m,\varphi}f(x)\coloneqq\lim_{t\to 0\rlim}\frac{m_{t}(x)f(\varphi_{t}(x))-f(x)}{t},\quad f\in D(A_{m,\varphi}),\,x\in\Omega.
\]
In other words, the Lie generator is the generator of the weighted composition semigroup $(C_{m,\varphi}(t))_{t\geq 0}$ 
w.r.t.~the topology of pointwise convergence. The Lie generator was introduced in \cite[p.~115]{dorroh1996} 
for the space $\F=\mathcal{C}_{b}(\Omega)$ of bounded continuous functions on a Polish space $\Omega$, 
equipped with the supremum norm $\|\cdot\|=\|\cdot\|_{\infty}$, and a jointly continuous co-semiflow $(\mathds{1},\varphi)$.
The following proposition generalises \cite[Proposition 2.12, p.~6]{farkas2020} and 
\cite[Proposition 2.4, p.~118]{dorroh1996} where $\F=\mathcal{C}_{b}(\Omega)$ and $\Omega$ a completely regular Hausdorff 
$k$-space resp.~Polish space.

\begin{prop}\label{prop:lie_generator}
Let $\Omega$ be a Hausdorff space, $(\F,\|\cdot\|,\tau_{\operatorname{co}})$ a sequentially complete Saks space 
such that $\F\subset\mathcal{C}(\Omega)$ and $(A,D(A))$ the generator of a locally bounded weighted composition semigroup 
$(C_{m,\varphi}(t))_{t\geq 0}$ on $\F$ w.r.t.~a jointly continuous co-semiflow $(m,\varphi)$. Then we have 
\[
D(A)=D(A_{m,\varphi})\quad\text{and}\quad A=A_{m,\varphi}.
\]
\end{prop}
\begin{proof}
Since $\gamma$ is stronger than $\tau_{\operatorname{co}}$ and thus stronger than the topology of pointwise convergence, we 
only need to prove the inclusion $D(A_{m,\varphi})\subset D(A)$. Let $f\in D(A_{m,\varphi})$ and set $g\coloneqq A_{m,\varphi}f$. 
Since $g\in\F$ and $(C_{m,\varphi}(t))_{t\geq 0}$ is $\gamma$-strongly continuous 
by \prettyref{thm:sg_weighted_comp_str_cont_equicont} (a), the $\gamma$-Riemann integral $\int_{0}^{t}C_{m,\varphi}(s)g\d s$ 
exists in $\F$ for every $t\geq 0$ by \cite[Proposition 1.1, p.~232]{komatsu1964} because $(\F,\gamma)$ is sequentially complete. 
Now, since $f\in D(A_{m,\varphi})$, the function $f_{x}\colon[0,\infty)\to \K$, $s\mapsto m_{s}(x)f(\varphi_{s}(x))$, is 
right-differentiable with right-derivative $g_{x}\colon[0,\infty)\to \K$, $s\mapsto m_{s}(x)g(\varphi_{s}(x))$, because 
\begin{align*}
&\phantom{=}\lim_{t\to 0\rlim}\frac{m_{s+t}(x)f(\varphi_{s+t}(x))-m_{s}(x)f(\varphi_{s}(x))}{t}\\
&=m_{s}(x)\lim_{t\to 0\rlim}\frac{m_{t}(\varphi_{s}(x))f(\varphi_{t}(\varphi_{s}(x)))-f(\varphi_{s}(x))}{t}
=m_{s}(x)g(\varphi_{s}(x))=g_{x}(s)
\end{align*}
for every $s\geq 0$ and $x\in\Omega$. The right-derivative $g_{x}$ is continuous for every $x\in\Omega$ 
as $g\in\F$ and $(m,\varphi)$ is jointly continuous. Hence $f_{x}\in\mathcal{C}^{1}[0,\infty)$ with derivative $g_{x}$ 
by \cite[Chap.~2, Corollary 1.2, p.~43]{pazy1983} and 
\begin{align*}
  C_{m,\varphi}(t)f(x)-f(x)
&=m_{t}(x)f(\varphi_{t}(x))-f(x)
 =\int_{0}^{t}g_{x}(s)\d s
 = \int_{0}^{t}m_{s}(x)g(\varphi_{s}(x))\d s\\
&= \int_{0}^{t}C_{m,\varphi}(s)g(x)\d s
\end{align*}
for every $t\geq 0$ and $x\in\Omega$ by the fundamental theorem of calculus. 
In combination with the existence of the $\gamma$-Riemann integral $\int_{0}^{t}C_{m,\varphi}(s)g\d s$ in $\F$ 
for every $t\geq 0$ this yields 
\[
C_{m,\varphi}(t)f-f=\int_{0}^{t}C_{m,\varphi}(s)g\d s
\]
for every $t\geq 0$, implying our statement by \cite[Proposition 1.2 (2), p.~260]{komura1968}.
\end{proof}

\begin{rem}
Let $\Omega$ be a Hausdorff space, $(\F,\|\cdot\|,\tau_{\operatorname{co}})$ a sequentially complete Saks space 
such that $\F\subset\mathcal{C}(\Omega)$ and $(A,D(A))$ the generator of a locally bounded weighted composition semigroup 
$(C_{m,\varphi}(t))_{t\geq 0}$ on $\F$ w.r.t.~a jointly continuous co-semiflow $(m,\varphi)$. 
We may also define the $\tau_{\operatorname{co}}$\emph{-generator} 
$(A_{\tau_{\operatorname{co}}},D(A_{\tau_{\operatorname{co}}}))$ by 
\[
D(A_{\tau_{\operatorname{co}}})\coloneqq \Bigl\{f\in\F\;|\;
\tau_{\operatorname{co}}\text{-}\lim_{t\to 0\rlim}\frac{C_{m,\varphi}(t)f-f}{t}\;\text{exists in }\F\Bigr\}
\] 
and 
\[
A_{\tau_{\operatorname{co}}}f\coloneqq \tau_{\operatorname{co}}\text{-}\lim_{t\to 0\rlim}\frac{C_{m,\varphi}(t)f-f}{t},
\quad f\in D(A_{\tau_{\operatorname{co}}}).
\]
Then it follows from $\tau_{\operatorname{co}}$ being coarser than $\gamma$ and \prettyref{prop:lie_generator} that 
\[
D(A)=D(A_{\tau_{\operatorname{co}}})\quad\text{and}\quad A=A_{\tau_{\operatorname{co}}}.
\]
\end{rem}

If we have more information on the co-semiflow than just joint continuity, then we may give a simpler characterisation 
of the generator of a weighted composition semigroup.

\begin{prop}\label{prop:generator_mult_seg}
Let $\Omega$ be a Hausdorff space, $(\F,\|\cdot\|,\tau_{\operatorname{co}})$ a sequentially complete Saks space 
such that $\F\subset\mathcal{C}(\Omega)$ and $(A,D(A))$ the generator of a locally bounded weighted composition semigroup 
$(C_{m,\id}(t))_{t\geq 0}$ on $\F$ w.r.t.~a jointly continuous co-semiflow $(m,\id)$. 
If $m_{(\cdot)}(x)$ is right-differentiable in $t=0$ for all $x\in\Omega$, then
\[
D(A)=\{f\in\F\;|\;\dt{m}_{0}f\in\F\}\quad\text{and}\quad Af=\dt{m}_{0}f,\quad f\in D(A).
\]
\end{prop}
\begin{proof}
For $f\in\F$ we have
\[
 \lim_{t\to 0\rlim}\frac{m_{t}(x)f(x)-f(x)}{t}
=\lim_{t\to 0\rlim}\frac{m_{t}(x)-1}{t}f(x)
=\dt{m}_{0}(x)f(x)
\]
for all $x\in\Omega$, yielding our statement by \prettyref{prop:lie_generator}.
\end{proof}

\begin{prop}\label{prop:mixed_gen_max_dom}
Let $\Omega\subset\K$ be open, $(\F,\|\cdot\|,\tau_{\operatorname{co}})$ a sequentially complete Saks space 
such that $\F\subset\mathcal{C}(\Omega)$ and $(A,D(A))$ the generator of a locally bounded weighted composition semigroup 
$(C_{m,\varphi}(t))_{t\geq 0}$ on $\F$ w.r.t.~a jointly continuous co-semiflow $(m,\varphi)$.
\begin{enumerate}
\item  If 
\begin{enumerate}
\item $m_{(\cdot)}(x)\in\mathcal{C}^{1}[0,\infty)$ and $\varphi_{(\cdot)}(x)\in\mathcal{C}^{1}[0,\infty)$ for all $x\in\Omega$,
\end{enumerate}
then 
\[
D_{0}\coloneqq\{f\in\mathcal{C}^{1}_{\K}(\Omega)\cap\F\;|\;\dt{\varphi}_{0}f'+\dt{m}_{0}f\in\F\}\subset D(A)
\]
and $Af=\dt{\varphi}_{0}f'+\dt{m}_{0}f$ for all $f\in D_{0}$.
\item Let $\omega\subset\Omega$ be open. If condition (i) is fulfilled and
\begin{enumerate}
\item[(ii)] $\varphi$ has a generator $G$, i.e.~there is a function $G\in\mathcal{C}(\Omega)$ such that 
$\dt{\varphi}_{t}(x)=(G\circ\varphi_{t})(x)$ for all $t\geq 0$ and $x\in\Omega$, and  
\item[(iii)] $t_{x}\coloneqq\inf\{t>0\;|\;\exists\; y\in\Omega\setminus\omega,\,y\neq x:\; \varphi_{t}(x)=y\}>0$ for all 
$x\in\Omega\setminus\omega$, 
\end{enumerate}
then 
\[
D_{1}\coloneqq\{f\in\mathcal{C}^{1}_{\K}(\omega)\cap\F\;|\;\dt{\varphi}_{0}f'+\dt{m}_{0}f\in\F\}\subset D(A)
\]
and $Af=\dt{\varphi}_{0}f'+\dt{m}_{0}f$ for all $f\in D_{1}$, where 
$\dt{\varphi}_{0}f'+\dt{m}_{0}f\in\F$ in the definition of $D_{1}$ means that there is an extension $g\in\F$ 
of the map $\dt{\varphi}_{0}f'+\dt{m}_{0}f\colon\omega\to\K$.
\end{enumerate}
\end{prop}
\begin{proof}
(a) Let $f\in\mathcal{C}^{1}_{\K}(\Omega)\cap\F$ and $x\in\Omega$. Fix $\widetilde{t}>0$. 
We note that the map $h_{x}\colon[0,\widetilde{t}\,]\to\K$, 
$h_{x}(s)\coloneqq m_{s}(x)\dt{\varphi}_{s}(x)f'(\varphi_{s}(x))+\dt{m}_{s}(x)f(\varphi_{s}(x))$, 
is continuous by condition (i), the continuity of $f'$ and the joint continuity of $(m,\varphi)$. 
Then we have
\begin{align*}
 \frac{m_{t}(x)f(\varphi_{t}(x))-f(x)}{t}
&=\frac{1}{t}\int_{0}^{t}m_{s}(x)\dt{\varphi}_{s}(x)f'(\varphi_{s}(x))+\dt{m}_{s}(x)f(\varphi_{s}(x))\d s\nonumber\\
&=\frac{1}{t}\int_{0}^{t}h_{x}(s)\d s
\end{align*}
for every $0<t\leq\widetilde{t}$ by (i) and the fundamental theorem of calculus. 
This implies 
\begin{align}\label{eq:mixed_gen_max_dom_1}
&\phantom{=}\frac{m_{t}(x)f(\varphi_{t}(x))-f(x)}{t}-(\dt{\varphi}_{0}(x)f'(x)+\dt{m}_{0}(x)f(x))\nonumber\\
&=\frac{1}{t}\int_{0}^{t}h_{x}(s)
  -\dt{\varphi}_{0}(x)f'(x)-\dt{m}_{0}(x)f(x)\d s
\end{align}
for every $0<t\leq\widetilde{t}$. Using that $h_{x}$ is continuous on the compact interval 
$[0,\widetilde{t}\,]$, thus uniformly continuous by the Heine--Cantor theorem, and that 
\[
 \lim_{s\to 0\rlim}h_{x}(s)
=m_{0}(x)\dt{\varphi}_{0}(x)f'(\varphi_{0}(x))+\dt{m}_{0}(x)f(\varphi_{0}(x))
=\dt{\varphi}_{0}(x)f'(x)+\dt{m}_{0}(x)f(x),
\]
for every $\varepsilon>0$ there is $0<\delta\leq \widetilde{t}$ such that for all $s\geq 0$ with $|s|=|s-0|<\delta$ we have 
\[
 \Bigl|\frac{m_{t}(x)f(\varphi_{t}(x))-f(x)}{t}-(\dt{\varphi}_{0}(x)f'(x)+\dt{m}_{0}(x)f(x))\Bigr|
\underset{\eqref{eq:mixed_gen_max_dom_1}}{<}\frac{1}{t}\int_{0}^{t}\varepsilon \d s=\varepsilon
\]
for all $0<t<\delta$. We deduce that
\begin{equation}\label{eq:mixed_gen_max_dom_2}
\lim_{t\to 0\rlim}\frac{m_{t}(x)f(\varphi_{t}(x))-f(x)}{t}=\dt{\varphi}_{0}(x)f'(x)+\dt{m}_{0}(x)f(x).
\end{equation}
The rest of the statement follows from \eqref{eq:mixed_gen_max_dom_2} and \prettyref{prop:lie_generator}.
 
(b) Let $f\in\mathcal{C}^{1}_{\K}(\omega)\cap\F$ and $x\in\Omega$. 
First, we consider the case that $x\in\omega$. 
Since $\omega$ is open, $\varphi_{0}(x)=x\in\omega$, and $\varphi_{(\cdot)}(x)$ is continuous, there is 
$\delta_{x}>0$ such that $\varphi_{t}(x)\in\omega$ for all $t\in[0,\delta_{x}]$. 
It follows that the map $h_{x}\colon[0,\widetilde{t}\,]\to\K$ from part (a) is still a well-defined 
continuous function for the choice $\widetilde{t}\coloneqq \delta_{x}$ and the rest of the proof carries over. 

Let us turn to the case $x\in\Omega\setminus\omega$. 
Now, we need the restriction that $\dt{\varphi}_{0}f'+\dt{m}_{0}f\in\F$. 
We set $p(x)\coloneqq \inf\{t>0\;|\;\varphi_{t}(x)=x\}$. 
If $p(x)=0$, then $x$ is a fixed point of $\varphi$, and thus 
\[
\lim_{t\to 0\rlim}\frac{m_{t}(x)f(\varphi_{t}(x))-f(x)}{t}=\lim_{t\to 0\rlim}\frac{m_{t}(x)-1}{t}f(x)=\dt{m}_{0}(x)f(x)
\]
Suppose that $p(x)>0$. Setting $t(x)\coloneqq\min\{p(x),t_{x}\}$, we observe that $t(x)>0$ by condition (iii). 
Hence the map $h_{x}\colon (0,\widetilde{t}\,]\to\K$ 
from part (a) is still a well-defined continuous function for the choice $\widetilde{t}\coloneqq t(x)$. 
Next, we show that $h_{x}$ is continuously extendable in $s=0$. We denote by $g\in\F$ the extension 
of $\dt{\varphi}_{0}f'+\dt{m}_{0}f$ and note that $g$ as an element of $\F$ is continuous on $\Omega$. 
Then $\widetilde{g}\coloneqq g-\dt{m}_{0}f$ is a continuous extension 
of $\dt{\varphi}_{0}f'$ on $\Omega$. For $0<s<t(x)$ we have by condition (ii) and \prettyref{rem:generator_semiflow}
that 
\[
\dt{\varphi}_{s}(x)=G(\varphi_{s}(x))=\dt{\varphi}_{0}(\varphi_{s}(x))
\]
and thus
\begin{align*}
&\phantom{=}m_{s}(x)\dt{\varphi}_{s}(x)f'(\varphi_{s}(x))-\widetilde{g}(x)\\
&=m_{s}(x)\dt{\varphi}_{s}(x)f'(\varphi_{s}(x))-\dt{\varphi}_{0}(\varphi_{s}(x))f'(\varphi_{s}(x))
  +\dt{\varphi}_{0}(\varphi_{s}(x))f'(\varphi_{s}(x))-\widetilde{g}(x)\\
&=(m_{s}(x)-1)\dt{\varphi}_{0}(\varphi_{s}(x))f'(\varphi_{s}(x))
  +\dt{\varphi}_{0}(\varphi_{s}(x))f'(\varphi_{s}(x))-\widetilde{g}(x)\\
&=(m_{s}(x)-1)\widetilde{g}(\varphi_{s}(x))+\widetilde{g}(\varphi_{s}(x))-\widetilde{g}(x).
\end{align*}
We derive that 
\[
\lim_{s\to 0\rlim} m_{s}(x)\dt{\varphi}_{s}(x)f'(\varphi_{s}(x))=\widetilde{g}(x)
\]
since $\widetilde{g}$ is continuous in $x$ and $(m,\varphi)$ is a $C_{0}$-co-semiflow. 
Hence $h_{x}$ is continuously extendable in $s=0$ by setting $h_{x}(0)\coloneqq \widetilde{g}(x)+\dt{m}_{0}(x)f(x)=g(x)$. 
From here the rest of the proof of part (a) carries over with $\dt{\varphi}_{0}(x)f'(x)+\dt{m}_{0}(x)f(x)$ 
replaced by $g(x)$.
\end{proof}

The expression $p(x)=\inf\{t>0\;|\;\varphi_{t}(x)=x\}$ in the proof of part (b) is also called the \emph{period} of $x\in\Omega$ 
w.r.t.~$\varphi$ (see \cite[p.~660]{robinson1986}). 
For the proof of the converse inclusion in \prettyref{prop:mixed_gen_max_dom} in the case that $\F$ is not a subspace of 
$\mathcal{C}^{1}_{\K}(\Omega)$ we need to know what happens with $\varphi_{t}$ and $m_{t}$ to the left of $t=0$, meaning 
we consider flows and cocycles instead of just semiflows and semicocycles.

\begin{defn}
Let $\Omega$ be a Hausdorff space. A family $\varphi\coloneqq (\varphi_{t})_{t\in\R}$ 
of continuous functions $\varphi_{t}\colon\Omega\to\Omega$ is called a \emph{flow} if 
\begin{enumerate}
\item[(i)] $\varphi_{0}(x)=x$ for all $x\in\Omega$, and 
\item[(ii)] $\varphi_{t+s}(x)=(\varphi_{t}\circ\varphi_{s})(x)$ for all $t,s\in\R$ and $x\in\Omega$.
\end{enumerate}
We call a flow \emph{trivial} and write $\varphi=\id$ if $\varphi_{t}=\id$ for all $t\in\R$.
We call a flow $\varphi$ a $C_{0}$\emph{-flow} if $\lim_{t\to 0}\varphi_{t}(x)=x$ for all $x\in\Omega$.
A family $m\coloneqq (m_{t})_{t\in\R}$ of continuous functions $m_{t}\colon\Omega\to\K$ 
is called a multiplicative \emph{cocycle} for a flow $\varphi$ if 
\begin{enumerate}
\item[(i)] $m_{0}(x)=1$ for all $x\in\Omega$, and 
\item[(ii)] $m_{t+s}(x)=m_{t}(x)m_{s}(\varphi_{t}(x))$ for all $t,s\in\R$ and $x\in\Omega$.
\end{enumerate}
We call a cocycle $m$ \emph{trivial} and write $m=\mathds{1}$ if $m_{t}=\mathds{1}$ for all $t\in\R$. 
We call a cocycle $m$ a $C_{0}$\emph{-cocycle} if $\lim_{t\to 0}m_{t}(x)=1$ for all $x\in\Omega$.
We call the tuple $(m,\varphi)$ a \emph{co-flow} on $\Omega$. 
We call a co-flow $(m,\varphi)$ jointly continuous (separately continuous, $C_{0}$) if 
$\varphi$ and $m$ are both jointly continuous (separately continuous, $C_{0}$). 
\end{defn}
   
We have the following characterisation of joint continuity of flows and cocycles on certain Hausdorff spaces $\Omega$. 
   
\begin{prop}\label{prop:jointly_cont_flow_cocycle}
Let $(m,\varphi)$ be a co-flow on a Hausdorff space $\Omega$. 
\begin{enumerate}
\item Let $\Omega$ be locally compact and $\sigma$-compact. 
Then $\varphi$ is jointly continuous if and only if $\varphi$ is $C_{0}$.
\item Let $\Omega$ be an open subset of a metric space and $\varphi$ jointly continuous. 
Then $m$ is jointly continuous if and only if $m$ is $C_{0}$.
\end{enumerate}
\end{prop}
\begin{proof}
(a) We only need to prove the implication $\Leftarrow$. 
We define $\psi\coloneqq (\psi_{t})_{t\geq 0}$ by $\psi_{t}(x)\coloneqq \varphi_{-t}(x)$ for all $t\geq 0$ 
and $x\in\Omega$. Then it is easily checked that $\psi$ is a semiflow. Further, we have
\[
 \lim_{t\to 0\rlim}\psi_{t}(x)
=\lim_{t\to 0\rlim}\varphi_{-t}(x)
=\lim_{t\to 0\llim}\varphi_{t}(x)
=x
\]
for all $x\in\Omega$. It follows from \prettyref{prop:jointly_cont_semiflow} that $\psi$ is jointly continuous 
and $(\varphi_{t})_{t\geq 0}$ as well. Since $\psi_{0}(x)=x=\varphi_{0}(x)$ for all $x\in\Omega$, 
we get that $\varphi$ is jointly continuous. The proof of part (b) is analogous and we only need 
to use \prettyref{prop:jointly_cont_cocycle} instead of \prettyref{prop:jointly_cont_semiflow}.
\end{proof}

\begin{thm}\label{thm:mixed_gen_max_dom_real}
Let $\Omega\subset\R$ be open, $(\F,\|\cdot\|,\tau_{\operatorname{co}})$ a sequentially complete Saks space 
such that $\F\subset\mathcal{C}(\Omega)$ and $(A,D(A))$ the generator of a locally bounded weighted composition semigroup 
$(C_{m,\varphi}(t))_{t\geq 0}$ on $\F$ w.r.t.~a $C_{0}$-co-flow $(m,\varphi)$ such that 
$m_{(\cdot)}(x)\in\mathcal{C}^{1}(\R)$ and $\varphi_{(\cdot)}(x)\in\mathcal{C}^{1}(\R)$ for all $x\in\Omega$,
and $m_{t}(x)\neq 0$ for all $(t,x)\in\R\times\Omega$.
\begin{enumerate}
\item If the map $V_{\varphi}\to\Omega$, $(t,x)\mapsto \varphi_{t}(x)$, is surjective, where 
$V_{\varphi}\coloneqq\{(t,x)\in\R\times\Omega\;|\;\dt{\varphi}_{t}(x)\neq 0\}$, 
then
\[
D(A)=\{f\in\mathcal{C}^{1}(\Omega)\cap\F\;|\;\dt{\varphi}_{0}f'+\dt{m}_{0}f\in\F\}
\]
and $Af=\dt{\varphi}_{0}f'+\dt{m}_{0}f$ for all $f\in D(A)$.
\item If 
\begin{enumerate}
\item[(i)] $(\varphi_{t})_{t\geq 0}$ has a generator $G$, and
\item[(ii)] $t_{x}\coloneqq\inf\{t>0\;|\;\exists\; y\in N_{G},\,y\neq x:\; \varphi_{t}(x)=y\}>0$ for all $x\in N_{G}$, 
where $N_{G}=\{z\in\Omega\;|\;G(z)=0\}$,
\end{enumerate}
then
\[
D(A)=\{f\in\mathcal{C}^{1}(\Omega\setminus N_{G})\cap\F\;|\;\dt{\varphi}_{0}f'+\dt{m}_{0}f\in\F\}
\]
and $Af=\dt{\varphi}_{0}f'+\dt{m}_{0}f$ for all $f\in D(A)$.
\end{enumerate}
\end{thm}
\begin{proof}
(a) Due to \prettyref{prop:mixed_gen_max_dom} (a) we only need to show that $D(A)\subset D_{0}$. 
Let $f\in D(A)$ and $x\in\Omega$. By assumption there is $(t_{0},x_{0})\in V_{\varphi}$ 
such that $x=\varphi_{t_{0}}(x_{0})$. The arguments in the proof of \prettyref{prop:lie_generator} 
in combination with \prettyref{prop:jointly_cont_flow_cocycle} applied to the 
$C_{0}$-co-flow $(m,\varphi)$ show that $f_{x_{0}}\colon\R\to \R$, $s\mapsto m_{s}(x_{0})f(\varphi_{s}(x_{0}))$, 
is continuously differentiable. We deduce that $f\circ\varphi_{(\cdot)}(x_{0})$ is continuously differentiable on $\R$ since 
\[
f(\varphi_{t}(x_{0}))=\frac{f_{x_{0}}(t)}{m_{t}(x_{0})}
\] 
for all $t\in\R$, $m_{(\cdot)}(x_{0})\in\mathcal{C}^{1}(\R)$ and $f_{x_{0}}\in\mathcal{C}^{1}(\R)$. 
By assumption we know that $\varphi_{(\cdot)}(x_{0})\in\mathcal{C}^{1}(\R)$ with $\dt{\varphi}_{t_{0}}(x_{0})\neq 0$. 
By the inverse function theorem there is an open neighbourhood $U\coloneqq U(t_{0})\subset\R$ of $t_{0}$
such that $\varphi_{(\cdot)}(x_{0})$ is invertible on $U$ and the inverse is continuously differentiable 
on the open neighbourhood $W\coloneqq \varphi_{U}(x_{0})\subset\Omega$ of $x=\varphi_{t_{0}}(x_{0})$. Noting that 
\[
f(y)=\Bigl(\bigl(f\circ\varphi_{(\cdot)}(x_{0})\bigr)\circ [\varphi_{(\cdot)}(x_{0})_{\mid U}]^{-1}\Bigr)(y)
\]
for all $y\in W$, we conclude that $f$ is continuously differentiable in $x=\varphi_{t_{0}}(x_{0})\in W$, 
yielding $f\in\mathcal{C}^{1}(\Omega)\cap\F$. 
Using \eqref{eq:mixed_gen_max_dom_2} and \prettyref{prop:lie_generator}, finishes the proof of part (a).

(b) Due to \prettyref{prop:mixed_gen_max_dom} (b) with $\omega\coloneqq \Omega\setminus N_{G}$ we only need to show 
that $D(A)\subset D_{1}$. 
Let $f\in D(A)$. Since $\varphi_{0}(x)=x$ and
\[
\dt{\varphi}_{0}(x)=G(\varphi_{0}(x))=G(x)\neq 0
\]
for every $x\in\Omega\setminus N_{G}$, we obtain that 
the map $V_{\varphi}\to\Omega\setminus N_{G}$, $(t,x)\mapsto \varphi_{t}(x)$, is surjective. 
Hence the proof of part (a) shows that $f$ is continuously differentiable in every $x\in\Omega\setminus N_{G}$. 
The first part of the proof of \prettyref{prop:mixed_gen_max_dom} (b) yields that 
\[
g(x)\coloneqq \lim_{t\to 0\rlim}\frac{m_{t}(x)f(\varphi_{t}(x))-f(x)}{t}=\dt{\varphi}_{0}(x)f'(x)+\dt{m}_{0}(x)f(x)
\]
holds for all $x\in\Omega\setminus N_{G}$. The left-hand side $g$ of this equation belongs to $\F$ as $f\in D(A)$,
yielding $f\in D_{1}$.
\end{proof}

Looking at the proof of \prettyref{thm:mixed_gen_max_dom_real} (b), we see that part (a) is a special case of (b) if there 
is a generator $G$ such that $N_{G}=\varnothing$. 
If $(\varphi_{t})_{t\geq 0}$ is the restriction of a jointly continuous holomorphic semiflow $\psi$ 
on an open set $\widetilde{\Omega}\subset\C$, 
i.e.~$\varphi_{t}=\psi_{t}$ on $\Omega\coloneqq\widetilde{\Omega}\cap\R$ for all $t\geq 0$, 
then the generator $G$ of $(\varphi_{t})_{t\geq 0}$ exists by \prettyref{thm:generator_hol_semiflow}, namely, 
it is the restriction of the generator of $\psi$ to $\Omega$. So condition (i) of 
\prettyref{thm:mixed_gen_max_dom_real} (b) is fulfilled in this case. If in addition $\psi$ is non-trivial, 
$\widetilde{\Omega}$ simply connected and $\widetilde{\Omega}\neq\C$, then $|N_{G}|\leq|\operatorname{Fix}(\psi)|\leq 1$ 
by \prettyref{prop:fixed_points_holomorphic} and thus $t_{x}=\inf\varnothing=\infty$ for all $x\in N_{G}$, 
yielding that condition (ii) of \prettyref{thm:mixed_gen_max_dom_real} (b) is also fulfilled.

Now, we will see that the proof of the converse inclusion in \prettyref{prop:mixed_gen_max_dom} is much simpler 
if $\F$ is a subspace of $\mathcal{C}^{1}_{\K}(\Omega)$. 

\begin{thm}\label{thm:mixed_gen_max_dom_diff_everywhere}
Let $\Omega\subset\K$ be open, $(\F,\|\cdot\|,\tau_{\operatorname{co}})$ a sequentially complete Saks space 
such that $\F\subset\mathcal{C}^{1}_{\K}(\Omega)$ and $(A,D(A))$ the generator of a locally bounded weighted 
composition semigroup $(C_{m,\varphi}(t))_{t\geq 0}$ on $\F$ w.r.t.~a jointly continuous co-semiflow $(m,\varphi)$.
If $m_{(\cdot)}(x)\in\mathcal{C}^{1}[0,\infty)$ and $\varphi_{(\cdot)}(x)\in\mathcal{C}^{1}[0,\infty)$ for all $x\in\Omega$, 
then 
\[
D(A)=\{f\in\F\;|\;\dt{\varphi}_{0}f'+\dt{m}_{0}f\in\F\}
\]
and $Af=\dt{\varphi}_{0}f'+\dt{m}_{0}f$ for all $f\in D(A)$.
\end{thm}
\begin{proof}
By \prettyref{prop:mixed_gen_max_dom} (a) and the assumption $\F\subset\mathcal{C}^{1}_{\K}(\Omega)$, we have
\[
D_{0}=\{f\in\F\;|\;\dt{\varphi}_{0}f'+\dt{m}_{0}f\in\F\}\subset D(A)
\]
and $Af=\dt{\varphi}_{0}f'+\dt{m}_{0}f$ for all $f\in D_{0}$. The converse inclusion holds by 
$\F\subset\mathcal{C}^{1}_{\K}(\Omega)$, \eqref{eq:mixed_gen_max_dom_2} and \prettyref{prop:lie_generator}.
\end{proof}

Now, we may use the theory on weighted composition semigroups developed so far to show 
(in combination with \prettyref{prop:strongly_mixed_cont}) that 
the condition that $m_{(\cdot)}(x)\in\mathcal{C}^{1}[0,\infty)$ for all $x\in\Omega$ is quite often a necessary condition 
for $\gamma$-strong continuity of the induced weighted composition semigroup. 
The underlying idea of the proof comes from the proof of \cite[Theorem 1, p.~470]{koenig1990}.

\begin{prop}\label{prop:cont_diff_semi_cocycle_diff_everywhere}
Let $\Omega\subset\K$ be open, $(\F,\|\cdot\|,\tau_{\operatorname{co}})$ a sequentially complete Saks space 
such that $\F\subset\mathcal{C}^{1}_{\K}(\Omega)$ and 
\begin{equation}\label{eq:exist_func_not_vanish}
\forall\;x\in\Omega\;\exists\;F\in\F:\;F(x)\neq 0.
\end{equation}
If $(m,\varphi)$ is a jointly continuous co-semiflow for $\F$, $(C_{m,\varphi}(t))_{t\geq 0}$ locally bounded 
and $\varphi$ has a generator $G$, 
then $m_{(\cdot)}(x)\in\mathcal{C}^{1}[0,\infty)$, $\dt{m}_{0}\in\mathcal{C}(\Omega)$ 
and $m_{t}(x)=\exp(\int_{0}^{t}\dt{m}_{0}(\varphi_{s}(x))\d s)$ for all $t\geq 0$ and $x\in\Omega$. 
If in addition $\K=\C$ and $G\in\mathcal{H}(\Omega)$, 
then $\dt{m}_{0}\in\mathcal{H}(\Omega)$.
\end{prop}
\begin{proof}
Let $(A,D(A))$ be the generator of the $\gamma$-strongly continuous weighted composition semigroup 
$(C_{m,\varphi}(t))_{t\geq 0}$ on $\F$. We fix $x\in\Omega$. By \eqref{eq:exist_func_not_vanish} there is 
$F\in\F$ such that $F(x)\neq 0$. Since $F$ is continuous on $\Omega$, there is a compact neighbourhood $U\subset\Omega$ of $x$ 
such that $F(z)\neq 0$ for all $z\in U$. 
Due to \prettyref{thm:sg_weighted_comp_str_cont_equicont} (a) $(C_{m,\varphi}(t))_{t\geq 0}$ is $\gamma$-strongly continuous 
on the sequentially complete space $(X,\gamma)$, impying that $D(A)$ is $\gamma$-dense in $\F$ 
by \cite[Proposition 1.3, p.~261]{komura1968}. 
Thus there is $f\in D(A)$ such that $f(z)\neq 0$ for all $z\in U$ since $\gamma$ is stronger than the topology 
$\tau_{\operatorname{co}}$ and $U$ compact. Using that $\varphi_{0}(x)=x$ and the joint continuity of $\varphi$, 
we deduce that there are $t_{0}>0$ and a neighbourhood $U_{0}$ of $x$ such that 
$\varphi_{s}(\zeta)\in U$ for all $s\in[0,t_{0}]$ and $\zeta\in U_{0}$. 
In particular, $f(\varphi_{s}(\zeta))\neq 0$ for all $s\in[0,t_{0}]$ and $\zeta\in U_{0}$.
Further, we have 
\begin{align*}
 \frac{m_{s}(\zeta)-1}{s}
&=\frac{1}{f(\varphi_{s}(\zeta))}\frac{m_{s}(\zeta)f(\varphi_{s}(\zeta))-f(\varphi_{s}(\zeta))}{s}\\
&=\frac{1}{f(\varphi_{s}(\zeta))}\frac{C_{m,\varphi}(s)f(\zeta)-f(\varphi_{s}(\zeta))}{s}
\end{align*}
for all $0<s\leq t_{0}$ and 
\begin{align*}
  \lim_{s\to 0\rlim}\frac{m_{s}(\zeta)-1}{s}
&=\lim_{s\to 0\rlim}\frac{1}{f(\varphi_{s}(\zeta))}\Biggl(\frac{C_{m,\varphi}(s)f(\zeta)-f(\zeta)}{s}
  -\frac{f(\varphi_{s}(\zeta))-f(\zeta)}{s}\Biggr)\\
&=\frac{1}{f(\zeta)}\bigl(Af(\zeta)-f'(\zeta)\dt{\varphi}_{0}(\zeta)\bigr)
 =\frac{1}{f(\zeta)}\bigl(Af(\zeta)-f'(\zeta)G(\zeta)\bigr)
\end{align*}
for all $\zeta\in U_{0}$. Therefore $\dt{m}_{0}(\zeta)=\frac{1}{f(\zeta)}(Af(\zeta)-f'(\zeta)G(\zeta))$ for all 
$\zeta\in U_0$, yielding that $\dt{m}_{0}$ is continuous on $\Omega$ because $x$ is arbitrary. 
If in addition $\K=\C$ and $G\in\mathcal{H}(\Omega)$, then we also get $\dt{m}_{0}\in\mathcal{H}(\Omega)$. 
The continuous differentiability of $m_{(\cdot)}(x)$ on $[0,\infty)$ and 
$m_{t}(x)=\exp(\int_{0}^{t}\dt{m}_{0}(\varphi_{s}(x))\d s)$ for all $t\geq 0$ and $x\in\Omega$ 
follow from \prettyref{prop:cont_diff_semicocycle}.
\end{proof}

We note that we may replace the condition $\F\subset\mathcal{C}^{1}_{\K}(\Omega)$ by the condition 
$D(A)\subset\mathcal{C}^{1}_{\K}(\Omega)$ because we only need it for the function $f\in D(A)$ in the proof of 
\prettyref{prop:cont_diff_semi_cocycle_diff_everywhere}. Condition \eqref{eq:exist_func_not_vanish} is for example 
fulfilled if $\mathds{1}\in\F$.

\begin{cor}\label{cor:cont_diff_semi_cocycle_diff_everywhere}
Let $\Omega\subset\C$ be open, $(\F,\|\cdot\|,\tau_{\operatorname{co}})$ a sequentially complete Saks space 
such that $\F\subset\mathcal{H}(\Omega)$ and $\mathds{1}\in\F$.
If $(m,\varphi)$ is a jointly continuous holomorphic co-semiflow for $\F$ and $(C_{m,\varphi}(t))_{t\geq 0}$ locally bounded, 
then $m_{(\cdot)}(z)\in\mathcal{C}^{1}[0,\infty)$, $\dt{m}_{0}\in\mathcal{H}(\Omega)$ and
$m_{t}(z)=\exp(\int_{0}^{t}\dt{m}_{0}(\varphi_{s}(z))\d s)$ for all $t\geq 0$ and $z\in\Omega$. 
\end{cor}
\begin{proof}
Our statement follows from \prettyref{prop:cont_diff_semi_cocycle_diff_everywhere} since
\eqref{eq:exist_func_not_vanish} is fulfilled as $\mathds{1}\in\F$ and $\varphi$ has a generator $G\in\mathcal{H}(\Omega)$ 
by \prettyref{thm:generator_hol_semiflow}. 
\end{proof}

As a special case of \prettyref{cor:cont_diff_semi_cocycle_diff_everywhere} we get, 
in combination with \cite[Chap.~I, 5.5 Proposition, p.~39]{engel_nagel2000}, \cite[Theorem 1, p.~470]{koenig1990} back 
where $\Omega=\D$ and $\mathcal{F}(\D)=H^{p}$, $1\leq p<\infty$, the Hardy space on $\D$. 
Moreover, in contrast to \prettyref{prop:cocycle_int_hol}, where $\Omega\subset\C$ is connected, 
\prettyref{cor:cont_diff_semi_cocycle_diff_everywhere} allows arbitrary open sets $\Omega$ but for the cost 
of more assumptions on $(m,\varphi)$. For example the condition $\mathds{1}\in\F$ 
in \prettyref{cor:cont_diff_semi_cocycle_diff_everywhere} implies that $m_{t}\in\F$ for all $t\geq 0$ 
(see \prettyref{rem:co_semiflow_for_space}).

\section{Converse of the holomorphic generation theorem}
\label{sect:converse_generators}

Let $\Omega\subset\C$ be open, $(\F,\|\cdot\|,\tau_{\operatorname{co}})$ a sequentially complete Saks space such that 
$\F\subset\mathcal{H}(\Omega)$. Suppose that $\Omega$ is connected or that $\mathds{1}\in\F$. 
Due to \prettyref{thm:generator_hol_semiflow}, \prettyref{thm:mixed_gen_max_dom_diff_everywhere}, 
and \prettyref{prop:cocycle_int_hol} or \prettyref{cor:cont_diff_semi_cocycle_diff_everywhere} 
we know that the generator $(A,D(A))$ of a locally bounded weighted composition semigroup 
$(C_{m,\varphi}(t))_{t\geq 0}$ on $\F$ w.r.t.~a jointly continuous holomorphic 
co-semiflow $(m,\varphi)$ is given by 
\begin{equation}\label{eq:converse_gen}
D(A)=\{f\in\F\;|\;Gf'+gf\in\F\},\quad 
Af=Gf'+gf,\,f\in D(A),
\end{equation}
where $G\in\mathcal{H}(\Omega)$ is the generator of $\varphi$ and 
$g\coloneqq\dt{m}_{0}\in\mathcal{H}(\Omega)$. 

In this section we want to prove the converse statement, namely, if we know that 
the domain of a $\gamma$-strongly continuous semigroup $(T(t))_{t\geq 0}$ on $\F$ 
is given by \eqref{eq:converse_gen}, we want to show, under suitable conditions, that this semigroup is a weighted composition 
semigroup whose semiflow has $G$ as a generator and whose semicocycle $m$ is given by 
$m_{t}(z)\coloneqq\exp(\int_{0}^{t}g(\varphi_{s}(z))\d s)$ for all $t\geq 0$ and $z\in\Omega$. 

Our first result in this direction is an analogon of \cite[Main theorem, p.~490]{gutierrez2019} ($g=0$) and 
\cite[Theorem 3.1, p.~69]{gutierrez2023} where $\Omega=\D$ and $(T(t))_{t\geq 0}$ is a $\|\cdot\|$-strongly continuous semigroup. 
We define the space of holomorphic germs near the closed unit disc $\overline{\D}$ by the inductive limit
\[
\mathcal{H}(\overline{\D})\coloneqq\lim\limits_{\substack{\longrightarrow \\ \omega\subset\C\;\text{open},\;\overline{\D}\subset\omega\;}}
(\mathcal{H}(\omega),\tau_{\operatorname{co}})
\]
equipped with its inductive limit topology (see e.g.~\cite[p.~81--82]{berenstein_gay1995}). 
For $n\in\N_{0}$ set $e_{n}\colon \C\to\C$, $e_{n}(z)\coloneqq z^{n}$. Then $e_{n}\in\mathcal{H}(\overline{\D})$ 
for all $n\in\N_{0}$. Hence the assumption that $\mathcal{H}(\overline{\D})\hookrightarrow (\mathcal{F}(\D),\|\cdot\|)$ 
embeds continuously implies that $e_{n}\in\mathcal{F}(\D)$ for all $n\in\N_{0}$.

\begin{thm}\label{thm:converse_gen}
Let $(\mathcal{F}(\D),\|\cdot\|,\tau_{\operatorname{co}})$ be a sequentially complete Saks space 
such that $\mathcal{F}(\D)\subset\mathcal{H}(\D)$, and $\mathcal{H}(\overline{\D})\hookrightarrow (\mathcal{F}(\D),\|\cdot\|)$ 
embeds continuously. 

If $\mathbf{T}\coloneqq (T(t))_{t\geq 0}$ is a $\gamma$-strongly continuous semigroup on $\mathcal{F}(\D)$ 
with generator $(A,D(A))$ of the form 
\[
D(A)=\{f\in\mathcal{F}(\D)\;|\;Gf'+gf\in\mathcal{F}(\D)\},\quad Af=Gf'+gf,\,f\in D(A),
\]
for some $G,g\in\mathcal{H}(\D)$, then there is a jointly continuous holomorphic co-semiflow $(m,\varphi)$ 
for $\mathcal{F}(\D)$ such that $G$ is the generator of $\varphi$, 
$m_{t}(z)=\exp(\int_{0}^{t}g(\varphi_{s}(z))\d s)$ for all $t\geq 0$, $z\in\D$, and 
$\mathbf{T}=(C_{m,\varphi}(t))_{t\geq 0}$.
\end{thm}
\begin{proof}
The proof of \cite[Theorem 3.1, p.~69]{gutierrez2023} carries over to our setting. We only have to adjust the proof in three instances. First, we have to use that $Af=\frac{\d}{\d t}T(t)f$ for all $t\geq 0$ and $f \in 
D(A)$ by \cite[Proposition 1.2 (1), p.~260]{komura1968} in the proof of \cite[Claim 1, p.~70]{gutierrez2023}. 
Second, $D(A)$ is $\gamma$-dense in $\mathcal{F}(\D)$ 
by \cite[Proposition 1.3, p.~261]{komura1968}. Thus for every $f\in\mathcal{F}(\D)$ there is a net $(f_{\iota})_{\iota\in I}$, 
$I$ a directed set, which is $\gamma$-convergent to $f$. Since $T(t)$ is $\gamma$-continuous for every $t\geq 0$, 
this implies that $(T(t)f_{\iota})_{\iota\in I}$ is $\gamma$-convergent to $T(t)f$ for every $t\geq 0$. 
This proves the validity of \cite[Eq.~(3.6), p.~71]{gutierrez2023} because $\gamma$ is finer than 
$\tau_{\operatorname{co}}$. 
Third, we note that $\delta_{z}\in (\mathcal{F}(\D),\gamma)'$ and $\delta_{z}\in (\mathcal{F}(\D),\|\cdot\|)'$ 
for all $z\in\D$, where $\delta_{z}(f)\coloneqq f(z)$ for all $f\in\mathcal{F}(\D)$, because $\delta_{z}\in (\mathcal{F}(\D),\tau_{\operatorname{co}})'$ 
and $\gamma$ and the $\|\cdot\|$-topology are finer than $\tau_{\operatorname{co}}$. 
We now get by \cite[Eq.~(3.10), p.~72]{gutierrez2023} that 
\[
\varphi_{t}(z)^{n}=\frac{1}{m_{t}(z)}T(t)(e_{n})(z)
\]
for all $z\in\D$, $n\in\N$ and $0\leq t<t_{0}$ with $t_{0}>0$ from \cite[p.~69]{gutierrez2023}. 
Let $\Gamma_{\gamma}$ denote a directed system of continuous seminorms that generates $\gamma$. 
Since $\delta_{z}\in (\mathcal{F}(\D),\gamma)'$, $T(t)\in\mathcal{L}(\mathcal{F}(\D),\gamma)$ 
and the $\|\cdot\|$-topology is finer than $\gamma$, 
there are $p=p_{z,t}\in \Gamma_{\gamma}$, $K_{1}=K_{1}(z,t)\geq 0$ and $K_{2}=K_{2}(p)\geq 0$ such that 
\[
     |\varphi_{t}(z)^{n}|
= \Bigl|\frac{1}{m_{t}(z)}T(t)(e_{n})(z)\Bigr|     
\leq K_{1}\frac{1}{m_{t}(z)}p(e_{n})
\leq K_{1}K_{2}\frac{1}{m_{t}(z)}\|e_{n}\|
\]
for all $z\in\D$, $n\in\N$ and $0\leq t<t_{0}$. This implies that there is $0<t_{1}\leq t_{0}$ 
such that $|\varphi_{t}(z)|<1$ for all $z\in\D$ and $0\leq t<t_{1}$ like in the proof of 
\cite[Claim 2, p.~492]{gutierrez2019}. 
This is all we have to change in the proof of 
\cite[Theorem 3.1, p.~69]{gutierrez2023}. 
Moreover, we note that the joint continuity of $\varphi$ follows from \cite[p.~65]{gutierrez2023} 
and \prettyref{prop:jointly_cont_semiflow} (b), and the joint continuity of $m$ 
from \prettyref{prop:jointly_cont_int_cocycle} (a).
\end{proof}

We note that \prettyref{thm:converse_gen} implies \cite[Theorem 3.1, p.~69]{gutierrez2023} if 
$(T(t))_{t\geq 0}$ is $\|\cdot\|$-strongly continuous and $T(t)\in\mathcal{L}(X,\gamma)$ for all $t\geq 0$ 
because then $(T(t))_{t\geq 0}$ is also $\gamma$-strongly continuous. 
The semicocycle $m$ in \prettyref{thm:converse_gen} is given by $m_{t}=T(t)\mathds{1}=T(t)e_{0}$ and the semiflow $\varphi$ 
by $\varphi_{t}=\frac{1}{m_{t}}T(t)\operatorname{id}=\frac{1}{m_{t}}T(t)e_{1}$ for all $t\geq 0$. 
Further, it is shown in \cite[p.~176--177]{bernard2022} that the assumption that 
$\mathcal{H}(\overline{\D})\hookrightarrow (\mathcal{F}(\D),\|\cdot\|)$ 
embeds continuously is equivalent to $\limsup_{n\to\infty}\|e_{n}\|^{\frac{1}{n}}\leq 1$.

\begin{exa}\label{ex:cont_embed_induc_lim}
For the following spaces $(\mathcal{F}(\D),\|\cdot\|)$ the embedding 
$\mathcal{H}(\overline{\D})\hookrightarrow (\mathcal{F}(\D),\|\cdot\|)$ is continuous:
\begin{enumerate}
\item The Hardy spaces $(H^{p},\|\cdot\|_{p})$ for $1\leq p\leq\infty$ since $\|e_{n}\|_{p}=1$ for all $n\in\N$.
\item The Bergman spaces $(A_{\alpha}^{p},\|\cdot\|_{\alpha,p})$ for $\alpha>-1$ and $1\leq p<\infty$ since 
\begin{align*}
 \|e_{n}\|_{\alpha,p}^{p}
&=\frac{\alpha+1}{\pi}\int_{\D}|z|^{np}(1-|z|^{2})^{\alpha}\d z 
 =\frac{\alpha+1}{\pi}\int_{0}^{2\pi}\int_{0}^{1}r^{np+1}(1-r^{2})^{\alpha}\d r \d\theta\\
&\leq 2(\alpha+1)\int_{0}^{1}r^{p+1}(1-r^{2})^{\alpha}\d r
 = (\alpha+1)B\left(\frac{p}{2}+1,\alpha+1\right)
\end{align*}
for all $n\in\N$ by \cite[Eq.~3251-1, p.~327]{gradshteyn2014} with $\mu\coloneqq p+2$, $\lambda\coloneqq 2$ 
and $\nu\coloneqq \alpha+1$, where $B$ denotes the Beta function, and thus 
$\limsup_{n\to\infty}\|e_{n}\|_{\alpha,p}^{\frac{1}{n}}\leq1$.
\item The Dirichlet space $(\mathcal{D},\|\cdot\|_{\mathcal{D}})$ since 
\[
 \|e_{n}\|_{\mathcal{D}}^{2}
=0+\frac{1}{\pi}\int_{\D}n^{2}|z|^{2n-2}\d z
=\frac{n^{2}}{\pi}\int_{0}^{2\pi}\int_{0}^{1}r^{2n-1}\d r \d\theta
=\frac{2n^{2}}{2n}
=n
\]
for all $n\in\N$ and thus $\lim_{n\to\infty}\|e_{n}\|_{\mathcal{D}}^{\frac{1}{n}}=1$.
\item The $v$-Bloch spaces $(\mathcal{B}v(\D),\|\cdot\|_{\mathcal{B}v(\D)})$ for bounded continuous $v\colon\D\to (0,\infty)$ 
since 
\[
\|e_{n}\|_{\mathcal{B}v(\D)}=0+\sup_{z\in\D}n|z|^{n-1}v(z)\leq n\|v\|_{\infty}
\]
for all $n\in\N$ and thus $\limsup_{n\to\infty}\|e_{n}\|_{\mathcal{B}v(\D)}^{\frac{1}{n}}\leq 1$.
\item The spaces $(\mathcal{H}v(\D),\|\cdot\|_{v})$ of weighted holomorphic functions on $\D$ 
for boun\-ded continuous $v\colon\D\to (0,\infty)$ since 
\[
\|e_{n}\|_{v}=\sup_{z\in\D}|z|^{n}v(z)\leq \|v\|_{\infty}
\] 
for all $n\in\N$ and thus $\limsup_{n\to\infty}\|e_{n}\|_{v}^{\frac{1}{n}}\leq 1$.
\end{enumerate}
\end{exa}

In particular, the embedding $\mathcal{H}(\overline{\D})\hookrightarrow 
(\mathcal{B}_{\alpha},\|\cdot\|_{\mathcal{B}_{\alpha}})$ for $\alpha>0$ is continuous by 
\prettyref{ex:cont_embed_induc_lim} (d) because $v_{\alpha}(z)=(1-|z|^2)^{\alpha}\leq 1$ 
for all $z\in\D$. Further, we note that \prettyref{thm:loc_bound_sg_hardy_bergman_dirichlet}, 
\prettyref{thm:loc_bound_sg_bloch} and \prettyref{thm:loc_bound_sg_Hv} in combination with 
\prettyref{ex:cont_embed_induc_lim} answer the question in 
\cite[Remark 3.2, p.~72]{gutierrez2023} for several spaces, namely, to give a sufficient condition 
such that $C_{m,\varphi}(t)\in\mathcal{F}(\D)$ for all $t\geq 0$ where $(\mathcal{F}(\D),\|\cdot\|)$ is a space 
whose embedding $\mathcal{H}(\overline{\D})\hookrightarrow (\mathcal{F}(\D),\|\cdot\|)$ is continuous.

\begin{rem}\label{rem:converse_gen}
Let $(\mathcal{F}(\D),\|\cdot\|,\tau_{\operatorname{co}})$ be a sequentially complete Saks space 
such that $\mathcal{F}(\D)\subset\mathcal{H}(\D)$, and
$\mathcal{H}(\overline{\D})\hookrightarrow (\mathcal{F}(\D),\|\cdot\|)$ embeds continuously. 
The assumption that $(\mathcal{F}(\D),\|\cdot\|,\tau_{\operatorname{co}})$ is a sequentially complete Saks space and
$\mathbf{T}$ a $\gamma$-strongly continuous semigroup on $\mathcal{F}(\D)$ may also be replaced by the assumption that  
$(\mathcal{F}(\D),\|\cdot\|,\tau_{\operatorname{co}})$ is a sequentially complete Saks space and 
$\mathbf{T}$ a $\tau_{\operatorname{co}}$-bi-continuous semigroup on $\mathcal{F}(\D)$ 
with generator $(A_{\|\cdot\|,\tau},D(A_{\|\cdot\|,\tau}))$. Indeed, we have $A_{\|\cdot\|,\tau}f=\frac{\d}{\d t}T(t)f$ 
for all $t\geq 0$ and $f \in D(A_{\|\cdot\|,\tau})$ by \cite[Proposition 11 (a), p.~214--215]{kuehnemund2003}. 
Further, for every $f\in\mathcal{F}(\D)$ there is a $\|\cdot\|$-bounded sequence $(f_{n})_{n\in\N}$ 
in $D(A_{\|\cdot\|,\tau})$ which is $\tau_{\operatorname{co}}$-convergent to $f$ 
by \cite[Corollary 13, p.~215]{kuehnemund2003}. Thus $(T(t)f_{n})_{n\in\N}$ is $\tau_{\operatorname{co}}$-convergent 
to $T(t)f$ for every $t\geq 0$ since $\mathbf{T}$ is locally $\tau_{\operatorname{co}}$-bi-continuous. 
This implies the validity of \cite[Eq.~(3.6), p.~71]{gutierrez2023} as well. Further, 
\[
     |\varphi_{t}(z)^{n}|
= \Bigl|\frac{1}{m_{t}(z)}T(t)(e_{n})(z)\Bigr|     
\leq \|\delta_{z}\|_{(\mathcal{F}(\D),\|\cdot\|)'}\|T(t)\|_{\mathcal{L}(\mathcal{F}(\D))}\frac{1}{m_{t}(z)}\|e_{n}\|
\]
for all $z\in\D$, $n\in\N$ and $0\leq t<t_{0}$ since 
$\delta_{z}\in (\mathcal{F}(\D),\|\cdot\|)'$ and $T(t)\in\mathcal{L}(\mathcal{F}(\D))$ by 
\cite[Definition 3, p.~207]{kuehnemund2003}.
\end{rem}

We remark that \prettyref{thm:converse_gen} is not restricted to the case $\Omega=\D$ (cf.~\cite[p.~177--178]{arendt2019} ($g=0$) 
and \cite[Remark 3.3, p.~72--73]{gutierrez2023} in the case of $\|\cdot\|$-strongly continuous semigroups). 

\begin{rem}
Let $\Omega\subsetneq\C$ be open and simply connected, and $(\F,\|\cdot\|,\tau_{\operatorname{co}})$ 
a sequentially complete Saks space such that $\F\subset\mathcal{H}(\Omega)$. 
Suppose that $(S(t))_{t\geq 0}$ is a $\gamma$-strongly continuous semigroup on $\F$ 
with generator $(A,D(A))$ of the form 
\[
D(A)=\{f\in\F\;|\;Gf'+gf\in\F\},\quad Af=Gf'+gf,\,f\in D(A),
\]
for some $G,g\in\mathcal{H}(\Omega)$. Choose a biholomorphic map $h\colon\D\to\Omega$ by the Riemann mapping theorem 
and define the space $\mathcal{F}_{h}(\D)\coloneqq \{f\circ h\;|\; f\in\F\}$, which becomes a Banach space when equipped with the 
norm $\|f\circ h\|_{\mathcal{F}_{h}(\D)}\coloneqq \|f\|$ for $f\in\F$. It is easy to check that the composition operator 
$C_{h}\colon\F\to\mathcal{F}_{h}(\D)$, $C_{h}(f)\coloneqq f\circ h$, is an isometric ismorphism w.r.t.~the norms, 
an isomorphism $(\F,\tau_{\operatorname{co}}) \to(\mathcal{F}_{h}(\D),\tau_{\operatorname{co}})$, and 
$C_{h}^{-1}=C_{h^{-1}}$. Then it follows as in the proof of \prettyref{prop:mixed_equicont} that 
$C_{h}\colon (\F,\gamma(\|\cdot\|,\tau_{\operatorname{co}})) 
\to(\mathcal{F}_{h}(\D),\gamma(\|\cdot\|_{\mathcal{F}_{h}(\D)},\tau_{\operatorname{co}}))$ is an isomorphism as well. 
Therefore $(\mathcal{F}_{h}(\D),\|\cdot\|_{\mathcal{F}_{h}(\D)},\tau_{\operatorname{co}})$ is a sequentially complete 
Saks space such that $\mathcal{F}_{h}(\D)\subset\mathcal{H}(\D)$ and 
the semigroup $(T(t))_{t\geq 0}$ on $\mathcal{F}_{h}(\D)$ given by 
\[
T(t)\coloneqq C_{h}\circ S(t)\circ C_{h^{-1}},\quad t\geq 0,
\]
is $\gamma(\|\cdot\|_{\mathcal{F}_{h}(\D)},\tau_{\operatorname{co}})$-strongly continuous. 
Assume that $\mathcal{H}(\overline{\D})\hookrightarrow (\mathcal{F}_{h}(\D),\|\cdot\|_{\mathcal{F}_{h}(\D)})$ 
embeds continuously. Like in \cite[Remark 3.3, p.~73]{gutierrez2023} it follows that 
the generator $(B,D(B))$ of $(T(t))_{t\geq 0}$ fulfils 
\[
D(B)=\{f\in\mathcal{F}_{h}(\D)\;|\;G_{1}f'+g_{1}f\in\mathcal{F}_{h}(\D)\},\quad Bf= G_{1}f'+g_{1}f,\,f\in D(B),
\]
where $G_{1}(z)\coloneqq\frac{1}{h'(z)}G(h(z))$ and $g_{1}(z)\coloneqq g(h(z))$ for all $z\in\D$. 
Hence we may apply \prettyref{thm:converse_gen} and obtain that there is a 
jointly continuous holomorphic co-semiflow $(m,\varphi)$ 
for $\mathcal{F}_{h}(\D)$ such that $G_{1}$ is the generator of $\varphi$, 
$m_{t}(z)=\exp(\int_{0}^{t}g_{1}(\varphi_{s}(z))\d s)$ for all $t\geq 0$, $z\in\D$, and 
$T(t)=C_{m,\varphi}(t)$ for all $t\geq 0$. Like in \cite[Remark 3.3, p.~73]{gutierrez2023} we get 
that $S(t)=C_{\mu,\psi}(t)$ for all $t\geq 0$ with the jointly continuous holomorphic semiflow $\psi$ on $\Omega$ given by 
$\psi_{t}\coloneqq h\circ \varphi_{t}\circ h^{-1}$ and its jointly continuous holomorphic 
semicocycle $\mu$ given by $\mu_{t}(z)\coloneqq (m_{t}\circ h^{-1})(z)=\exp(\int_{0}^{t}g(\psi_{s}(z))\d s)$ for all $t\geq 0$ 
and $z\in\Omega$.
\end{rem}

Our second result transfers \cite[Theorem 3.2, p.~168--169]{arendt2019} ($g=0$) and \cite[Theorem 2.11, p.~174--175]{bernard2022}
from $\|\cdot\|$-strongly continuous semigroups to $\gamma$-strongly continuous semigroups. 
We recall the following from \cite[p.~167]{arendt2019}. Let $\Omega\subset\C$ be open, $G\in\mathcal{H}(\Omega)$ and consider 
the initial value problem 
\begin{equation}\label{eq:ivp_generator}
u'(t)=G(u(t)),\quad u(0)=z,
\end{equation}
for each $z\in\Omega$. By $\varphi(\cdot,z)\colon [0,\tau(z)]\to\Omega$ we denote the maximal (w.r.t.~$\tau(z)$) 
unique solution of \eqref{eq:ivp_generator}. 
We have $\tau(z)>0$ and call $(\varphi_{t})_{0\leq t<\tau(z)}$ the \emph{local semiflow 
generated by} $G$ for $z\in\Omega$ where $\varphi_{t}(z)\coloneqq \varphi(t,z)$ for $0\leq t<\tau(z)$. 
It holds that
\[
\varphi_{t+s}(z)=\varphi_{t}(\varphi_{s}(z))
\]
where $t,s\geq 0$, $t+s<\tau(z)$ and $t<\tau(\varphi_{s}(z))$. If $\tau(z)=\infty$ for all $z\in\Omega$, then 
$\varphi\coloneqq (\varphi_{t})_{t\geq 0}$ is a semiflow with generator $G$ in the sense of 
\prettyref{defn:semiflow} and \prettyref{defn:generator_semiflow}, and also called a \emph{global} semiflow. 
Moreover, we need to recall the \emph{evaluation condition} \cite[p.~168]{arendt2019} for a Banach space $(\F,\|\cdot\|)$ 
consisting of holomorphic functions on $\Omega$:
\begin{align}
\begin{split}\label{eq:eval_cond}
&\text{If } (z_{n})_{n\in\N} \text{ is a sequence in } \Omega \text{ that converges to some } z\in\overline{\Omega}\cup\{\infty\} 
\text{ and} \\ 
&\lim_{n\to\infty}f(z_{n}) \text{ exists in } \C \text{ for all } f\in\F, \text{ then } z\in\Omega.
\end{split}\tag{E}
\end{align}
Here the one-point compactification $\C\cup\{\infty\}$ of $\C$ is used, which is only needed if $\Omega$ is an unbounded set. 
For instance, condition \eqref{eq:eval_cond} is fulfilled for 
the Hardy spaces $H^{p}$, $1\leq p<\infty$, and further examples may be found in \cite[Example 3.9, p.~173]{arendt2019}.
In the following theorem we denote by $(X,\upsilon)'$ the topological linear dual space of 
a Hausdorff locally convex space $(X,\upsilon)$. 

\begin{thm}\label{thm:converse_gen_eval}
Let $\Omega\subset\C$ be open, $(\F,\|\cdot\|,\tau_{\operatorname{co}})$ 
a sequentially complete Saks space such that $\F\subset\mathcal{H}(\Omega)$ and suppose that $\F$ 
fulfils the evaluation condition \eqref{eq:eval_cond}. 
Let $G,g\in\mathcal{H}(\Omega)$ where $G$ generates the local semiflow $(\varphi_{t})_{0\leq t<\tau(z)}$ for each $z\in\Omega$.
 
If $\mathbf{T}\coloneqq (T(t))_{t\geq 0}$ is a $\gamma$-strongly continuous semigroup on $\F$ with 
generator $(A,D(A))$ such that $Af=Gf'+gf$ for all $f\in D(A)$, 
then the semiflow $\varphi$ is global and there is a semicocycle $m$ for $\varphi$ 
such that $(m,\varphi)$ is a jointly continuous holomorphic co-semiflow 
for $\F$ such that $G$ is the generator of $\varphi$, $m_{t}(z)=\exp(\int_{0}^{t}g(\varphi_{s}(z))\d s)$ for all $t\geq 0$, 
$z\in\Omega$, and $\mathbf{T}=(C_{m,\varphi}(t))_{t\geq 0}$. Furthermore, we have
\[
D(A)=\{f\in\F\;|\;Gf'+gf\in\F\}.
\]
\end{thm}
\begin{proof}
The proof of \cite[Theorem 2.11, p.~174--175]{bernard2022} carries over to our setting. We only need to adapt the proofs 
of \cite[Lemmas 2.7, 2.8, p.~173]{bernard2022}. We note that $\delta_{z}\in (\F,\gamma)'$ and $\delta_{z}\in (\F,\|\cdot\|)'$ 
for all $z\in\Omega$ where $\delta_{z}(f)\coloneqq f(z)$ for all $f\in\F$ (see the proof of \prettyref{thm:converse_gen}). 
In particular, this means that we 
are not restricted to simply connected sets $\Omega$ in comparison to \cite{bernard2022}. 
We have already observed in \prettyref{rem:converse_gen} that $(A,D(A))$ 
is $\gamma$-densely defined, i.e.~$D(A)$ is $\gamma$-dense in $X$. 
We define the $\gamma$-dual operator $(A',D(A'))$ on $(\F,\gamma)'$ by  
\[
D(A')\coloneqq\{x'\in(\F,\gamma)'\;|\;\exists\;y'\in (\F,\gamma)'\;\forall\;f\in D(A):\;\langle Af,x'\rangle=\langle f,y'\rangle\}
\]
and $A'x'\coloneqq y'$ for $x'\in D(A')$. Using this definition, \cite[Lemma 2.7, p.~173]{bernard2022} is still valid 
with $\Gamma$ and $\Gamma'$ replaced by $A$ and $A'$, respectively. From this adapted version of 
\cite[Lemma 2.7, p.~173]{bernard2022} and $Af=\frac{\d}{\d t}T(t)f$ for all $t\geq 0$ and 
$f\in D(A)$ follows that \cite[Lemma 2.8, p.~173]{bernard2022} is also valid in our setting. 

This implies that \cite[Lemmas 2.9, 2.10, p.~173--174]{bernard2022} hold in our setting as well and thus the semiflow $\varphi$ 
is global with generator $G$ and $T(t)=C_{m,\varphi}(t)$ for all $t\geq 0$ 
with $m_{t}(z)\coloneqq\exp(\int_{0}^{t}g(\varphi_{s}(z))\d s)$ for all $t\geq 0$, $z\in\Omega$. 
By \cite[Chap.~8, \S 7, Theorem 2, p.~175]{hirsch1974} the semiflow $\varphi$ is jointly continuous and 
by the proof of \cite[Chap.~1, Theorem 9, p.~13]{hurewicz1970} holomorphic, hence the semicocycle $m$ 
by \prettyref{prop:jointly_cont_int_cocycle}, too. 
It follows that $\varphi_{(\cdot)}(z)\in\mathcal{C}^{1}[0,\infty)$ and $m_{(\cdot)}(z)\in\mathcal{C}^{1}[0,\infty)$ 
for all $z\in\Omega$. We deduce that $D(A)=\{f\in\F\;|\;Gf'+gf\in\F\}$ by \prettyref{thm:mixed_gen_max_dom_diff_everywhere}
since $\dt{\varphi}_{0}=G$ and $\dt{m}_{0}=g$.
\end{proof}

Again, \prettyref{thm:converse_gen_eval} implies \cite[Theorem 2.11, p.~174--175]{bernard2022} if 
$(T(t))_{t\geq 0}$ is $\|\cdot\|$-strongly continuous and $T(t)\in\mathcal{L}(X,\gamma)$ for all $t\geq 0$.

\section{Applications}
\label{sect:applications}

In this short section we apply our results from the preceding sections. 

\begin{prop}\label{prop:mult_sg_Cv}
Let $\Omega$ be a Hausdorff $k_{\R}$-space, $v\colon\Omega\to(0,\infty)$ continuous, 
$g\colon\Omega\to \C$ continuous such that $M\coloneqq\sup_{x\in\Omega}\re g(x)<\infty$, 
and set $m_{t}(x)\coloneqq \euler^{tg(x)}$ for $t\geq 0$ and $x\in\Omega$. 
Then the following assertions hold.
\begin{enumerate}
\item The weighted composition semigroup 
$(C_{m,\id}(t))_{t\geq 0}$ on $\mathcal{C}v(\Omega)$ w.r.t.~the co-semiflow $(m,\id)$ 
is $\gamma$-strongly continuous, $\tau_{\operatorname{co}}$-bi-continuous, locally $\tau_{\operatorname{co}}$-equicontinuous 
and locally $\gamma$-equicontinuous.  
\item If $\Omega$ is Polish or hemicompact, then $(C_{m,\id}(t))_{t\geq 0}$ is quasi-$\gamma$-equicontinuous 
and quasi-$(\|\cdot\|_{v},\tau_{\operatorname{co}})$-equitight.
\item If $M\leq 0$, then $(C_{m,\id}(t))_{t\geq 0}$ is 
$\tau_{\operatorname{co}}$-equicontinuous and $\gamma$-equicontinuous.
\item The generator $(A,D(A))$ of $(C_{m,\id}(t))_{t\geq 0}$ fulfils
\[
D(A)=\{f\in\mathcal{C}v(\Omega)\;|\;gf\in\mathcal{C}v(\Omega)\},\quad Af=gf,\, f\in D(A).
\]
\end{enumerate}
\end{prop}
\begin{proof}
We note that 
\[
|tg(x)-sg(y)|\leq t|g(x)-g(y)|+|t-s||g(y)|
\]
for all $t,s\geq 0$ and $x,y\in\Omega$, which implies that $m$ is jointly continuous and thus $(m,\id)$ as well. 
Further, $C_{m,\id}(t)f=m_{t}\cdot f$ is continuous on $\Omega$ and
\[
 \|C_{m,\id}(t)f\|_{v}
=\sup_{x\in\Omega}|\euler^{tg(x)}f(x)|v(x)
=\sup_{x\in\Omega}\euler^{t\re g(x)}|f(x)|v(x)
\leq \euler^{tM}\|f\|_{v}
\]
for all $t\in\R$ and $f\in \mathcal{C}v(\Omega)$.
Hence $(m,\id)$ is a jointly continuous co-semiflow for $\mathcal{C}v(\Omega)$ and $(C_{m,\id}(t))_{t\geq 0}$ locally bounded. 
Therefore the parts (a), (b) and (d) follow from \prettyref{ex:cont_saks}, \prettyref{thm:sg_weighted_comp_str_cont_equicont} 
and \prettyref{prop:generator_mult_seg}. 
Part (c) is a consequence of \prettyref{prop:mixed_equicont} with $I\coloneqq[0,\infty)$ because $\id_{I}(K)=K$ and
$\sup\{|x|\;|\;x\in m_{I}(K)\}\leq 1$ for all compact $K\subset\Omega$ as well as 
$\sup_{t\in I}\|C_{m,\id}\|_{\mathcal{L}(\mathcal{C}v(\Omega))}\leq 1$ if $M\leq 0$. 
\end{proof}

For a locally compact Hausdorff space $\Omega$, $v=\mathds{1}$ and $M<0$ it can also be found in \cite[p.~353]{budde2019} 
that $(C_{m,\id}(t))_{t\geq 0}$ is a $\tau_{\operatorname{co}}$-bi-continuous semigroup on $\mathcal{C}_{b}(\Omega)$ 
whose generator $(A,D(A))$ fulfils (d). 

Let $v\colon\R\to(0,\infty)$ be continuous. In the following proposition $\mathcal{C}v(\R)$ denotes the weighted space of 
continuous functions from \prettyref{ex:cont_saks} for $\K=\R$, 
and $\mathcal{C}^{1}v(\R)$ its subspace of functions $f\in\mathcal{C}^{1}(\R)$ such that 
$vf$ and $vf'$ are bounded on $\R$. 
Further, we write $\mathcal{C}_{b}^{1}(\R)\coloneqq\mathcal{C}^{1}v(\R)$ 
for $v=\mathds{1}$. 

\begin{prop}\label{prop:weighted_translation_Cb}
Let $v\colon\R\to(0,\infty)$ be continuous, $\varphi_{t}(x)\coloneqq x+t$, $t,x\in\R$, 
and $m\coloneqq (m_{t})_{t\in\R}$ be a cocycle for $\varphi\coloneqq (\varphi_{t})_{t\in\R}$ such that 
$\lim_{s\to 0}m_{s}(x)=1$ for all $x\in\R$, $m_{t}\in\mathcal{C}_{b}(\R)$ for all $t\in\R$,
$m_{(\cdot)}(x)\in\mathcal{C}^{1}(\R)$, $m_{t}(x)\neq 0$ for all $(t,x)\in\R^2$,
\begin{equation}\label{eq:loc_bound_sg_Cv_1}
K(\varphi_{t})\coloneqq\sup_{x\in\R}\frac{v(x)}{v(t+x)}<\infty
\end{equation}
for all $t\in\R$ and there exists $t_{0}>0$ such that $\sup_{t\in [0,t_{0}]}K(\varphi_{t})<\infty$.
Then the following assertions hold.
\begin{enumerate}
\item The weighted composition semigroup 
$(C_{m,\varphi}(t))_{t\geq 0}$ on $\mathcal{C}v(\R)$ w.r.t.~the co-flow $(m,\varphi)$ 
is $\gamma$-strongly continuous, $\tau_{\operatorname{co}}$-bi-continuous, locally $\tau_{\operatorname{co}}$-equiconti\-nuous, 
quasi-$\gamma$-equicontinuous and quasi-$(\|\cdot\|_{v},\tau_{\operatorname{co}})$-equitight.
\item The generator $(A,D(A))$ of $(C_{m,\varphi}(t))_{t\geq 0}$ fulfils
\[
D(A)=\{f\in\mathcal{C}^{1}(\R)\cap\mathcal{C}v(\R)\;|\;f'+\dt{m}_{0}f\in\mathcal{C}v(\R)\}
\]
and $Af=f'+\dt{m}_{0}f$ for all $f\in D(A)$. If in addition $\dt{m}_{0}\in\mathcal{C}_{b}(\R)$, 
then $D(A)=\mathcal{C}^{1}v(\R)$.
\end{enumerate}
\end{prop}
\begin{proof}
(a) The flow $\varphi$ is clearly $C_{0}$. The assumption $\lim_{s\to 0}m_{s}(x)=1$ for all $x\in\R$ means that its cocycle 
$m$ is also $C_{0}$ and so the co-flow $(m,\varphi)$ is jointly continuous by \prettyref{prop:jointly_cont_flow_cocycle}. 
Further, $C_{m,\varphi}(t)f=m_{t}\cdot(f\circ \varphi_{t})$ is continuous on $\R$ and
\begin{align*}
 \|C_{m,\varphi}(t)f\|_{v}
&=\sup_{x\in\R}|m_{t}(x)f(x+t)|v(t+x)\frac{v(x)}{v(t+x)}
 \leq K(\varphi_{t})\|m_{t}\|_{\infty}\|f\|_{v}
\end{align*}
for all $t\in\R$ and $f\in \mathcal{C}v(\R)$ by \eqref{eq:loc_bound_sg_Cv_1}. 
Hence $(m,\varphi)$ is a jointly continuous co-flow for $\mathcal{C}v(\R)$ and $(C_{m,\varphi}(t))_{t\geq 0}$ 
locally bounded by \prettyref{prop:loc_bounded_sg} since $\sup_{t\in [0,t_{0}]}K(\varphi_{t})<\infty$ and 
$\sup_{t\in [0,t_{1}]}\|m_{t}\|_{\infty}<\infty$ for some $t_{1}> 0$ by \prettyref{prop:semicocycle_bounded}. 
We conclude that statement (a) is valid by \prettyref{ex:cont_saks} and
\prettyref{thm:sg_weighted_comp_str_cont_equicont}.  

(b) The first part of (b) follows from \prettyref{thm:mixed_gen_max_dom_real} (a) since 
$\varphi_{0}(x)=x$ and $\dt{\varphi}_{0}(x)=1$ for all $x\in\R$. 
If additionally $\dt{m}_{0}\in\mathcal{C}_{b}(\R)$, then we have $\dt{m}_{0}f\in\mathcal{C}v(\R)$ 
and thus $D(A)=\mathcal{C}^{1}v(\R)$.
\end{proof}

\prettyref{prop:weighted_translation_Cb} (a) generalises 
\cite[Example 4.2 (a), p.~19]{kruse_schwenninger2022} where $m=v=\mathds{1}$. 
\prettyref{prop:weighted_translation_Cb} (b) generalises \cite[Proposition 1.8, p.~6]{budde2019a} 
where it is shown for $m=v=\mathds{1}$ that $D(A)=\mathcal{C}^{1}_{b}(\R)$ and $Af=f'$ 
for all $f\in D(A)$. \prettyref{prop:weighted_translation_Cb} (b) is also a consequence of 
\prettyref{thm:mixed_gen_max_dom_real} (b) with $G\coloneqq \mathds{1}$ since 
$\dt{\varphi}_{t}(x)=1$ for all $(t,x)\in\R^2$.

\begin{prop}\label{prop:weighted_flow_Cb}
Let $v\colon\R\to(0,\infty)$ be continuous, 
$\varphi_{t}(x)\coloneqq \left(x^{\frac{1}{3}}+\frac{t}{3}\right)^3$, $t,x\in\R$, and $m\coloneqq (m_{t})_{t\in\R}$ be 
a cocycle for $\varphi\coloneqq (\varphi_{t})_{t\in\R}$ such that
$\lim_{s\to 0}m_{s}(x)=1$ for all $x\in\R$, $m_{t}\in\mathcal{C}_{b}(\R)$ for all $t\in\R$,
$m_{(\cdot)}(x)\in\mathcal{C}^{1}(\R)$, $m_{t}(x)\neq 0$ for all $(t,x)\in\R^2$,
\[
K(\varphi_{t})\coloneqq\sup_{x\in\R}\frac{v(x)}{v(\varphi_{t}(x))}<\infty
\]
for all $t\in\R$ and there exists $t_{0}>0$ such that $\sup_{t\in [0,t_{0}]}K(\varphi_{t})<\infty$. 
Then the following assertions hold.
\begin{enumerate}
\item The weighted composition semigroup 
$(C_{m,\varphi}(t))_{t\geq 0}$ on $\mathcal{C}v(\R)$ w.r.t.~the co-flow $(m,\varphi)$ 
is $\gamma$-strongly continuous, $\tau_{\operatorname{co}}$-bi-continuous, locally $\tau_{\operatorname{co}}$-equiconti\-nuous, 
quasi-$\gamma$-equicontinuous and quasi-$(\|\cdot\|_{v},\tau_{\operatorname{co}})$-equitight.
\item The generator $(A,D(A))$ of $(C_{m,\varphi}(t))_{t\geq 0}$ fulfils
\[
D(A)=\{f\in\mathcal{C}^{1}(\R\setminus\{0\})\cap\mathcal{C}v(\R)\;|\;[x\mapsto x^{\frac{2}{3}}f'(x)+\dt{m}_{0}(x)f(x)]\in
\mathcal{C}v(\R)\}
\]
and $Af(x)=x^{\frac{2}{3}}f'(x)+\dt{m}_{0}(x)f(x)$, $x\neq 0$, for all $f\in D(A)$. 
If in addition $\dt{m}_{0}\in\mathcal{C}_{b}(\R)$, 
then $D(A)=\{f\in\mathcal{C}^{1}(\R\setminus\{0\})\cap\mathcal{C}v(\R)\;|\;
[x\mapsto x^{\frac{2}{3}}f'(x)]\in\mathcal{C}v(\R)\}$.
\end{enumerate}
\end{prop}
\begin{proof}
(a) The flow $\varphi$ is clearly $C_{0}$. The assumption $\lim_{s\to 0}m_{s}(x)=1$ for all $x\in\R$ means that its cocycle 
$m$ is also $C_{0}$ and so the co-flow $(m,\varphi)$ is jointly continuous by \prettyref{prop:jointly_cont_flow_cocycle}. 
The rest follows like in the proof of \prettyref{prop:weighted_translation_Cb} (a).

(b) Setting $G\colon\R\to\R$, $G(x)\coloneqq x^\frac{2}{3}$, we note that 
\[
\dt{\varphi}_{t}(x)=\left(x^{\frac{1}{3}}+\frac{t}{3}\right)^2=(G\circ\varphi_{t})(x)
\]
for all $(t,x)\in\R^2$. Thus $N_{G}=\{0\}$ and $t_{0}=\inf\varnothing=\infty$. 
We deduce that the first part of (b) follows from \prettyref{thm:mixed_gen_max_dom_real} (b).
If additionally $\dt{m}_{0}\in\mathcal{C}_{b}(\R)$, then we have $\dt{m}_{0}f\in\mathcal{C}v(\R)$ 
and thus $D(A)=\{f\in\mathcal{C}^{1}(\R\setminus\{0\})\cap\mathcal{C}v(\R)\;|\;[x\mapsto x^{\frac{2}{3}}f'(x)]
\in\mathcal{C}v(\R)\}$.
\end{proof}

For $m=v=\mathds{1}$ the statement of \prettyref{prop:weighted_flow_Cb} (a) is known due to 
e.g.~\cite[Theorems 2.1, 2.2, 2.3, p.~5]{dorroh1993}. 
\prettyref{prop:weighted_flow_Cb} (b) generalises \cite[Example 4.2, p.~124--125]{dorroh1996} where $m=v=\mathds{1}$.

\begin{thm}\label{thm:holomorphic_weighted_comp_norm_strong}
Let $\Omega\subset\C$ be open and connected, $(\F,\|\cdot\|,\tau_{\operatorname{co}})$ a sequentially complete Saks 
space such that $\F\subset\mathcal{H}(\Omega)$ and $\{\mathds{1},\id\}\subset\F$, 
and $(C_{m,\varphi}(t))_{t\geq 0}$ the weighted composition semigroup on $\F$ w.r.t.~a holomorphic co-semiflow $(m,\varphi)$. 
Then the following assertions hold.
\begin{enumerate}
\item $(C_{m,\varphi}(t))_{t\geq 0}$ is $\gamma$-strongly continuous and locally $\gamma$-equicontinuous if and only if 
$(m,\varphi)$ is a $C_{0}$-co-semiflow and $(C_{m,\varphi}(t))_{t\geq 0}$ locally bounded. 
\item Suppose that $(\F,\|\cdot\|)$ is reflexive. Then $(C_{m,\varphi}(t))_{t\geq 0}$ is 
$\|\cdot\|$-strongly continuous if and only if $(m,\varphi)$ is a $C_{0}$-co-semiflow and $(C_{m,\varphi}(t))_{t\geq 0}$ 
locally bounded.
\item If $(C_{m,\varphi}(t))_{t\geq 0}$ is $\|\cdot\|$-strongly continuous 
with generator $(A_{\|\cdot\|},D(A_{\|\cdot\|}))$, then we have $\dt{\varphi}_{0},\dt{m}_{0}\in\mathcal{H}(\Omega)$ and 
\[
D(A_{\|\cdot\|})=\{f\in\F\;|\;\dt{\varphi}_{0}f'+\dt{m}_{0}f\in\F\}
\]
and $A_{\|\cdot\|}f=\dt{\varphi}_{0}f'+\dt{m}_{0}f$ for all $f\in D(A_{\|\cdot\|})$.
\end{enumerate}
\end{thm}
\begin{proof}
First, we note that the open set $\Omega\subset\C$ is a locally compact, $\sigma$-compact space 
w.r.t.~the relative topology induced by the metric space $\C$. In particular, $\Omega$ is a Hausdorff $k_{\R}$-space, and the 
co-semiflow $(m,\varphi)$ is jointly continuous if and only if it is $C_{0}$ by \prettyref{prop:jointly_cont_semiflow} 
and \prettyref{prop:jointly_cont_cocycle}. Second, if $\varphi$ is $C_{0}$, so jointly continuous, 
then $\varphi_{(\cdot)}(z)\in\mathcal{C}^{1}[0,\infty)$ for all $z\in\Omega$ and $\dt{\varphi}_{0}\in\mathcal{H}(\Omega)$ 
by \prettyref{thm:generator_hol_semiflow}.
Third, $m_{(\cdot)}(z)\in\mathcal{C}^{1}[0,\infty)$, $\dt{m}_{0}\in\mathcal{H}(\Omega)$ and $m_{t}(z)\neq 0$ for all $t\geq 0$ 
and $z\in\Omega$ by \prettyref{prop:jointly_cont_int_cocycle} (a) and \prettyref{prop:cocycle_int_hol} 
since $\Omega\subset\C$ is connected.

(a) Due to \cite[Proposition 3.6 (ii), p.~1137]{federico2020} and \cite[I.1.10 Proposition, p.~10]{cooper1978} 
a $\gamma$-strongly continuous, locally $\gamma$-equicontinuous semigroup of linear operators is 
locally bounded. Thus implication $\Rightarrow$ follows from \prettyref{prop:strongly_mixed_cont} (b) 
and \prettyref{rem:initial_like} (b). 
The converse implication $\Leftarrow$ follows from \prettyref{thm:sg_weighted_comp_str_cont_equicont} (a). 

(b) Let $(C_{m,\varphi}(t))_{t\geq 0}$ be $\|\cdot\|$-strongly continuous. 
Then $(C_{m,\varphi}(t))_{t\geq 0}$ is locally bounded by \cite[Chap.~I, 5.5 Proposition, p.~39]{engel_nagel2000} and 
$\gamma$-strongly continuous by \prettyref{prop:mixed_equicont} with $I\coloneqq\{t\}$ for all $t\geq 0$ and since $\gamma$ is coarser 
than the $\|\cdot\|$-topology. 
Due to the assumption $\{\mathds{1},\id\}\subset\F$ the topology of $\Omega$ is initial-like w.r.t.~$(\varphi,\F)$ 
by \prettyref{rem:initial_like} (b), and thus the co-semiflow $(m,\varphi)$ jointly continuous 
by \prettyref{prop:strongly_mixed_cont} (b). 

Let us turn to the converse. Suppose that $(m,\varphi)$ is a $C_{0}$-co-semiflow and 
$(C_{m,\varphi}(t))_{t\geq 0}$ locally bounded. Then $(m,\varphi)$ is jointly continuous and 
$(C_{m,\varphi}(t))_{t\geq 0}$ $\gamma$-strongly continuous by \prettyref{prop:strongly_mixed_cont} (a). 
We deduce from \prettyref{prop:norm_generator} (b) that $(C_{m,\varphi}(t))_{t\geq 0}$ is $\|\cdot\|$-strongly continuous 
since $(\F,\|\cdot\|)$ is reflexive.

(c) By the the first part of the proof of (b) we get that $(C_{m,\varphi}(t))_{t\geq 0}$ is locally bounded 
and $(m,\varphi)$ a $C_{0}$-co-semiflow. Applying \prettyref{thm:mixed_gen_max_dom_diff_everywhere} and 
\prettyref{prop:norm_generator} (a), this finishes the proof of part (c). 
\end{proof}

\prettyref{thm:holomorphic_weighted_comp_norm_strong} (b) implies \cite[Lemma 3.1, p.~474]{koenig1990} 
and \cite[Theorem 1, p.~362]{siskakis1986} for the reflexive Hardy spaces $H^{p}$, $1<p<\infty$, 
by \cite[Definition 1 (i), p.~469]{koenig1990}, \prettyref{ex:hardy_bergman_dirichlet} (a), 
\prettyref{ex:semicoboundary} and \prettyref{thm:loc_bound_sg_hardy_bergman_dirichlet} (a). 
It also answers the questions in \cite[Example 7.4, p.~247--248]{siskakis1998} because it implies 
the $\|\cdot\|$-strong continuity of the weighted composition semigroup induced by the semicoboundary $(\varphi',\varphi)$ 
for a jointly continuous holomorphic semiflow $\varphi$ on $\D$ (see \prettyref{ex:semicoboundary_gen}) on 
reflexive spaces $(\mathcal{F}(\D),\|\cdot\|)$ such that $\mathcal{F}(\D)\subset\mathcal{H}(\D)$ such as the Hardy spaces 
$H^{p}$ for $1<p<\infty$, the weighted Bergman spaces $A_{\alpha}^{p}$ for $\alpha>-1$ and $1<p<\infty$, 
and the Dirichlet space $\mathcal{D}$ due to \prettyref{ex:hardy_bergman_dirichlet} {\color{green}and}
\prettyref{thm:loc_bound_sg_hardy_bergman_dirichlet} (and the computations thereafter). 
Further, \prettyref{thm:holomorphic_weighted_comp_norm_strong} (b) in combination 
with \prettyref{thm:loc_bound_sg_hardy_bergman_dirichlet} (a) and (b) gives back 
\cite[Lemmas 2.13, 2.14, p.~828]{perlich2019} and \cite[Corollaries 3, 4, p.~8--9]{wu2021} 
for $1<p<\infty$ by a different proof.\footnote{We note that the necessary condition that $m$ has to be jointly continuous 
is missing in \cite[Corollaries 3, 4, p.~8--9]{wu2021} (see \cite[p.~2]{wu2021}) even though the cited references 
therein concerning semicocycles like \cite{koenig1990,jafari1997} have this incorporated in their definition of a semicocycle.} 
It also implies in combination with \prettyref{thm:loc_bound_sg_hardy_bergman_dirichlet} (c) 
\cite[Theorem 1, p.~167]{siskakis1996} where $\mathcal{F}(\D)=\mathcal{D}$ and $m=\mathds{1}$. 

If $(\F,\|\cdot\|,\tau_{\operatorname{co}})$ is a sequentially complete Saks such that 
$\F\subset\mathcal{H}(\Omega)$ and $\{\mathds{1},\id\}\subset\F$, 
then \prettyref{thm:holomorphic_weighted_comp_norm_strong} (c) implies \cite[Theorem 2, p.~72]{blasco2013} 
where $\Omega=\D$ and $m=\mathds{1}$ but it is only assumed that $(\F,\|\cdot\|)$ is complete and $\tau_{\operatorname{co}}$ 
coarser than the $\|\cdot\|$-topology (see \cite[p.~67]{blasco2013}).
Moreover, \prettyref{thm:holomorphic_weighted_comp_norm_strong} (c) generalises \cite[Theorem 2, p.~364]{siskakis1986} 
(where $m=m^{\omega}$ is a semicoboundary for $\omega\in\mathcal{H}(\D)$) and 
the first part of \cite[Theorem 2 (b), p.~471]{koenig1990} 
by \prettyref{ex:hardy_bergman_dirichlet} (a) and \prettyref{ex:semicoboundary} 
where $\mathcal{F}(\D)=H^{p}$ is the Hardy space for $1\leq p<\infty$.  
\prettyref{thm:holomorphic_weighted_comp_norm_strong} (c) also yields \cite[Theorem 1 (ii), p.~400--401]{siskakis1987}
by \prettyref{ex:hardy_bergman_dirichlet} (b) where $\mathcal{F}(\D)=A_{\alpha}^{p}$ is the Bergman space for $\alpha>-1$ 
and $1\leq p<\infty$ and $m=\mathds{1}$.
\prettyref{thm:holomorphic_weighted_comp_norm_strong} (b) and (c) 
combined with \prettyref{thm:loc_bound_sg_hardy_bergman_dirichlet} (c) 
also imply the $\|\cdot\|$-strong continuity of the 
weighted composition semigroup on $\mathcal{D}$ w.r.t.~the holomorphic $C_{0}$-co-semiflow $(m,\varphi)$ 
in \cite[Corollary 2, p.~170]{siskakis1996} and the form of its generator where $\operatorname{Fix}(\varphi)=\{0\}$ and 
$m=m^{\omega}$ with $\omega(z)\coloneqq z^{p}$ for all $z\in\D$ for some $p\in\N$. 
If $(\F,\|\cdot\|)$ is a Banach space such that $\tau_{\operatorname{co}}$ is coarser than 
the $\|\cdot\|$-topology and $\F\subset\mathcal{H}(\Omega)$, 
then the assumptions in \cite[Theorem 2.1 (ii), p.~68]{gutierrez2023} that $\Omega=\D$ and the continuity of 
$\mathcal{H}(\overline{\D})\hookrightarrow (\F,\|\cdot\|)$ are stronger than 
in \prettyref{thm:holomorphic_weighted_comp_norm_strong} (c) whereas vice versa the latter theorem has the stronger 
assumption that $(\F,\|\cdot\|,\tau_{\operatorname{co}})$ is a sequentially complete Saks space.

We may also apply \prettyref{thm:holomorphic_weighted_comp_norm_strong} (a) to weighted composition semigroups on 
the Hardy space $H^{1}$ by \prettyref{thm:loc_bound_sg_hardy_bergman_dirichlet} (a), the Bergman space $A_{\alpha}^{1}$ 
for $\alpha>-1$ by \prettyref{thm:loc_bound_sg_hardy_bergman_dirichlet} (b), the Bloch type space $\mathcal{B}_{\alpha}$ 
for $\alpha>0$ by \prettyref{thm:loc_bound_sg_bloch}, the weighted spaces $\mathcal{H}v(\Omega)$ and $\mathcal{C}v(\Omega)$ 
of holomorphic resp.~continuous functions on $\Omega$ by \prettyref{thm:loc_bound_sg_Hv} resp.~\prettyref{thm:loc_bound_sg_Cv} 
if $\{\mathds{1},\id\}\subset\mathcal{H}v(\Omega)$ resp.~$\{\mathds{1},\id\}\subset\mathcal{C}v(\Omega)$. For instance, we have 
the following result for the Hardy space $H^{\infty}$. 

\begin{cor}
Let $(m,\varphi)$ be a holomorphic co-semiflow on $\D$. 
The weighted composition semigroup $(C_{m,\varphi}(t))_{t\geq 0}$ is $\gamma$-strongly continuous and 
locally $\gamma$-equicontinu\-ous on $H^{\infty}$ if and only 
if $(m,\varphi)$ is a $C_{0}$-co-semiflow and $\limsup_{t\to 0\rlim}\|m_{t}\|_{\infty}<\infty$. 
\end{cor}
\begin{proof}
This statement follows from \prettyref{ex:weighted_holom_saks}, \prettyref{thm:loc_bound_sg_Hv} with $v\coloneqq\mathds{1}$ 
and \prettyref{thm:holomorphic_weighted_comp_norm_strong} (a).
\end{proof}

\begin{thm}\label{thm:holomorphic_weighted_comp}
Let $\Omega\subset\C$ be open, $(\F,\|\cdot\|,\tau_{\operatorname{co}})$ a sequentially complete Saks space such that 
$\F\subset\mathcal{H}(\Omega)$, and $(C_{m,\varphi}(t))_{t\geq 0}$ a locally bounded weighted composition semigroup 
on $\F$ w.r.t.~a holomorphic $C_{0}$-co-semiflow $(m,\varphi)$.
Then the following assertions hold.
\begin{enumerate}
\item $(C_{m,\varphi}(t))_{t\geq 0}$ is $\gamma$-strongly continuous, $\tau_{\operatorname{co}}$-bi-continuous, 
locally $\tau_{\operatorname{co}}$-equicon\-tinuous and locally $\gamma$-equicontinuous. 
\item If $(\F,\|\cdot\|,\tau_{\operatorname{co}})$ is a C-sequential Saks space, 
then $(C_{m,\varphi}(t))_{t\geq 0}$ is quasi-$\gamma$-equicontinuous. Furthermore, $(C_{m,\varphi}(t))_{t\geq 0}$ 
is quasi-$(\|\cdot\|,\tau_{\operatorname{co}})$-equitight if additionally $\gamma=\gamma_{s}$.
\item If $m_{(\cdot)}(z)\in\mathcal{C}^{1}[0,\infty)$ for all $z\in\Omega$, 
then $\dt{\varphi}_{0}\in\mathcal{H}(\Omega)$ and the generator $(A,D(A))$ of $(C_{m,\varphi}(t))_{t\geq 0}$ fulfils
\[
D(A)=\{f\in\F\;|\;\dt{\varphi}_{0}f'+\dt{m}_{0}f\in\F\}
\]
and $Af=\dt{\varphi}_{0}f'+\dt{m}_{0}f$ for all $f\in D(A)$.
\item If $m_{(\cdot)}(z)\in\mathcal{C}^{1}[0,\infty)$ for all $z\in\Omega$, 
then $\dt{\varphi}_{0}\in\mathcal{H}(\Omega)$ and 
\[
\bigl[(m,\varphi),\F\bigr]=\overline{\{f\in\F\;|\;\dt{\varphi}_{0}f'+\dt{m}_{0}f\in\F\}}^{\|\cdot\|}
\]
where $\bigl[(m,\varphi),\F\bigr]$ is the space of $\|\cdot\|$-strong continuity of $(C_{m,\varphi}(t))_{t\geq 0}$.
\end{enumerate}
\end{thm}
\begin{proof}
By the first part of the proof of \prettyref{thm:holomorphic_weighted_comp_norm_strong} we know that 
the $C_{0}$-co-semiflow $(m,\varphi)$ is jointly continuous, $\varphi_{(\cdot)}(z)\in\mathcal{C}^{1}[0,\infty)$ 
for all $z\in\Omega$ and $\dt{\varphi}_{0}\in\mathcal{H}(\Omega)$.

The parts (a), (b) and (c) follow directly from \prettyref{thm:sg_weighted_comp_str_cont_equicont} 
and \prettyref{thm:mixed_gen_max_dom_diff_everywhere}.
Part (d) is a consequence of \prettyref{prop:norm_generator} (c) and \prettyref{thm:mixed_gen_max_dom_diff_everywhere}.
\end{proof}

Comparing \prettyref{thm:holomorphic_weighted_comp} (d) (and \eqref{eq:space_deriv_semiflow_integral}) 
with \cite[Theorem 10, p.~9]{arevalo2017} resp.~\cite[Theorem 1, p.~71]{blasco2013}
where $\Omega=\D$ and $m_{t}=\varphi_{t}'$ for all $t\geq 0$ resp.~$m=\mathds{1}$ we see that the former theorem is more general 
w.r.t.~the semicocycles and does not need the assumption $\mathds{1}\in\F$, 
and less general w.r.t.~the tuples $(\F,\|\cdot\|,\tau_{\operatorname{co}})$. 
In the latter theorem it is only assumed that $(\F,\|\cdot\|)$ is complete and $\tau_{\operatorname{co}}$ coarser than 
the $\|\cdot\|$-topology (see \cite[p.~67]{blasco2013}). 
However, both theorems have the assumption on joint continuity of $\varphi$ (see \cite[(3'), p.~68]{blasco2013}) and 
local boundedness of $(C_{m,\varphi}(t))_{t\geq 0}$ (see \prettyref{prop:loc_bounded_sg}). 

For $H^{\infty}\subset\mathcal{F}(\D)\subset\mathcal{B}_{1}$, $(\mathcal{F}(\D),\|\cdot\|)$ Banach, it holds 
$\bigl[(\mathds{1},\varphi),\mathcal{F}(\D)\bigr]\subsetneq\mathcal{F}(\D)$ by \cite[Theorem 1.1, p.~844]{anderson2017} 
for any non-trivial jointly continuous holomorphic semicocycle $(\mathds{1},\varphi)$ for $\F$ such that 
$C_{\mathds{1},\varphi}(t)\in\mathcal{L}(\mathcal{F}(\D))$ for all $t\geq 0$. 
The weighted version is given in \cite[Theorem 4.1, p.~74]{gutierrez2023}. 

In the case $\mathcal{F}(\D)=H^{\infty}$ it holds $\mathcal{A}\subset\bigl[(\mathds{1},\varphi),H^{\infty}\bigr]$ 
by \cite[Corollary 1.4, p.~844]{anderson2017} for any jointly continuous holomorphic $\varphi$, and 
$\mathcal{A}\neq\bigl[(\mathds{1},\varphi),H^{\infty}\bigr]$ for some $\varphi$ 
by \cite[Proposition 4.3, p.~852]{anderson2017} where $\mathcal{A}$ is the \emph{disc-algebra} of holomorphic functions 
on $\D$ that extend continuously to $\overline{\D}$. If $\varphi$ consists of rotations or dilations, 
then $\mathcal{A}=\bigl[(\mathds{1},\varphi),H^{\infty}\bigr]$ by \cite[Proposition 4.1, p.~850]{anderson2017}. 

Further, we have $\bigl[(\mathds{1},\varphi),\mathcal{B}_{\alpha}\bigr]\subsetneq\mathcal{B}_{\alpha}$ for $\alpha>0$
and any non-trivial jointly continuous holomorphic semicocycle $(\mathds{1},\varphi)$ for $\mathcal{B}_{\alpha}$ such that 
$C_{\mathds{1},\varphi}(t)\in\mathcal{L}(\mathcal{B}_{\alpha})$ for all $t\geq 0$ by \cite[Theorem 3, p.~73]{blasco2013}. 
In the case $\alpha=1$ it holds $\mathcal{B}_{0}\subset\bigl[(\mathds{1},\varphi),\mathcal{B}_{1}\bigr]$ for any $\varphi$ 
by \cite[p.~73]{blasco2013}, and $\mathcal{B}_{0}=\bigl[(\mathds{1},\varphi),\mathcal{B}_{1}\bigr]$ if and only if 
the resolvent operator $R(\lambda,A)\in\mathcal{L}(\mathcal{B}_{1})$ is weakly compact on $\mathcal{B}_{0}$ 
by \cite[Corollary 1, p.~76]{blasco2013} where $\mathcal{B}_{0}$ is the \emph{little Bloch space} of $\mathcal{B}_{1}$, 
i.e.~the space consisting of all $f\in\mathcal{B}_{1}$ such that $\lim_{|z|\to 1\llim}(1-|z|^{2})|f'(z)|=0$. 
The assertion $\mathcal{B}_{0}=\bigl[(\mathds{1},\varphi),\mathcal{B}_{1}\bigr]$ is also equivalent to 
$\varphi$ being elliptic and its generator $G$ fulfilling the logarithmic vanishing Bloch condition 
by \cite[Theorem 1.1, p.~4]{chalmoukis2022}.

\bibliography{biblio_composition_sg}
\bibliographystyle{plainnat}
\end{document}